\documentclass[10pt,reqno]{amsart}
\usepackage{graphicx,enumerate,amssymb,bbold,tikz-cd}
\usepackage[most]{tcolorbox}
\usepackage[notrig]{physics}
\usepackage[font=footnotesize,labelfont=small]{subcaption}
\usepackage[ruled,vlined]{algorithm2e}
\usepackage[left=3.0cm,right=3.0cm,top=2.5cm,bottom=2.5cm,includeheadfoot]{geometry} 
\usepackage{tikz}
\usetikzlibrary{shapes.geometric, arrows}

\usepackage{multirow}

\usepackage[colorlinks=true, pdfstartview=FitV, linkcolor=blue, 
            citecolor=blue, urlcolor=blue]{hyperref}
\usepackage{comment}

\numberwithin{equation}{section}
\numberwithin{figure}{section}
\theoremstyle{plain}
\newtheorem{theorem}{Theorem}[section]
\newtheorem{lemma}[theorem]{Lemma}
\newtheorem{proposition}[theorem]{Proposition}
\newtheorem{corollary}[theorem]{Corollary}

\theoremstyle{definition}
\newtheorem{definition}[theorem]{Definition}

\newtheorem{remark}[theorem]{Remark}

\newcommand{\bitem}{\begin{itemize}}
\newcommand{\eitem}{\end{itemize}}
\newcommand{\mc}[1]{\mathcal{#1}}
\newcommand{\mb}[1]{\mathbb{#1}}

\newcommand{\N}{\mathbb{N}}
\newcommand{\R}{\mathbb{R}}
\newcommand{\C}{\mathbb{C}}

\newcommand{\TT}{\mathbb{T}}

\newcommand{\bpm}{\begin{pmatrix}}
\newcommand{\epm}{\end{pmatrix}}
\newcommand{\bvm}{\begin{vmatrix}}
\newcommand{\evm}{\end{vmatrix}}
\newcommand{\bsm}{\left(\begin{smallmatrix}}
\newcommand{\esm}{\end{smallmatrix}\right)}
\newcommand{\T}{\top}

\newcommand{\ol}[1]{\overline{#1}}
\newcommand{\wh}[1]{\widehat{#1}}
\newcommand{\wt}[1]{\widetilde{#1}}
\newcommand{\la}{\langle}
\newcommand{\ra}{\rangle}

\newcommand{\mrm}[1]{\mathrm{#1}}

\newcommand{\veps}{\varepsilon}

\newcommand{\embedded}{\hookrightarrow}
\newcommand{\w}{\omega}

\newcommand{\gdw}{\Leftrightarrow}
\newcommand{\vphi}{\varphi}

\newcommand{\Rep}{\mathfrak{R}}

\newcommand{\eins}{\mathbb{1}}

\newcommand{\LG}[1]{\mathrm{#1}}

\newcommand{\expm}{\exp_{\mathrm{m}}}
\newcommand{\logm}{\log_{\mathrm{m}}}

\DeclareMathSymbol{\mydiv}{\mathbin}{symbols}{"04}

\DeclareMathOperator{\Diag}{Diag}
\DeclareMathOperator{\diag}{diag}

\DeclareMathOperator{\rint}{rint}

\DeclareMathOperator{\vvec}{vec}

\DeclareMathOperator{\ggrad}{grad}

\DeclareMathOperator{\Exp}{Exp}

\newcommand{\sst}[1]{{\scriptscriptstyle #1}}

\makeatletter
\def\widebreve{\mathpalette\wide@breve}
\def\wide@breve#1#2{\sbox\z@{$#1#2$}%
     \mathop{\vbox{\m@th\ialign{##\crcr
\kern0.08em\brevefill#1{0.8\wd\z@}\crcr\noalign{\nointerlineskip}%
                    $\hss#1#2\hss$\crcr}}}\limits}
\def\brevefill#1#2{$\m@th\sbox\tw@{$#1($}%
  \hss\resizebox{#2}{\wd\tw@}{\rotatebox[origin=c]{90}{\upshape(}}\hss$}
\makeatletter

\title[Quantum State Assignment Flows]{Quantum State Assignment Flows}
\author[J.~Schwarz, J.~Cassel, B.~Boll, M.~G\"arttner, P.~Albers, C.~Schn\"{o}rr]{Jonathan Schwarz, Jonas Cassel, Bastian Boll, Martin G\"arttner, Peter Albers, Christoph Schn\"{o}rr}
\address[Jonathan Schwarz, Jonas Cassel, Bastian Boll, Christoph Schn\"{o}rr]{Institute for Mathematics, Image and Pattern Analysis Group, Heidelberg University, Germany} 
\email{jonathan.schwarz@iwr.uni-heidelberg.de}
\urladdr{\url{http://ipa.math.uni-heidelberg.de}}
\address[Peter Albers]{Institute for Mathematics, Reserach Station Geometry \& Dynamics, Heidelberg University, Germany} 
\email{palbers@mathi.uni-heidelberg.de}
\urladdr{\url{https://www.mathi.uni-heidelberg.de/~geodyn/members/Albers.html}}
\address[Martin G\"{a}rttner]{Physikalisches Institut and Kirchhoff Institute for Physics, Heidelberg University, Germany}
\email{martin.gaerttner@kip.uni-heidelberg.de}
\urladdr{\url{https://mbqd.de}}
\date{\today} 
\thanks{This work is funded by the Deutsche Forschungsgemeinschaft (DFG), grant SCHN 457/17-1, within the priority programme SPP 2298: Theoretical Foundations of Deep Learning. This work is funded by the Deutsche Forschungsgemeinschaft (DFG) under Germany's Excellence Strategy EXC-2181/1 - 390900948 (the Heidelberg STRUCTURES Excellence Cluster).}


\makeatletter
\@namedef{subjclassname@2020}{\textup{2020} Mathematics Subject Classification}
\makeatother


\makeatletter
\newsavebox{\@brx}
\newcommand{\llangle}[1][]{\savebox{\@brx}{\(\m@th{#1\langle}\)}%
  \mathopen{\copy\@brx\kern-0.5\wd\@brx\usebox{\@brx}}}
\newcommand{\rrangle}[1][]{\savebox{\@brx}{\(\m@th{#1\rangle}\)}%
  \mathclose{\copy\@brx\kern-0.5\wd\@brx\usebox{\@brx}}}
\makeatother

\begin{document}
\maketitle

\begin{abstract}
  This paper introduces assignment flows for density matrices as state spaces for representing and analyzing data associated with vertices of an underlying weighted graph. Determining an assignment flow by geometric integration of the defining dynamical system causes an interaction of the non-commuting states across the graph, and the assignment of a pure (rank-one) state to each vertex after convergence. Adopting the Riemannian Bogoliubov-Kubo-Mori metric from information geometry leads to closed-form local expressions which can be computed efficiently and implemented in a fine-grained parallel manner. 
  
  Restriction to the submanifold of commuting density matrices recovers the assignment flows for categorial probability distributions, which merely assign labels from a finite set to each data point. As shown for these flows in our prior work, the novel class of quantum state assignment flows can also be characterized as Riemannian gradient flows with respect to a non-local non-convex potential, after proper reparametrization and under mild conditions on the underlying weight function. This weight function generates the parameters of the layers of a neural network, corresponding to and generated by each step of the geometric integration scheme.
  
Numerical results indicates and illustrate the potential of the novel approach for data representation and analysis, including the representation of correlations of data across the graph by entanglement and tensorization.
\end{abstract}

\vspace{0.5cm}
\tableofcontents

\section{Introduction}\label{sec:Introduction}

\subsection{Overview and Motivation}

A basic task of data analysis is categorization of observed data. We consider the following scenario: On a given undirected, weighted graph $\mc{G}=(\mc{V},\mc{E},w)$, data 
$D_{i}\in\mc{X}$
are observed as points in a metric space 
$(\mc{X},d_{\mc{X}})$ at each vertex $i\in\mc{V}$. Categorization means to determine an assignment 
\begin{equation}
D_{i}\;\to\; j\in\{1,\dotsc,c\}=:[c]
\end{equation}
of a \textit{class label} $j$ out of a \textit{finite} set of labels to each data point $D_{i}$. Depending on the application, labels carry a specific meaning, e.g.~type of tissue in medical image data, object type in computer vision or land use in remote sensing data. The decision at any vertex typically depends on decisions at other vertices. Thus the overall task of labeling data on a graph constitutes a particular form of \textit{structured prediction} in the field of machine learning \cite{StructuredPrediction:2007aa}.

\textit{Assignment flows} denote a particular class of approaches for data labeling on graphs \cite{Astrom:2017ac,Schnorr:2019aa}. The basic idea is to represent each possible label assignment at vertex $i\in\mc{V}$ by an \textit{assignment vector} $S_{i}\in\Delta_{c}$ in the standard probability simplex, whose vertices encode the unique label assignment for every label by the corresponding unit vector $e_{j},\; j\in[c]$. Data labeling is accomplished by computing the flow $S(t)$ of the dynamical system
\begin{equation}\label{eq:S-flow-intro}
\dot S = R_{S}[\Omega S],\qquad S(0)=S_{0},
\end{equation}
with the row-stochastic matrix $S(t)$ and row vectors $S_{i}(t)$ as state, which under mild conditions converges to unique label assignment vectors (unit vectors) at every vertex $i\in\mc{V}$ \cite{zern2021assignment}. The vector field on the right-hand side in \eqref{eq:S-flow-intro} is parametrized by parameters collected in a matrix $\Omega$. These parameters strongly affect the contextual label assignments. They can be learned from data in order to take into account typical relations of data in the current field of application \cite{Huhnerbein:2021th}. For a demonstration of the application of this approach to a challenging medical imaging problem, we refer to \cite{Sitenko:2021vu}.

From a geometric viewpoint, the system \eqref{eq:S-flow-intro} can be characterized as a collection of individual flows $S_{i}(t)$ at each vertex which are \textit{coupled} by the parameters $\Omega$. Each individual flow is determined by a \textit{replicator equation} which constitutes a basic class of dynamical systems known from evolutionary game theory \cite{Hofbauer:2003aa,Sandholm:2010aa}. By restricting each vector $S_{i}(t)$ to the relative interior $\mathring{\Delta}_{c}$ of the probability simplex (i.e.~the set of strictly positive discrete probability vectors) and by turning this convex set into a statistical manifold equipped with the Fisher-Rao geometry \cite{Amari:2000aa}, the assignment flow \eqref{eq:S-flow-intro} becomes a Riemannian ascent flow on the corresponding product manifold. The underlying information geometry is not only important for making the flow converge to unique label assignments but also for the design of efficient algorithms that actually determine the assignments \cite{Zeilmann:2020aa}. For extensions of the basic assignment flow approach to unsupervised scenarios of machine learning and for an in-depth discussion of connections to other closely related work on structured prediction on graphs, we refer to \cite{Zern:2020ab,Zisler:2020aa} and \cite{Sitenko:2023aa}, respectively.

In this paper, we study a novel and substantial generalization of assignment flows from the different point of view: assignment of labels to metric data where the labels are elements of a \textit{continuous} set. This requires to replace the simplex $\Delta_{c}$ as state space which can only represent assignments of labels from a \textit{finite} set. The substitute for assignment vectors $S_{i},\; i\in\mc{V}$ are Hermitian positive definite \textit{density matrices} $\rho_{i},\; i\in\mc{V}$ with unit trace, 
\begin{equation}
\mc{D}_{c} = \{\rho\in\C^{c\times c}\colon \rho=\rho^{\ast},\; \tr\rho=1\}.
\end{equation}
Accordingly, the finite set of unit vectors $e_{j},\; j\in[c]$ (vertices of $\Delta_{c}$) are replaced by \textit{rank-one} density matrices $\rho^{\infty}$, a.k.a.~\textit{pure states} in quantum mechanics \cite{Bengtsson:2017aa}. The resulting \textit{quantum state assignment flow (QSAF)},
\begin{equation}\label{eq:QSAF-intro}
\dot\rho = \Rep_{\rho}\big[\Omega[\rho]\big],\quad
\rho(0)=\rho_{0},
\end{equation}
has a form similar to \eqref{eq:S-flow-intro} due to adopting the design strategy: the system \eqref{eq:QSAF-intro} couples the individual evolutions $\rho_{i}(t)$ at each vertex $i\in\mc{V}$ through parameters $\Omega$, and the underlying information geometry causes convergence of each $\rho_{i}(t)$ towards a pure state. Using a different state space  $\mc{D}_{c}$ (rather than $\mathring\Delta_{c}$ in \eqref{eq:S-flow-intro}) requires to adopt a different Riemannian metric which results in a corresponding definition of the operator $\Rep_{\rho}$.

Our approach is natural in that restricting \eqref{eq:QSAF-intro} to \textit{diagonal} density matrices results in \eqref{eq:S-flow-intro}, after identifying each vector $\diag(\rho_{i})$ of diagonal entries of the density matrix $\rho_{i}$ with an assignment vector $S_{i}\in\mathring{\Delta}_{c}$. Conversely, \eqref{eq:QSAF-intro} considerably generalizes \eqref{eq:S-flow-intro} and enhances modelling expressivity due to the \textit{noncommutative} interaction of the state spaces $\rho_{i},\; i\in\mc{V}$ across the underlying graph $\mc{G}$, when the quantum state assignment flow is computed by applying geometric numerical integration to \eqref{eq:QSAF-intro}.

We regard our approach merely as an \textit{approach to data representation and analysis}, rather than a contribution to quantum mechanics. For example, the dynamics \eqref{eq:QSAF-intro} clearly differs from the Hamiltonian evolution of quantum systems. Yet we adopt the term `quantum state' since not only density matrices as state spaces, but also the related information geometry, have been largely motivated by quantum mechanics and quantum information theory \cite{Amari:2000aa,Petz:2008aa}.

\subsection{Contribution and Organization}

Section~\ref{sec:information-geometry} summarizes the information geometry of both the statistical manifold of categorial distributions and the manifold of strictly positive definite density matrices. 
Section~\ref{sec:AF} summarizes the assignment flow approach \eqref{eq:S-flow-intro}, as a reference for the subsequent generalization to \eqref{eq:QSAF-intro}. This generalization is the main contribution of this paper and presented in Section~\ref{sec:DAF-all}. Each row of the table below specifies the section where an increasingly general version of the original assignment flow (left column) is generalized to the corresponding quantum state assignment flow (right column, same row).

\begin{center}\small
\begin{tabular}{|c|c|} 
\hline
\textbf{Assignment Flow (AF)}  & \textbf{Quantum State AF (QSAF)}
\\ \hline\hline
single-vertex AF (Section \ref{sec:single-vertex-AF}) 
& single-vertex QSAF (Section \ref{sec:single-vertex-DAF}) 
\\ \hline
AF approach (Section \ref{sec:introduction_AF})
& QSAF approach (Section \ref{sec:DAF})
\\ \hline
Riemannian gradient AF (Section \ref{sec:standardAF_S-flow})
& Riemannian gradient QSAF (Section \ref{sec:S-DAF})
\\ \hline
\multicolumn{2}{|c|}{
recovery of the AF from the QSAF by restriction (Section \ref{sec:recovering-af-cat-distr})
}
\\ \hline
\end{tabular}
\end{center}

Alternative metrics on the positive definite matrix manifold which have been used in the literature, are reviewed in Section~\ref{sec:other_geometries}, in order to position our approach also from this point of view. 
Few academical experiments illustrate properties of the novel approach in Section~\ref{sec:Experiments}. Working out a particular scenario of data analysis is beyond the scope of this paper. 
We conclude and indicate directions of further work in Section~\ref{sec:Conclusion}. In order not to compromise the reading flow, proofs are listed in Section \ref{sec:appendix}.

This paper considerably elaborates the short preliminary conference version \cite{Schwarz:2023aa}. \nocite{SSVM:2023aa}

\subsection{Basic Notation}
For the readers convenience, we specify below the basic notation and notational conventions used in this paper.

\vspace{0.2cm}
\noindent
\begin{tabular}{ll}
$[c]$ &
$\{1,2,\dotsc,c\},\quad c\in\N$\\
$\eins_{c}$ & $(1,1,\dotsc,1)^{\T}\in\R^{c}$\\
$\R^{c}_{+}$ &
$\{x\in\R^{c}\colon x_{i}\geq 0,\; i\in[c]\}$\\
$\R^{c}_{++}$ &
$\{x\in\R^{c}\colon x_{i}> 0,\; i\in[c]\}$\\
$e_{1}, e_{2}, \dotsc$ & canonical basis vectors of $\R^{c}$\\
$\la u,v\ra$ & Euclidean inner vector product\\
$\|u\|$ & Euclidean norm $\sqrt{\la u,u \ra}$\\
$I_{c}$ & unit matrix of $\R^{c\times c}$\\
$p\cdot q$ & componentwise vector multiplication $(p\cdot q)_{i}=p_{i}q_{i},\;i\in[c],\;p,q\in\R^{c}$\\
$\frac{q}{p}$ & componentwise division $\big(\frac{q}{p}\big)_{i}=\frac{q_{i}}{p_{i}},\; i\in[c],\; q\in\R^{c},\; p\in\R_{++}^{c}$\\
$\mc{H}_{c}$ & space of Hermitian $c\times c$ matrices (cf.~\eqref{eq:def-Hc}) \\
$\tr(A)$ & trace $\sum_{i} A_{ii}$ of a matrix $A$\\
$\la A, B \ra$ & matrix inner product $\tr(A B)$, $A, B\in\mc{H}_{c}$ \\
$[A,B]$ & commutator $A B - B A$ \\
$\Diag(v)$ & the diagonal matrix with vector $v$ as entries\\
$\diag(V)$ & the vector of the diagonal entries of a square matrix $V$\\
$\expm$ & the matrix exponential\\
$\logm$ & the matrix logarithm $\expm^{-1}$\\
$\Delta_{c}$ & the set of discrete probability vectors of dimension $c$ (cf.~\eqref{eq:def-Delta-c})\\
$\mc{S}_{c}$ & the relative interior of $\Delta_{c}$, i.e.~the set of strictly positive probability vectors (cf.~$\eqref{eq:def-mcSc}$)\\
$\mc{W}_{c}$ & the product manifold $\mc{S}_{c}\times\dotsb\times\mc{S}_{c}$ (cf.~$\eqref{eq:def-mcW}$)\\
$\mc{P}_{c}$ & the set of symmetric positive definite $c\times c$ matrices (cf.~\eqref{eq:def-mcPc})\\
$\mc{D}_{c}$ & the subset of matrices in $\mc{P}_{c}$ whose trace is equal to $1$ (cf.~\eqref{eq:def-mcDc})\\
$\mc{Q}_{c}$ & the product manifold $\mc{D}_{c}\times\dotsb\times\mc{D}_{c}$ (cf.~\eqref{eq:def-mcQ})\\
$\eins_{\mc{S}_{c}}$ & barycenter $\frac{1}{c}\eins_{c}$ of the manifold $\mc{S}_{c}$ \\
$\eins_{\mc{W}_{c}}$ & barycenter $(\eins_{\mc{S}_{c}}, \eins_{\mc{S}_{c}},\dotsc, \eins_{\mc{S}_{c}})^{\T}$ of the manifold $\mc{W}$ \\
$\eins_{\mc{D}_{c}}$ & matrix $\Diag(\eins_{\mc{S}_{c}})\in\mc{D}_{c}\subset \C^{c\times c}$\\
$g_{p}, g_{W}, g_{\rho}$ & the Riemannian metrics on $\mc{S}_{c}, \mc{W}_{c}, \mc{D}_{c}$ (cf.~\eqref{eq:def-gp}, \eqref{eq:def-g-W}, \eqref{eq:Bogoliubov})\\
$T_{c,0},\mc{T}_{c,0},\mc{H}_{c,0}$ & the tangent spaces to $\mc{S}_{c}, \mc{W}_{c}, \mc{D}_{c}$ (cf.~\eqref{eq:def-T0}, \eqref{eq:def-g-W}, \eqref{eq:def-mcH-0-c})\\
$\pi_{c,0}, \Pi_{c,0}$ & orthogonal projections onto $T_{0}, \mc{H}_{c,0}$ (cf.~\eqref{eq:def-Pi0-p}, \eqref{eq:def-Pic0})\\
$R_{p}, R_{W}, \Rep_{\rho}$ & replicator operators associated with the assignment flows \\
& on $\mc{S}_{c}, \mc{W}_{c}, \mc{D}_{c}, \mc{Q}_{c}$ (cf.~\eqref{eq:def-Rp}, \eqref{eq:def-replicator-mcW}, \eqref{eq:def-Rrho}, \eqref{eq:def-R-rho})\\
$\partial$ & Euclidean gradient operator: $\partial f(p) = \big(\partial_{p_{1}} f(p), \partial_{p_{2}} f(p),\dotsc\big)^{\T}$\\
$\ggrad$ & Riemannian gradient operator with respect to the Fisher-Rao metric\\
$R_{W}[\cdot], \Omega[\cdot]$, etc. & square brackets indicate a linear operator which acts in a non-standard way, \\ & e.g.~row-wise to a matrix argument.
\end{tabular}

\section{Information Geometry}
~\label{sec:information-geometry}

\textit{Information geometry}~\cite{Amari:1985vc,Lauritzen:1987aa} is concerned with the representation of parametric probability distributions from a geometric viewpoint like, 
e.g., the exponential familiy of distributions \cite{Brown:1986vy}. 
Specifically, an open convex set $\mc{M}$ of parameters of a probability distribution becomes a Riemannian manifold $(\mc{M},g)$ 
when equipped with a Riemannian metric $g$.
The \textit{Fisher-Rao metric} is the canonical choice due to its invariance properties with respect to reparametrization~\cite{Cencov:1981aa}. 
A closely related scenario concerns the representation of the interior of compact convex bodies as Riemannian manifolds $(\mc{M},g)$ 
due to the correspondence between compactly supported Borel probability measures and an affine equivalence class of convex 
bodies~\cite{Brazitikos:2014aa}.

A key ingredient of information geometry is 
the so-called \textit{$\alpha$-family of affine connections} introduced by Amari~\cite{Amari:1985vc}, which comprises the so-called $e$-connection 
$\nabla$ and $m$-connection $\nabla^{\ast}$ as special cases. These connections are torsion-free and dual to each other 
in the sense that they jointly satisfy the equation which uniquely characterizes the Levi-Civita connection as metric 
connection~\cite[Def.~3.1, Thm.~3.1]{Amari:1985vc}. Regarding numerical computations, working with the exponential map induced 
by the $e$-connection is particularly convenient since its domain is the entire tangent space. We refer 
to~\cite{Amari:2000aa,Calin:2014aa,Ay:2017aa} for further reading and to~\cite{Petz:1994aa},~\cite[Ch.~7]{Amari:2000aa} 
for the specific case of quantum state spaces.

In this paper, we are concerned with two classes of convex sets, 
\begin{itemize}
\item
the relative interior of probability simplices, 
each of which represents the categorical (discrete) distributions of the corresponding dimension, and 
\item
the set of 
positive-definite symmetric matrices with trace one.
\end{itemize}
Sections~\ref{sec:intro-simplex} and~\ref{sec:intro-density-matrices} 
introduce the information geometry for the former and the latter class of sets, respectively.

\subsection{Categorical Distributions}\label{sec:intro-simplex}
We set
\begin{equation}
[c]:=\{1,2,\dotsc,c\},\qquad c\in\N.
\end{equation}
and denote the probability simplex of distributions on $[c]$ by
\begin{equation}\label{eq:def-Delta-c}
\Delta_{c}:=\Big\{p\in\R_{+}^{c}\colon \la\eins_{c},p\ra=\sum_{i\in[c]}p_{i}=1\Big\},\qquad
\eins_{c} := (1,1,\dotsc,1)^{\T}\in\R^{c}.
\end{equation}
Its relative interior equipped with the Fisher-Rao metric becomes the Riemannian manifold $(\mc{S}_{c},g)$,
\begin{equation}\label{eq:def-mcSc}
\mc{S}_{c} := \rint\Delta_{c}
= \{p\in\Delta_{c}\colon p_{i}>0,\;i\in[c]\},
\end{equation}
\begin{equation}\label{eq:def-gp}
g_{p}(u,v) 
:= \sum_{i\in[c]}\frac{u_{i} v_{i}}{p_{i}}
= \la u, \Diag(p)^{-1} v\ra,\quad
\forall u,v\in T_{c,0},\quad
p\in\mc{S}_{c},
\end{equation}
with trivial tangent bundle given by
\begin{equation}
T\mc{S}_{c} \cong \mc{S}_{c}\times T_{c,0}
\end{equation}
and the tangent space
\begin{equation}\label{eq:def-T0}
T_{c,0}
:= T_{\eins_{\mc{S}_{c}}}\mc{S}_{c}
=\{v\in\R^{c}\colon \la\eins_{c},v\ra=0\}.
\end{equation}
The orthogonal projection onto $T_{c,0}$ is denoted by
\begin{equation}\label{eq:def-Pi0-p}
\pi_{c,0}\colon\R^{c}\to T_{c,0},\qquad
\pi_{c,0} v 
:= v - \frac{1}{c}\la\eins_{c}, v\ra\eins_{c}
= \Big(I_{c}-\eins_{c}\eins_{\mc{S}_{c}}^{\T}\Big) v.
\end{equation}

\vspace{0.2cm}
The mapping defined next plays a major role in all dynamical systems being under consideration in this paper.
\begin{definition}[\textbf{replicator operator}]
The replicator operator is the linear mapping of the tangent space
\begin{equation}\label{eq:def-Rp}
R\colon\mc{S}_{c}\times T_{c,0}\to T_{c,0},\qquad
R_{p}v := (\Diag(p)-p p^{\T}) v,\qquad
p\in\mc{S}_{c},\quad v\in T_{c,0}
\end{equation}
parametrized by $p\in\mc{S}_{c}$.
\end{definition}
The name `replicator' is due to the role of this mapping in evolutionary game theory; see Remark~\ref{rem:replicator-eq} on 
page~\pageref{rem:replicator-eq}.
\begin{proposition}[\textbf{properties of $R_{p}$}]\label{prop:properties-of-Rp}
The mapping \eqref{eq:def-Rp} satisfies 
\begin{subequations}\label{eq:Rp-Pi0}
\begin{align}
R_{p}\eins_{c} &= 0, 
\label{eq:Rp-eins} \\ \label{eq:Rp-Pi0-commute}
\pi_{c,0}R_{p} &= R_{p} \pi_{c,0}
= R_{p},\quad \forall p\in\mc{S}_{c}.
\end{align}
\end{subequations}
Furthermore, let $f\colon\mc{S}_{c}\to\R$ be a smooth function and $\wt{f}\colon U\to\R$ a smooth extension of $f$ to an open neighborhood $U$ of $\mc{S}_{c}\subset \R^{c}$ with $\wt{f}\vert_{\mc{S}_{c}} = f$. Then the Riemannian gradient of $f$ with respect to the Fisher-Rao metric \eqref{eq:def-gp} is given by
\begin{equation}\label{eq:grad-R-simplex}
\ggrad f(p) = R_{p}\partial \wt{f}(p).
\end{equation} 
\end{proposition}
\begin{proof}
Appendix~\ref{sec:appendix-sec-information-geometry}
\end{proof}
\begin{remark} 
Equations \eqref{eq:grad-R-simplex} and \eqref{eq:Rp-partial-f-local}, respectively, show that the replicator operator $R_{p}$ is the inverse metric tensor with respect to the Fisher-Rao metric \eqref{eq:def-gp},  expressed in the ambient coordinates.
\end{remark}

\vspace{0.2cm}
The exponential map induced by the $e$-connection is defined on the entire space $T_{c,0}$ and reads \cite{Ay:2017aa}
\begin{equation}
\Exp\colon\mc{S}_{c}\times T_{c,0}\to\mc{S}_{c},\qquad
\Exp_{p}(v) := \frac{p\cdot e^{\frac{v}{p}}}{\la p,e^{\frac{v}{p}}\ra},\qquad
p\in\mc{S}_{c},\quad v\in T_{c,0}.
\end{equation}

\subsection{Density Matrices}\label{sec:intro-density-matrices}
We denote the open convex cone of positive definite matrices by
\begin{equation}\label{eq:def-mcPc}
\mc{P}_{c} := \{\rho\in\C^{c\times c}\colon \rho=\rho^{*},\;\rho\succ 0\}
\end{equation}
and the manifold of strictly positive definite density matrices by
\begin{equation}\label{eq:def-mcDc}
\mc{D}_{c} := \{\rho\in\mc{P}_{c}\colon \tr\rho=1\}.
\end{equation}
$\mc{D}_{c}$ is the intersection of $\mc{P}_{c}$ and the hyperplane defined by the trace-one constraint. Its closure $\ol{\mc{D}}_{c}$ is convex and compact. We can identify the space $\mc{D}_{c}$ as the space of invertible density operators, in the sense of quantum mechanics, on the finite-dimensional Hilbert space $\C^{c}$ without loss of generality.
Any matrix ensemble of the form
\begin{equation}\label{eq:POVM}
\{M_{i}\}_{i\in[n]}\subset\ol{\mc{P}}_{c}\colon\quad
\sum_{i\in[n]}M_{i}=I_{c}
\end{equation}
induces the probability distribution on $[n]$ via the Born rule
\begin{equation}
p\in\Delta_{n}\colon\quad
p_{i} = \la M_{i},\rho \ra = \tr(M_{i}\rho),\quad i\in[n].
\end{equation}
\eqref{eq:POVM} is called \textit{positive operator valued measure (POVM)}. We refer to~\cite{Bengtsson:2017aa} for the physical background and to~\cite{Bodmann:2020wx} and references therein for the mathematical background.

The analog of~\eqref{eq:def-T0} is the tangent space which, at any point $\rho \in \mc{D}_c$, is equal to the space of trace-less symmetric matrices
\begin{subequations}
\begin{align}\label{eq:def-mcH-0-c}
\mc{H}_{c,0}
&:= \mc{H}_{c}\cap \{X\in\C^{c\times c}\colon \tr X=0\},
\intertext{where}\label{eq:def-Hc}
\mc{H}_{c} &:= \{X\in\C^{c\times c}\colon X^{*}=X\}.
\end{align}
\end{subequations}
The manifold $\mc{D}_c$ therefore has a trivial tangent bundle given by
\begin{equation}
    T\mc{D}_{c} = \mc{D}_{c}\times \mc{H}_{c,0},
\end{equation}
with the tangent space $\mc{H}_{c,0}=T_{\eins_{\mc{D}_{c}}}\mc{D}_{c}$ defined in equation~\eqref{eq:def-mcH-0-c}.
The corresponding orthogonal projection onto the tangent space $\mc{H}_{c,0}$ reads
\begin{equation}\label{eq:def-Pic0}
\Pi_{c,0}\colon\mc{H}_{c}\to\mc{H}_{c,0},\qquad
\Pi_{c,0}[X] := X-\frac{\tr{X}}{c} I_{c}.
\end{equation}
Equipping the manifold $\mc{D}_c$ as defined in equation~\eqref{eq:def-mcDc} with the \textit{Bogoliubov-Kubo-Mori (BKM) metric}~\cite{petz1993bogoliubov}
results in a Riemannian manifold $(\mc{D}_{c},g)$. 
Using $T_\rho \mc{D}_c = \mc{H}_{c,0}$, 
this metric can be expressed by
\begin{equation}\label{eq:Bogoliubov}
g_{\rho}(X,Y) := \int_{0}^{\infty}\tr\big(X (\rho+\lambda I)^{-1} Y (\rho+\lambda I)^{-1}\big)d\lambda,\quad
X, Y\in\mc{H}_{c,0},\quad
\rho\in\mc{D}_{c}.
\end{equation}
This metric uniquely ensures the existence of a symmetric e-connection $\nabla$ on $\mc{D}_c$ 
that it mutually dual to its m-connection $\nabla^{\ast}$ in the sense of information geometry, 
leading to the \textit{dually-flat} structure $(g,\nabla,\nabla^{\ast})$~\cite{Grasselli:2001aa},~\cite[Thm.~7.1]{Amari:2000aa}.

The following map and its inverse, defined in terms of the matrix exponential $\expm$ and its inverse $\logm=\expm^{-1}$, will be convenient.
\begin{subequations}
\begin{align}
\mb{T}\colon\mc{D}_{c}\times\mc{H}_{c}&\to\mc{H}_{c},
\\ 
\mb{T}_{\rho}[X] &:= \frac{d}{dt}\logm(\rho+t X)\big|_{t=0}
= \int_{0}^{\infty}(\rho+\lambda I)^{-1} X (\rho+\lambda I)^{-1}d\lambda, 
\label{eq:def-mcT} \\ \label{eq:def-mcT-inverse}
\mb{T}_{\rho}^{-1}[X]
&= \frac{d}{dt}\expm(H+t X)\big|_{t=0}
= \int_{0}^{1}\rho^{1-\lambda} X\rho^{\lambda}d\lambda,\qquad
\rho = \exp_{m}(H).
\end{align}
\end{subequations}
The inner product \eqref{eq:Bogoliubov} may now be written in the form
\begin{equation}\label{eq:def-inner-rho-T}
g_{\rho}(X, Y) = \la\mb{T}_{\rho}[X],Y\ra,
\end{equation}
since the trace is invariant with respect to cyclic permutations of a matrix product as argument. Likewise, 
\begin{equation}\label{eq:ToDo-tr}
\la\rho,X\ra = \tr(\rho X)
= \tr\mb{T}_{\rho}^{-1}[X].
\end{equation}
We consider also two subspaces on the tangent space $T_\rho \mc{D}_c$, 
\begin{subequations}\label{eq:decomp_H_c0}
\begin{align}
T_\rho^{u} \mc{D}_{c}
&:= \left\{ X \in \mc{H}_{c,0} \colon \exists \Omega = - \Omega^* \text{ such that } X = [\Omega,\rho] \right\},
\\ \label{eq:decomp_H_c0-b}
T_{\rho}^{c} \mc{D}_{c} 
&:= \left\{ X \in \mc{H}_{c,0} \colon \ [\rho,X ]=0 \right\},
\end{align}
\end{subequations}
which yield the decomposition \cite{Amari:2000aa} 
\begin{equation}
    T_\rho \mc{D}_c =  T_\rho^{c} \mc{D}_c \oplus  T_\rho^{u} \mc{D}_c.
\end{equation}
In Section~\ref{sec:recovering-af-cat-distr}, we will use this decomposition to recover the assignment flow for categorical distributions from 
the quantum state assignment flow, by restriction to a submanifold of commuting matrices.


\subsection{Alternative Metrics and Geometries}\label{sec:other_geometries}

The positive definite matrix manifold $\mc{P}_c$\footnote{We confine ourselves in this subsection to the case of of real density matrices, as our main references for comparison only deal with real matrix manifolds.} has become a tool for data modelling and analysis during the last two decades. Accordingly, a range of Riemannian metrics exist with varying properties. A major subclass is formed by the $O(n)$-invariant metrics, 
including the log-Euclidean, affine-invariant, Bures-Wasserstein and Bogoliubov-Kubo-Mori (BKM) metric. We refer to \cite{thanwerdas2023n} for a comprehensive recent survey.

This section provides a brief comparison of the \textit{BKM metric} \eqref{eq:Bogoliubov}, adopted in this paper, with two often employed metrics in the literature, the \textit{affine-invariant metric} and the \textit{log-Euclidean metric}, which may be regarded as `antipodal points' in the space of metrics from the geometric and the computational viewpoint, respectively. 

\subsubsection{Affine-Invariant Metrics}

The affine-invariant metric has been derived in various ways, e.g.~based on the canonical matrix inner product on the tangent space  \cite[Section 6]{Bhatia:2006aa} or as Fisher-Rao metric on the statistical manifold of centered multivariate Gaussian densities \cite{skovgaard1984riemannian}. 
The metric is given by 
\begin{equation}\label{eq:affine-invariant-metric}
g_{\rho}(X,Y) = \mathrm{tr} \big( \rho^{-\frac{1}{2}}X\rho^{-\frac{1}{2}} \rho^{-\frac{1}{2}}Y\rho^{-\frac{1}{2}} \big) = \mathrm{tr} \left( \rho^{-1}X\rho^{-1}Y \right),\qquad \rho \in \mc{P}_c,\quad X,Y \in T_{\rho}\mc{P}_c.
\end{equation}
The exponential map with respect to the Levi-Civita connection reads
\begin{equation}\label{eq:exp_affine_inv}
\exp_\rho^{(\text{aff})}(X) = \rho^{\frac{1}{2}} \expm \big(\rho^{-\frac{1}{2}}X\rho^{-\frac{1}{2}}\big)\rho^{\frac{1}{2}},\qquad \rho \in \mc{P}_c,\quad X \in T_{\rho}\mc{P}_c.
\end{equation}
This Riemannian structure turns $\mc{P}_{c}$ into a manifold with negative sectional curvature \cite[Ch.~II.10]{Bridson:1999aa}, which is convenient from the geometric viewpoint due to uniquely defined Riemannian means and geodesic convexity \cite[Section 6.9]{Jost:2017aa}. On the other hand, evaluating \eqref{eq:affine-invariant-metric} and \eqref{eq:exp_affine_inv} is computationally expensive, in particular when computing the quantum state assignment flow which essentially involves geometric averaging.

\subsubsection{Log-Euclidean Metric}

The log-Euclidean metric, introduced by \cite{Arsigny:2007aa}, is the pullback of the canonical matrix inner product with respect to the matrix logarithm and given by
\begin{equation}\label{eq:log_euclid_metric}
g_{\rho}(X,Y) = \mathrm{tr} \left( d \logm(\rho)[X], d\logm(\rho)[Y] \right) 
\overset{\eqref{eq:def-mcT}}{=} \langle \TT_\rho[X],\TT_\rho[Y] \rangle,\qquad\rho \in \mc{P}_c\quad
X,Y \in T_\rho \mc{P}_c.
\end{equation}
The exponential map reads
\begin{equation}\label{eq:exp_log_euclidean}
\exp_\rho^{(\text{log})}(X) = \expm \big(\logm(\rho) + \mathbb{T}_\rho[X]\big),\qquad\rho \in \mc{P}_c\quad
X,Y \in T_\rho \mc{P}_c
\end{equation}
and is much more convenient from the computational viewpoint. 
Endowed with this metric, the space $\mc{P}_{c}$ is isometric to a Euclidean space. The log-Euclidean metric is not curved and merely invariant under orthogonal transforms and dilations \cite{thanwerdas2023n}.

\subsubsection{Comparison to Bogoliubov-Kubo-Mori Metric}

The BKM metric \eqref{eq:Bogoliubov}, \eqref{eq:ToDo-tr}, given by
\begin{equation}
g_{\rho}(X, Y) = \la\mb{T}_{\rho}[X],Y\ra,\qquad\rho \in \mc{P}_c\quad
X,Y \in T_\rho \mc{P}_c,
\end{equation}
looks similar to the log-Euclidean metric \eqref{eq:log_euclid_metric}. Regarding them both as members of the class of \textit{mean kernel metrics} \cite[Def. 4.1]{thanwerdas2023n} enables an intuitive comparison. For real-valued matrices, mean kernel metrics have the form
\begin{equation}
g_{\rho}(X,X) = g_{D}(X',X')
= \sum_{i,j\in[c]}\frac{(X'_{ij})^{2}}{\phi(D_{ii},D_{jj})},\qquad \rho= V D V^{\T}\quad V\in O(n),\quad X = V X' V^{\T},\quad 
\end{equation}
with a diagonal matrix $D = \Diag(D_{11},\dotsc,D_{cc})$ and a bivariate function $\phi(x,y)= a \,m(x,y)^{\theta},\; a>0$ in terms of a symmetric homogeneous mean $m\colon \R_{+}\times \R_{+}\to\R_{+}$. Regarding the log-Euclidean metric, one has $\phi(x,y) = \big(\frac{x-y}{\log x -\log y}\big)^{2}$, whereas for the BKM metric one has $\phi(x,y) = \frac{x-y}{\log x -\log y}$.

Taking also the restriction to density matrices $\mc{D}_{c}\subset\mc{P}_{c}$ into account, one has the relation 
\begin{subequations}
\begin{align}\label{eq:exp-log-e}
\exp_{\rho}^{(\log)}(Y)
&= \Exp_{\rho}^{(e)}(X),\qquad \rho\in\mc{D}_{c},\quad X \in\mc{H}_{c,0},
\\ \label{eq:exp-log-e-Y}
Y &= X - \log\Big(\tr\exp_{m}\big(\logm(\rho) + \mb{T}_{\rho}[X]\big)\Big) \rho,
\end{align}
\end{subequations}
as will be shown below as Remark \ref{rem:similarity-to-AF-2}. Here, the left-hand side of \eqref{eq:exp-log-e} is the exponential map \eqref{eq:exp_log_euclidean} induced by the log-Euclidean metric and $\Exp_{\rho}^{(e)}$ is the exponential map with respect to the affine e-connection of information geometry, as detailed below by Proposition \ref{prop:e-geodesic}. This close relationship of the e-exponential map $\Exp_{\rho}^{(e)}$ to the exponential map of the log-Euclidean metric highlights the computational efficiency of using BKM metric, which we adopt for our approach. This is also motivated by the lack of an explicit formula for the exponential map with respect to the Levi-Civita connection \cite{michor2000curvature}. To date, the sign of the curvature is not known either.

We note that to our best knowledge, the introduction of the affine connections of information geometry, as surrogates of the Riemannian connection for any statistical manifold, predates the introduction of the log-Euclidean metric for the specific space $\mc{P}_{c}$.

\section{Assignment Flows}
\label{sec:AF}
The assignment flow approach has been informally introduced in Section \ref{sec:Introduction}. In this section, we summarize the mathematical ingredients of this approach, as a reference for the subsequent generalization to quantum states (density matrices) in Section \ref{sec:DAF-all}. Sections \ref{sec:single-vertex-AF} and \ref{sec:introduction_AF} introduce the assignment flow on a single vertex and on an arbitrary graph, respectively. A reparametrization turns the latter into a Riemannian gradient flow (Section \ref{sec:standardAF_S-flow}). 
Throughout this section, we refer to definitions and notions  introduced in Section \ref{sec:intro-simplex}.

\subsection{Single-Vertex Assignment Flow}\label{sec:single-vertex-AF}
Let $D=(D_{1},\dotsc,D_{c})^{\T}\in\R^{c}$ and consider the task to pick the smallest components of $D$. Formulating this operation as optimization problem amounts to evaluating the support function (in the sense of convex analysis \cite[p.~28]{Rockafellar:1970ab}) of the probability simplex $\Delta_{c}$ at $-D$,
\begin{equation}\label{eq:data-term}
\min_{j\in[c]} \{D_{1},\dotsc,D_{c}\}
= \max_{p\in\Delta_{c}}\la -D, p\ra.
\end{equation}
In practice, the vector $D$ represents real-valued noisy measurements at some vertex $i\in\mc{V}$ of an underlying graph $\mc{G}=(\mc{V},\mc{E})$ and hence will be in `general position', that is the minimal component will be unique: if $j^{\ast}\in[c]$ indexes the minimal component $D_{j^{\ast}}$, then the corresponding unit vector $p^{\ast}=e_{j^{\ast}}$ will maximize the right-hand side of \eqref{eq:data-term}. We call \textit{assignment vectors} such vectors which assign a label (index) to observed data vectors.

If $D$ varies, the operation \eqref{eq:data-term} is non-smooth, however. In view of a desired interaction of label assignments across the graph (cf.~Section \ref{sec:introduction_AF}), we therefore replace this operation by a \textit{smooth} dynamical system whose solution converges to the desired assignment vector. To this end, the vector $D$ is represented on $\mc{S}_{c}$ as \textit{likelihood vector}
\begin{equation}\label{eq:def-LpD}
L_{p}(D) := \exp_{p}(-\pi_{c;0} D) 
\overset{\eqref{eq:Rp-Pi0-commute}}{=}
\exp_{p}(-D),\qquad
p\in\mc{S}_{c},
\end{equation}
where
\begin{equation}\label{eq:def_exp_standard_af}
\exp\colon\mc{S}_{c}\times T_{c;0}\to\mc{S}_{c},\qquad
\exp_{p}(v) := \Exp_{p}\circ R_{p}(v)
= \frac{p\cdot e^{v}}{\la p, e^{v}\ra},\qquad
p\in\mc{S}_{c}.
\end{equation}
The \textit{single-vertex assignment flow} equation reads
\begin{equation}\label{eq:single-vertex-AF}
\dot p = R_{p} L_{p}(D)
= p\cdot \big(L_{p}(D)-\la p,L_{p}(D)\ra\eins_{c}\big),\qquad
p(0)=\eins_{\mc{S}_{c}}.
\end{equation}
Its solution $p(t)$ converges to the vector that solves the label assignment problem \eqref{eq:data-term}, see Corollary \ref{prop:single-vertex-AF} below.
\begin{remark}[\textbf{replicator equation}]\label{rem:replicator-eq}
Differential equations of the form \eqref{eq:single-vertex-AF}, with some $\R^{c}$-valued function $F(p)$ in place of $L_{p}(D)$, are known as \textit{replicator equation} in evolutionary game theory \cite{Hofbauer:2003aa}. 
\end{remark}
\begin{lemma}\label{lem:derivatives_exp_standard_af}
Let $p\in\mc{S}_{c}$. Then the differentials of the mapping \eqref{eq:def_exp_standard_af} with respect to $p$ and $v$ are given by
\begin{subequations}
\begin{align}
d_{v} \exp_{p}(v)[u]
&= R_{\exp_{p}(v)} u,
\label{eq:dexp-p} \\ \label{eq:dLpD}
d_{p} \exp_{p}(v)[u]
&= R_{\exp_{p}(v)}\frac{u}{p},\qquad p\in\mc{S}_{c},\quad u, v\in T_{c;0}.
\end{align}
\end{subequations}
\end{lemma}
\begin{proof}
Appendix~\ref{sec:appendix_standard_af_proofs}.
\end{proof}
\begin{theorem}[\textbf{single vertex assignment flow}]\label{thm:single-vertex-AF-parametrization}
The single-vertex assignment flow equation \eqref{eq:single-vertex-AF} is equivalent to the system
\begin{subequations}\label{eq:single-vertex-AF-parametrization}
\begin{align}
\label{eq:single-vertex-AF-parametrization-p}
\dot p &= R_{p} q,\qquad p(0) = \eins_{\mc{S}_{c}},
\\ \label{eq:single-vertex-AF-parametrization-b}
\dot q &= R_{q} q,\qquad q(0) = L_{\eins_{\mc{S}_{c}}}(D),
\intertext{with solution given by} \label{eq:single-vertex-AF-parametrization-a}
p(t) &= \exp_{\eins_{\mc{S}_{c}}}\Big(\int_{0}^{t} q(\tau)d\tau\Big).
\end{align}
\end{subequations}
\end{theorem}
\begin{proof}
Appendix~\ref{sec:appendix_standard_af_proofs}.
\end{proof}
\begin{corollary}[\textbf{single vertex label assignment}]\label{prop:single-vertex-AF}
Let $\mc{J}^{\ast}:=\arg\min_{j\in[c]}\{D_{j}\colon j\in[c]\}\subseteq [c]$. Then the solution $p(t)$ to \eqref{eq:single-vertex-AF} satisfies
\begin{equation}\label{eq:single-AF-p-limit}
\lim_{t\to\infty}p(t) = \frac{1}{|\mc{J}^{\ast}|}\sum_{j\in J^{\ast}}e_{j}
\in \arg\max_{p\in\Delta_{c}}\la -D, p\ra.
\end{equation}
In particular, if $D$ has a unique minimal component $D_{j^{\ast}}$, then $p(t)\to e_{j^{\ast}}$ as $t\to\infty$.
\end{corollary}
\begin{proof}
Appendix~\ref{sec:appendix_standard_af_proofs}.
\end{proof}
\subsection{Assignment Flows}
\label{sec:introduction_AF}
The assignment flow approach consists of the weighted interaction -- as define below -- of single-vertex assignment flows, associated with vertices $i\in\mc{V}$ of a weighted graph $\mc{G}=(\mc{V},\mc{E},\w)$ with nonnegative weight function 
\begin{equation}\label{eq:def-omega}
\w\colon\mc{E}\to\R_{+},\qquad
ik\mapsto\w_{ik}.
\end{equation}
The assignment vectors are denoted by $W_{i},\,i\in\mc{V}$ and form the row vectors of a row-stochastic matrix
\begin{equation}\label{eq:def-mcW}
W \in\mc{W}_{c} := \underbrace{\mc{S}_{c}\times\dotsb\times\mc{S}_{c}}_{|\mc{V}|\;\text{factors}}.
\end{equation}
The product space $\mc{W}_{c}$ is called \textit{assignment manifold} $(\mc{W}_{c},g)$, where the metric $g$ is defined by applying \eqref{eq:def-gp} row-wise, 
\begin{equation}\label{eq:def-g-W}
g_{W}(U,V) := \sum_{i\in\mc{V}}g_{W_{i}}(U_{i},V_{i}),\qquad
U, V\in\mc{T}_{c;0}:=T_{c;0}\times\dotsb\times T_{c;0}.
\end{equation}
The \textit{assignment flow equation} generalizing \eqref{eq:single-vertex-AF} reads
\begin{equation}\label{eq:AF}
\dot W = R_{W}[S(W)],
\end{equation}
where the \textit{similarity vectors}
\begin{equation}\label{eq:def-Si}
S_{i}(W) := \Exp_{W_{i}}\Big(\sum_{k\in\mc{N}_{i}}\w_{ik} \Exp_{W_{i}}^{-1}\big(L_{W_{k}}(D_{k})\big)\Big),\qquad i\in\mc{V}
\end{equation}
form the row vectors of the matrix $S(W)\in\mc{W}_{c}$. The neigborhoods
\begin{equation}\label{eq:def-mcN-i}
\mc{N}_{i}
:=\{i\}\cup\{k\in\mc{V}\colon ik \in\mc{E}\}
\end{equation}
are defined by the adjacency relation of the underlying graph $\mc{G}$, and $R_{W}[\cdot]$ of \eqref{eq:AF} applies \eqref{eq:def-Rp} row-wise, 
\begin{equation}\label{eq:def-replicator-mcW}
R_{W}[S(W)]_{i} = R_{W_{i}} S_{i}(W),\qquad i\in\mc{V}.
\end{equation}
Note that the similarity vectors $S_{i}(W)$ given by \eqref{eq:def-Si} result from geometric weighted averaging of the velocity vectors $\Exp_{W_{i}}^{-1}\big(L_{W_{k}}(D_{k})\big)$. The velocities represent given data $D_{i},\; i\in\mc{V}$ via the likelihood vectors $L_{W_{i}}(D_{i})$ given by \eqref{eq:def-LpD}. Each choice of the weights $\w_{ik}$ in \eqref{eq:def-Si} associated with every edge $ik\in\mc{E}$ defines an assignment flow $W(t)$ solving \eqref{eq:AF}. Thus these weight parameters determine how individual label assignments by \eqref{eq:def-LpD} and \eqref{eq:single-vertex-AF} are \textit{regularized}. 

Well-posedness, stability and quantitative estimates of basins of attraction to integral label assignment vectors have been established in \cite{zern2021assignment}. Reliable and efficient algorithms for computing numerically the assignment flow have been devised by \cite{Zeilmann:2020aa}.

\subsection{Reparametrized Assignment Flows}
\label{sec:standardAF_S-flow}
In \cite[Prop.~3.6]{savarino2021continuous}, the following parametrization of the general assignment flow equation \eqref{eq:AF} was introduced, which generalizes the parametrization \eqref{eq:single-vertex-AF-parametrization} of the single-vertex assignment flow \eqref{eq:single-vertex-AF}.
\begin{subequations}\label{eq:S-flow}
\begin{align}
\dot W &= R_{W}[\ol{S}],\qquad 
W(0)=\eins_{\mc{W}_{c}},
\label{eq:S-flow-a} \\ \label{eq:S-flow-b}
\dot{\ol{S}} &= R_{\ol{S}}[\Omega\ol{S}],\qquad
\ol{S}(0) = S(\eins_{\mc{W}_{c}}),
\end{align}
\end{subequations}
with the nonnegative weight matrix corresponding to the weight function \eqref{eq:def-omega},
\begin{equation}\label{eq:def-Omega}
\Omega = (\Omega_{1},\dotsc,\Omega_{|\mc{V}|})^{\T} \in\R^{|\mc{V}|\times |\mc{V}|},
\qquad\qquad
\Omega_{ik} := \begin{cases}
\w_{ik}, &\text{if}\; k\in\mc{N}_{i}, \\
0, &\text{otherwise.}
\end{cases}
\end{equation}
This formulation reveals in terms of \eqref{eq:S-flow-b} the `essential' part of the assignment flow equation, since \eqref{eq:S-flow-a} depends on \eqref{eq:S-flow-b}, but not vice versa. Furthermore, the data and weights show up only in the initial point and in the vector field on the right-hand side of \eqref{eq:S-flow-b}, respectively.

Henceforth, we solely focus on \eqref{eq:S-flow-b} rewritten for convenience as
\begin{equation}\label{eq:S-AF}
\dot S = R_{S}[\Omega S],\qquad S(0)=S_{0},
\end{equation}
where $S_{0}$ comprises the similarity vectors \eqref{eq:def-Si} evaluated at the barycenter $W=\eins_{\mc{W}_{c}}$.

\section{Quantum State Assignment Flows}\label{sec:DAF-all}

In this section, we generalize the assignment flow equations \eqref{eq:AF} and \eqref{eq:S-AF} to the product manifold $\mc{Q}_{c}$ of density matrices as state space. The resulting equations have a similar mathematical form. 
Their derivation requires 
\begin{itemize}
\item to determine the form of the Riemannian gradient of functions $f\colon\mc{D}_{c}\to\R$ with respect to the BKM-metric \eqref{eq:Bogoliubov}, the corresponding replicator operator and exponential mappings $\Exp$ and $\exp$ together with their differentials (Section \ref{sec:grad-R-density}), 
\item to define the single-vertex quantum state assignment flow (Section \ref{sec:single-vertex-DAF}),
\item to devise the general quantum state assignment flow equation for an arbitrary graph (Section \ref{sec:DAF})
\item and its alternative parametrization (Section \ref{sec:S-DAF}) which generalizes formulation \eqref{eq:S-AF} of the assignment flow accordingly.

\end{itemize}
A natural question is: What does `label' mean for a generalized assignment flow evolving on the product manifold $\mc{Q}_{c}$ of density matrices? For the single vertex quantum state assignment flow, i.e.~without interaction of these flows on a graph, it turns out that the pure state corresponding to the minimal eigenvalue of the initial density matrix is assigned to the given data point (Proposition \ref{prop:single-vertex-matrix-flow}). Coupling non-commuting density matrices over the graph through the novel quantum state assignment flow, therefore, generates an interesting complex dynamics as we illustrate in Section \ref{sec:Experiments}. It is shown in Section \ref{sec:recovering-af-cat-distr} that the restriction of the novel quantum state assignment flow to commuting density matrices recovers the original assignment flow for discrete labels. 

Throughout this section, we refer to definitions and notions  introduced in Section \ref{sec:intro-density-matrices}.

\subsection{Riemannian Gradient, Replicator Operator and Further Mappings}\label{sec:grad-R-density}

%
\begin{proposition}[\textbf{Riemannian gradient}]\label{prop:grad-Dc}
	Let $f\colon\mc{D}_{c}\to\R$ be a smooth function defined on the manifold \eqref{eq:def-mcDc}.
	and $\wt{f}\colon U\to\R$ a smooth extension of $f$ to an open neighborhood $U$ of $\mc{D}_{c}\subset \C^{c\times c}$ with $\wt{f}\vert_{\mc{D}_{c}} = f$.
	Then its Riemannian gradient with respect to the BKM-metric \eqref{eq:Bogoliubov} is given by
	\begin{equation}\label{eq:grad-Dc}
	\ggrad_{\rho}f = \mb{T}_{\rho}^{-1}[\partial \wt{f}] - \la \rho, \partial \wt{f} \ra\rho,
	\end{equation}
	where $\mb{T}_{\rho}^{-1}$ is given by \eqref{eq:def-mcT-inverse} and $\partial\wt{f}$ is the ordinary gradient with respect to the Euclidean structure of the ambient space $\textcolor{blue}{\C}^{c\times c}$.
	\end{proposition}
	\begin{proof}
		Appendix \ref{sec:appendix-quantum-state-assignment-flow}.
	\end{proof}

	\vspace{0.25cm}
	Comparing the result \eqref{eq:grad-Dc} with \eqref{eq:grad-R-simplex} motivates the following
	\begin{equation}\label{eq:def-Rrho}
	\Rep_{\rho}\colon\mc{H}_{c}\to\mc{H}_{c,0},\qquad
	\Rep_{\rho}[X] := \mb{T}_{\rho}^{-1}[X]-\la \rho,X\ra\rho,\qquad\rho\in\mc{D}_{c}
	\qquad\qquad(\textbf{replicator map})
	\end{equation}
	The following lemma shows that the properties \eqref{eq:Rp-Pi0} extend to \eqref{eq:def-Rrho}.
	\begin{lemma}[\textbf{properties of $\Rep_{\rho}$}]\label{lem:repl-proj-rho-commute}
	Let $\Pi_{c,0}$ denote the orthogonal projection \eqref{eq:def-Pic0}. Then the replicator map \eqref{eq:def-Rrho} satisfies
	\begin{equation}\label{eq:Pic0-Rrho-commute}
	\Pi_{c,0}\circ \Rep_{\rho}
	= \Rep_{\rho}\circ \Pi_{c,0} 
	= \Rep_{\rho},
	\quad\forall\rho\in\mc{D}_{c}.
	\end{equation}
	\end{lemma}
\begin{proof}
	Appendix \ref{sec:appendix-quantum-state-assignment-flow}.
\end{proof}
	
	\vspace{0.25cm}
Next, using the tangent space $\mc{H}_{c,0}$, we define a parametrization of the manifold $\mc{D}_{c}$ in terms of the mapping
	\begin{subequations}\label{eq:def-Gamma}
	\begin{align}\label{eq:def-Gamma-a}
	\Gamma\colon\mc{H}_{c,0}\to\mc{D}_{c},\qquad
	\Gamma(X) 
	&:= \frac{\expm(X)}{\tr\expm(X)}
	= \expm\big(X-\psi(X)I\big),\qquad\qquad (\textbf{$\Gamma$-map})
	\intertext{where}\label{eq:def-psi-Gamma}
	\psi(X) &:= \log\big(\tr\expm(X)\big).
	\end{align}
	\end{subequations}
	The following lemma and proposition show that the domain of $\Gamma$ extends to $\R^{c\times c}$.
	\begin{lemma}[\textbf{extension of $\Gamma$}]\label{lem:Gamma-Pi0}
	The extension to $\C^{c\times c}$ of the mapping $\Gamma$ defined by \eqref{eq:def-Gamma} is well-defined and given by
	\begin{equation}\label{eq:Gamma-Pi0}
	\Gamma\colon \C^{c\times c}\to\mc{D}_{c},\qquad
	\Gamma(Z) = \Gamma(\Pi_{c,0}[Z]).
	\end{equation}
	\end{lemma}
	\begin{proof}
		Appendix \ref{sec:appendix-quantum-state-assignment-flow}.
	\end{proof}

	\begin{proposition}[\textbf{inverse of $\Gamma$}]\label{prop:Gamma-invertible}
	The map $\Gamma$ defined by \eqref{eq:def-Gamma} is bijective with inverse
	\begin{equation}\label{eq:def-Gamma-1}
	\Gamma^{-1}\colon\mc{D}_{c}\to\mc{H}_{c,0},\qquad
	\Gamma^{-1}(\rho) = \Pi_{c,0}[\logm\rho].
	\end{equation}
	\end{proposition}
	\begin{proof}
		Appendix \ref{sec:appendix-quantum-state-assignment-flow}.
	\end{proof}
The following lemma provides the diffentials of the mappings $\Gamma$ and $\Gamma^{-1}$.
\begin{lemma}[\textbf{differentials $d\Gamma$, $d\Gamma^{-1}$}]\label{lem:dGamma-1}
Let $H, X\in\mc{H}_{c,0}$ with $\Gamma(H)=\rho$ and $Y\in T\mc{H}_{c,0}\cong\mc{H}_{c,0}$. Then
\begin{subequations}
\begin{align}\label{eq:dGamma-H}
d\Gamma(H)[Y] 
&= \mb{T}_{\rho}^{-1}\big[Y-\la\rho,Y\ra I\big],\qquad
\rho=\Gamma(H),
\\ \label{eq:dGamma-1}
d\Gamma^{-1}(\rho)[X] 
&= \Pi_{c,0}\circ\mb{T}_{\rho}[X].
\end{align}
\end{subequations}
\end{lemma}
\begin{proof}
	Appendix \ref{sec:appendix-quantum-state-assignment-flow}.
\end{proof}
We finally compute a closed-form expression of the e-geodesic, i.e.~the geodesic resp.~exponential map induced by the e-connection on the manifold $(\mc{D}_{c},g)$.
\begin{proposition}[\textbf{e-geodesics}]\label{prop:e-geodesic}
The e-geodesic emanating at $\rho\in\mc{D}_{c}$ in the direction $X\in\mc{H}_{c,0}$ and the corresponding exponential map are given by
\begin{subequations}\label{eq:e-geodesic}
\begin{align}
\gamma_{\rho,X}^{(e)}(t)
&:= \Exp_{\rho}^{(e)}(t X),\quad t\geq 0
&&(\textbf{e-geodesic})
\label{eq:e-geodesic-a} \\ \label{eq:def-Exp-rho-e}
\Exp_{\rho}^{(e)}(X)
&:= \Gamma\big(\Gamma^{-1}(\rho)+d\Gamma^{-1}(\rho)[X]\big)
&& (\textbf{exponential map})
\\ \label{eq:def-Exp-rho-e-b}
&= \Gamma\big(\Gamma^{-1}(\rho)+\Pi_{c,0}\circ\mb{T}_{\rho}[X]\big).
\end{align}
\end{subequations}
\end{proposition}
\begin{proof}
	Appendix \ref{sec:appendix-quantum-state-assignment-flow}.
\end{proof}

\begin{corollary}[\textbf{inverse exponential map}]\label{cor:Exp-1}
The inverse of the exponential mapping \eqref{eq:e-geodesic} is given by
\begin{equation}\label{eq:Exp-1}
\big(\Exp_{\rho}^{(e)}\big)^{-1}\colon\mc{D}_{c}\to\mc{H}_{c,0},\qquad
\big(\Exp_{\rho}^{(e)}\big)^{-1}(\mu)
= d\Gamma\big(\Gamma^{-1}(\rho)\big)\big[\Gamma^{-1}(\mu)-\Gamma^{-1}(\rho)\big].
\end{equation}
\end{corollary}
\begin{proof}
	Appendix \ref{sec:appendix-quantum-state-assignment-flow}.
\end{proof}
Analogous to \eqref{eq:def_exp_standard_af}, we define the mapping $\exp_{\rho}$, where both the subscript and the argument disambiguate the meaning of `$\exp$'.
\begin{lemma}[\textbf{$\exp$-map}]\label{lem:exp-rho}
The mapping defined using \eqref{eq:def-Exp-rho-e} and \eqref{eq:def-Rrho} by
\begin{subequations}\label{eq:def-exp-rho}
\begin{align}\label{eq:def-exp-rho-a}
\exp_{\rho}\colon\mc{H}_{0,c}\to\mc{D}_{c},\qquad
\exp_{\rho}(X) 
&:= \Exp_{\rho}^{(e)}\circ \Rep_{\rho}[X],\qquad
\rho\in\mc{D}_{c}
&& (\textbf{$\exp$-map})
\intertext{has the explicit form}\label{eq:def-exp-rho-b}
\exp_{\rho}(X) 
&= \Gamma \big( \Gamma^{-1}(\rho) + X \big).
\end{align}
\end{subequations}
\end{lemma}
\begin{proof}
	Appendix \ref{sec:appendix-quantum-state-assignment-flow}.
\end{proof}
The following lemma provides the explicit form of the differential of the mapping \eqref{eq:def-exp-rho-b} which resembles the corresponding formula \eqref{eq:dexp-p} of the assignment flow.
\begin{lemma}[\textbf{differential $d\exp_{\rho}$}]\label{lem:dexp-rho}
The differential of the mapping \eqref{eq:def-exp-rho} reads with $\rho\in\mc{D}_{c}$, $X\in\mc{H}_{c,0}$ and $Y\in T\mc{H}_{c,0}\cong\mc{H}_{c,0}$
\begin{equation}\label{eq:dexp-rho}
d\exp_{\rho}(X)[Y]
= \Rep_{\exp_{\rho}(X)}[Y].
\end{equation}
\end{lemma}
\begin{proof}
	Appendix \ref{sec:appendix-quantum-state-assignment-flow}.
\end{proof}
\begin{remark}[\textbf{comparing $\exp$-maps -- I}]\label{rem:similarity-to-AF-1}
Since \eqref{eq:dexp-rho} resembles \eqref{eq:dexp-p}, one may wonder about the connection of \eqref{eq:def-exp-rho-b} and \eqref{eq:def_exp_standard_af}. In view of \eqref{eq:def-Gamma-a}, we define
\begin{equation}
\gamma\colon T_{c,0}\to\mc{S}_{c},\qquad
\gamma(v) := \frac{e^{v}}{\la\eins,e^{v}\ra}
= \exp_{\eins_{\mc{S}_{c}}}(v)
\end{equation}
and compute with the expression for its inverse (cf.~\cite{savarino2021continuous})
\begin{subequations}
\begin{align}
\gamma^{-1}(p) 
&= \pi_{c,0}\log\frac{p}{\eins_{\mc{S}_{c}}}
= \pi_{c,0}(\log p-\log\eins_{\mc{S}_{c}})
= \pi_{c,0}\log p
\\
&\overset{\eqref{eq:def-Pi0-p}}{=} 
\log p - \la \eins_{\mc{S}_{c}},\log p\ra\eins_{c}
\end{align}
\end{subequations}
which resembles \eqref{eq:def-Gamma-1}. Moreover, in view of \eqref{eq:def-exp-rho-b}, the analogous expression using $\gamma$, instead of $\Gamma$, reads
\begin{subequations}
\begin{align}
\gamma\big(\gamma^{-1}(p)+v\big)
&= \frac{e^{\pi_{c,0}\log p + v}}{\la\eins,e^{\pi_{c,0}\log p + v}\ra}
= \frac{\la\eins_{\mc{S}_{c}},\log p\ra p\cdot e^{v}}{\la \la\eins_{\mc{S}_{c}},\log p\ra p,e^{v}\ra}
= \frac{p\cdot e^{v}}{\la p,e^{v}\ra}
\\
&= \exp_{p}(v).
\end{align}
\end{subequations}
\end{remark}

\begin{remark}[\textbf{comparing $\exp$-maps -- II}]\label{rem:similarity-to-AF-2}
Using the above definitions and relations, we check equation \eqref{eq:exp-log-e}, $\exp_{\rho}^{(\log)}(Y)
= \Exp_{\rho}^{(e)}(X)$, where the relation \eqref{eq:exp-log-e-Y} between $Y$ and $X$ can now be written in the form 
\begin{equation}
Y \overset{\eqref{eq:def-psi-Gamma}}{=} X - \psi\big(\logm(\rho)+\mb{T}_{\rho}[X]\big) \rho. 
\end{equation}
Direct computation yields
\begin{subequations}
\begin{align}
\exp_{\rho}^{(\log)}(Y)
&\overset{\eqref{eq:exp_log_euclidean}}{=} 
\expm (\logm(\rho) + \mathbb{T}_\rho[Y])
\\
&\overset{\eqref{eq:exp-log-e-Y}}{=}
\expm \Big(\logm(\rho) + \mathbb{T}_\rho[X] - \psi\big(\logm(\rho) + \mb{T}_{\rho}[X]\big) \overbrace{\mb{T}_{\rho}\circ \underbrace{\mb{T}_{\rho}^{-1}[I_{c}]}_{=\rho}}^{=I_{c}}\Big)
\\
&\overset{\substack{\eqref{eq:def-Gamma} \\\eqref{eq:Gamma-Pi0}}}{=} 
\Gamma\big(\Pi_{c,0}[\logm(\rho)] + \Pi_{c,0}\circ\mb{T}_{\rho}[X]\big)
= \Gamma\big(\Gamma^{-1}(\rho)+\Pi_{c,0}\circ\mb{T}_{\rho}[X]\big)
\\
&= \Exp_{\rho}^{(e)}(X).
\end{align}
\end{subequations}
\end{remark}

\subsection{Single-Vertex Density Matrix Assignment Flow}\label{sec:single-vertex-DAF}

We generalize the single vertex assignment flow equation \eqref{eq:single-vertex-AF} to the manifold $(\mc{D}_{c}, g_{\rho})$ given by \eqref{eq:def-mcDc} with the BKM metric \eqref{eq:Bogoliubov}.

Defining in view of \eqref{eq:def-LpD} the \textit{likelihood matrix}
\begin{equation}\label{eq:def-L-rho-D}
L_{\rho}\colon\mc{H}_{c}\to\mc{D}_{c},\qquad
L_{\rho}(D) := \exp_{\rho}(-\Pi_{c,0}[D]),\qquad
\rho\in\mc{D}_{c},
\end{equation}
the corresponding \textit{single vertex quantum state assignment flow (SQSAF)} equation reads
\begin{subequations}\label{eq:single-vertex-matrix-flow}
\begin{align}
\dot\rho &= \Rep_{\rho}[L_{\rho}(D)]
&& (\textbf{SQSAF})
\\ \label{eq:single-vertex-matrix-flow-b}
&\overset{\eqref{eq:def-Rrho}}{=} 
\mb{T}_{\rho}^{-1}[L_{\rho}(D)] - \la\rho,L_{\rho}(D)\ra\rho,\qquad
\rho(0)=\eins_{\mc{D}_{c}}=\Diag(\eins_{\mc{S}_{c}}).
\end{align}
\end{subequations}
Proposition \ref{prop:single-vertex-matrix-flow} below specifies its properties after a preparatory Lemma.
\begin{lemma}\label{lem:rho-LD-diagonalized}
Assume 
\begin{equation}\label{eq:lem-rho-LD-decomposition}
D=Q\Lambda_{D}Q^{\T}\in\mc{H}_{c} 
\qquad\text{and}\qquad 
\rho=Q\Lambda_{\rho}Q^{\T}\in\mc{D}_{c}
\end{equation}
can be simultaneously diagonalized with $Q\in\LG{O}(c)$, $\Lambda_{D}=\Diag(\lambda_{D})$, $\Lambda_{\rho}=\Diag(\lambda_{\rho})$ and $\lambda_{\rho}\in\mc{S}_{c}$ since $\tr\rho=1$. Then
\begin{equation}
L_{\rho}(D) = Q \Diag\big(\exp_{\lambda_{\rho}}(-\lambda_{D})\big) Q^{\T}.
\end{equation}
\end{lemma}
\begin{proof}
	Appendix \ref{sec:appendix-quantum-state-assignment-flow}.
\end{proof}
\begin{proposition}[\textbf{SQSAF limit}]\label{prop:single-vertex-matrix-flow}
Let $D = Q\Lambda_{D}Q^{\T}$ be the spectral decomposition of $D$ with eigenvalues $\lambda_{1}\geq\dotsb\geq\lambda_{c}$ and orthonormal eigenvectors $Q = (q_{1},\dotsc,q_{c})$. Assume the minimal eigenvalue $\lambda_{c}$ is unique. Then the solution $\rho(t)$ to \eqref{eq:single-vertex-matrix-flow} satisfies
\begin{equation}
\lim_{t\to\infty}\rho(t) = \Pi_{q_{c}} := q_{c}q_{c}^{\T}.
\end{equation}
\end{proposition}
\begin{proof}
	Appendix \ref{sec:appendix-quantum-state-assignment-flow}.
\end{proof}

\subsection{Quantum State Assignment Flow}\label{sec:DAF}
This section describes our main result, the definition of a novel flow of coupled density matrices in terms of a parametrized interaction of single vertex flows of the form \eqref{eq:single-vertex-matrix-flow} on a given graph $\mc{G} = (\mc{V},\mc{E},\w)$.

We assume the weight function $\w\colon\mc{E}\to\R_{+}$ to be nonnegative with $\w_{ij}=0$ if $ij\not\in\mc{E}$ and 
\begin{equation}\label{eq:omega-sum-1}
\sum_{k\in\mc{N}_{i}}\w_{ik}=1,
\end{equation}
where we adopt the notation \eqref{eq:def-mcN-i} for neighborhoods $\mc{N}_{i},\;i\in\mc{V}$. Analogous to \eqref{eq:def-mcW}, we define the product manifold
\begin{equation}\label{eq:def-mcQ}
\rho \in\mc{Q}_{c} := \underbrace{\mc{D}_{c}\times\dotsb\times\mc{D}_{c}}_{|\mc{V}|\;\text{factors}}
\end{equation}
with $\mc{D}_{c}$ given by \eqref{eq:def-mcDc}. The corresponding factors of $\rho$ are denoted by 
\begin{equation}
\rho = (\rho_{i})_{i\in[c]},\quad
\rho_{i}\in\mc{D}_{c},\quad i\in\mc{V}.
\end{equation}
$\mc{Q}_{c}$ becomes a Riemannian manifold when equipped with the metric
\begin{equation}\label{eq:KM-metric-AMF}
	g_\rho (X,Y) := \sum_{i \in \mc{V}} g_{\rho_i}(X_i,Y_i), \qquad X,Y \in T\mc{Q}_{c}:=\mc{H}_{c,0} \times \dotsb \times \mc{H}_{c,0},
\end{equation}
with $g_{\rho_{i}}$ given by \eqref{eq:Bogoliubov} for each $i\in\mc{V}$. We set
\begin{equation}
\eins_{\mc{Q}_{c}} := (\eins_{\mc{D}_{c}})_{i\in\mc{V}} \in\mc{Q}_{c},
\end{equation}
with $\eins_{\mc{D}_{c}}$ given by \eqref{eq:single-vertex-matrix-flow-b}. Our next step is to define a \textit{similarity mapping} analogous to \eqref{eq:def-Si},
\begin{equation}\label{eq:def-Si-rho}
S\colon\mc{V}\times\mc{Q}_{c},\qquad
S_{i}(\rho) := \Exp^{(e)}_{\rho_{i}}\Big(\sum_{k\in\mc{N}_{i}}\w_{ik}\big(\Exp^{(e)}_{\rho_{i}}\big)^{-1}\big(L_{\rho_{k}}(D_{k})\big)\Big),
\end{equation}
based on the mappings \eqref{eq:def-Exp-rho-e} and \eqref{eq:def-L-rho-D}. 
Thanks to using the exponential map of the e-connection, the matrix $S_{i}(\rho)$ can be rewritten and computed in a simpler, more explicit form.
\begin{lemma}[\textbf{similarity map}]\label{lem:Si-rho-explicit}
Equation \eqref{eq:def-Si-rho} is equivalent to
\begin{equation}\label{eq:def-Si-rho-simple}
S_{i}(\rho)=\Gamma\Big(\sum_{k\in\mc{N}_{i}}\w_{ik}(\logm\rho_{k}-D_{k})\Big).
\end{equation}
\end{lemma}
\begin{proof}
	Appendix~\ref{sec:appendix-quantum-state-assignment-flow}.
\end{proof}
Expression \eqref{eq:def-Si-rho}, which defines the similarity map, looks like a single iterative step for computing the Riemannian center of mass of the likelihood matrices $\{L_{\rho_{k}}(D_{k})\colon k\in\mc{N}_{i}\}$ if(!) the exponential map of the Riemannian (Levi Civita) connection were used. Instead, when using the exponential map $\Exp^{(e)}$, $S_{i}(\rho)$ may be interpreted as carrying out a single iterative step for the corresponding \textit{geometric mean} on the manifold $\mc{D}_{c}$.

A natural idea therefore is to define the similarity map to be this geometric mean, rather than just by a single iterative step. Surprisingly, analogous to the similarity map \eqref{eq:def-Si} for categorial distributions (cf.~\cite{Schnorr:2019aa}), both definitions are \textit{identical}, as shown next.
\begin{proposition}[\textbf{geometric mean property}]\label{prop:Si-riemannian_center}
Assume that $\ol{\rho}\in\mc{D}_{c}$ solves the equation
\begin{equation}
0 = \sum_{k\in\mc{N}_{i}}\w_{ik}\big(\Exp^{(e)}_{\ol{\rho}}\big)^{-1}\big(L_{\rho_{k}}(D_{k})\big)
\end{equation}
which corresponds to the optimality condition for Riemannian centers of mass \cite[Lemma 6.9.4]{Jost:2017aa}, except for using a different exponential map. Then
\begin{equation}
\ol{\rho}=S_{i}(\rho)
\end{equation}
with the right-hand side given by \eqref{eq:def-Si-rho}.
\end{proposition}
\begin{proof}
	Appendix~\ref{sec:appendix-quantum-state-assignment-flow}.
\end{proof}
We are now in the position to define the \textit{quantum state assignment flow} along the lines of the original assignment flow \eqref{eq:AF},
\begin{equation}\label{eq:density-AF}
\dot\rho = \Rep_{\rho}[S(\rho)],\qquad \rho(0)=\eins_{\mc{Q}_{c}},\qquad\qquad
(\textbf{QSAF})
\end{equation}
where both the replicator map $\Rep_{\rho}$ and the similarity map $S(\cdot)$ apply factorwise,
\begin{subequations}\label{eq:def-R-S-rho}
\begin{align}
S(\rho)_{i} &= S_{i}(\rho), \\ \label{eq:def-R-rho}
\Rep_{\rho}[S(\rho)]_{i} &= \Rep_{\rho_{i}}[S_{i}(\rho)],\quad
i\in\mc{V}
\end{align}
\end{subequations}
with the mappings $S_{i}$ given by \eqref{eq:def-Si-rho-simple} and $\Rep_{\rho_{i}}$ by \eqref{eq:def-Rrho}.

\subsection{Reparametrization, Riemannian Gradient Flow}\label{sec:S-DAF}
The reparametrization of the assignment flow \eqref{eq:S-flow} for categorial distributions described in Section \ref{sec:standardAF_S-flow} has proven to be useful for characterizing and analyzing assignment flows. Under suitable conditions on the parameter matrix $\Omega$, the flow performs a Riemannian descent flow with respect to a non-convex potential \cite[Prop.~3.9]{savarino2021continuous} and has convenient stability and convergence properties \cite{zern2021assignment}. 

In this section, we derive a similar reparametrization of the quantum state assignment flow \eqref{eq:density-AF}.
\begin{proposition}[\textbf{reparametrization}]\label{prop:reparametrization_QSAF}
Define the linear mapping
\begin{equation}\label{eq:def-Omega-rho}
\Omega\colon\mc{Q}_{c}\to\mc{Q}_{c},\qquad
\Omega[\rho]_{i} := \sum_{k\in\mc{N}_{i}}\w_{ik}\rho_{k}.
\end{equation}
Then the density matrix assignment flow equation \eqref{eq:density-AF} is equivalent to the system
\begin{subequations}\label{eq:S-flow-density}
\begin{align}
\dot\rho &= \Rep_{\rho}[\mu],\qquad\quad\;\;\,\rho(0)=\eins_{\mc{Q}_{c}},
\label{eq:S-flow-density-a} \\ \label{eq:S-flow-density-b}
\dot\mu &= \Rep_{\mu}\big[\Omega[\mu]\big],\qquad
\mu(0)=S(\eins_{\mc{Q}_{c}}).
\end{align}
\end{subequations}
\end{proposition}
\begin{proof}
	Appendix~\ref{sec:appendix-quantum-state-assignment-flow}.
\end{proof}
For the following, we adopt the \textit{symmetry assumption}
\begin{subequations}\label{eq:ass-symmetry}
\begin{gather}
\w_{ij}=\w_{ji},\qquad\forall i,j\in\mc{V}
\label{eq:ass-symmetry-om} \\ \label{eq:ass-symmetry-mcN}
j\in\mc{N}_{i}\quad\gdw\quad
i\in\mc{N}_{j},\qquad i,j\in\mc{V}.
\end{gather}
\end{subequations}
As a consequence, the mapping \eqref{eq:def-Omega-rho} is self-adjoint,
\begin{subequations}
\begin{align}
\la \mu,\Omega[\rho]\ra
&= \sum_{i\in\mc{V}} \la \mu_{i},\Omega[\rho]_{i}\ra
= \sum_{i\in\mc{V}} \sum_{k\in\mc{N}_{i}}\w_{ik}\la \mu_{i},\rho_{k}\ra
= \sum_{i\in\mc{V}} \sum_{k\in\mc{N}_{i}}\w_{ki}\la \mu_{i},\rho_{k}\ra
\\
&= \sum_{k\in\mc{V}}\sum_{i\in\mc{N}_{k}}\w_{ki}\la \mu_{i},\rho_{k}\ra
= \sum_{k\in\mc{N}_{i}}\la\Omega[\mu]_{k},\rho_{k}\ra
= \la\Omega[\mu],\rho\ra.
\end{align}
\end{subequations}
\begin{proposition}[\textbf{Riemannian gradient QSAF flow}]\label{prop:potentail-QSAF}
Suppose the mapping $\Omega[\cdot]$ given by \eqref{eq:def-Omega-rho} is self-adjoint with respect to the canonical matrix inner product. Then the solution $\mu(t)$ to \eqref{eq:S-flow-density-b} also solves
\begin{subequations}
\begin{align}\label{eq:dot-mu-grad}
\dot\mu 
&= -\ggrad_{\mu} J(\mu)
\qquad\text{with}\qquad
\big(\ggrad_{\mu} J(\mu)\big)_{i} 
= \ggrad_{\mu_{i}} J(\mu)
\intertext{with respect to the potential}\label{eq:dot-mu-potential}
J(\mu) &:= -\frac{1}{2}\la\mu,\Omega[\mu]\ra.
\end{align}
\end{subequations}
\end{proposition}
\begin{proof}
	Appendix~\ref{sec:appendix-quantum-state-assignment-flow}.
\end{proof}
We conclude this section by rewriting the potential in a more explicit, informative form.
\begin{proposition}[\textbf{nonconvex potential}]\label{prop:potentail-QSAF-laplacian}
Define
\begin{equation}
L_{\mc{G}}\colon\mc{Q}_{c}\to\mc{Q}_{c},\qquad
L_{\mc{G}} := \mrm{id}-\Omega
\end{equation}
with $\Omega$ given by \eqref{eq:def-Omega-rho}. 
Then the potential \eqref{eq:dot-mu-potential} can be rewritten as
\begin{subequations}
\begin{align}
J(\mu) &= \frac{1}{2}\big(\la\mu,L_{\mc{G}}[\mu]\ra - \|\mu\|^{2}\big)
\\
&= \frac{1}{4}\sum_{i\in\mc{V}}\sum_{j\in\mc{N}_{i}}\w_{ij}\|\mu_{i}-\mu_{j}\|^{2}-\frac{1}{2}\|\mu\|^{2}.
\end{align}
\end{subequations}
\end{proposition}
\begin{proof}
	Appendix~\ref{sec:appendix-quantum-state-assignment-flow}.
\end{proof}

\subsection{Recovering the Assignment Flow for Categorial Distributions}\label{sec:recovering-af-cat-distr}

In the following we show how the assignment flow \eqref{eq:S-AF} for categorial distributions arises as special case of the quantum state assignment flow, under suitable conditions as detailed below.
\begin{definition}[\textbf{commutative submanifold}]\label{def:mcD-Pi}
Let 
\begin{equation}\label{eq:def-Pi-projectors}
\Pi=\{\pi_{i}\colon i\in[l]\},\qquad
l \leq c
\end{equation}
denote a set of operators which orthogonally project onto disjoint subspaces of $\mathbb{C}^{c}$,
\begin{subequations}
\begin{align}
\pi_{i}^{2} &= \pi_{i},\quad\forall i\in[l], \\
\pi_{i}\pi_{j} &= 0,\quad\forall i,j\in[l],\; i\neq j,
\end{align}
\end{subequations}
and which are complete in the sense that
\begin{equation}
\sum_{i\in[l]}\pi_{i} = I_{c}.
\end{equation}
Given a family $\Pi$ of operators, we define by
\begin{equation}\label{eq:def-mcD-Pi}
\mc{D}_{\Pi} := \bigg\{\sum_{i\in[l]}\frac{p_{i}}{\tr\pi_{i}}\pi_{i}\colon p\in\mc{S}_{l}\bigg\} \subset \mc{D}_{c}
\end{equation}
the \textit{submanifold of commuting Hermitian matrices} which can be diagonalized simultaneously.
\end{definition}
A typical example for a family \eqref{eq:def-Pi-projectors} is 
\begin{equation}\label{eq:Pi-ONB}
\Pi_{\mc{U}} = \{\pi_{i}=u_{i}u_{i}^{\ast}\colon i\in[c]\},
\end{equation}
where $\mc{U} = \{u_{1},\dotsc,u_{c}\}$ is an orthonormal basis of $\mathbb{C}^{c}$. The following lemma elaborates the bijection $D_{\Pi}\leftrightarrow\mc{S}_{l}$. 
\begin{lemma}[\textbf{properties of $\mc{D}_{\Pi}$}]
\label{lem:prop_of_d_pi}
Let $\mc{D}_{\Pi}\subset\mc{D}_{c}$ be given by \eqref{eq:def-mcD-Pi} and denote the corresponding inclusion map by $\iota\colon \mc{D}_{\Pi}\embedded\mc{D}_{c}$. Then
\begin{enumerate}[(a)]
\item the submanifold $(\mc{D}_{\Pi},\iota^{\ast}g_{\sst{\mrm{BKM}}})$ with the induced BKM metric is isometric to $(\mc{S}_{l},g_{\sst{\mrm{FR}}})$;
%
\item if $\mu\in\mc{D}_{\Pi}$, then the tangent subspace $T_{\mu}\mc{D}_{\Pi}$ is contained in the subspace $T_{\mu}^{c}\mc{D}_{c} \subseteq T_\mu \mathcal D_c$ defined by \eqref{eq:decomp_H_c0-b}.
\item Let $\mc{U} = \{ u_{1},\dots,u_{c} \}$ denote an orthonormal basis of $\mathbb{C}^{c}$ such that for every $\pi_{i} \in \Pi,\;i\in[l]$, there are $u_{i_{1}},\dots,u_{i_{k}} \in \mc{U}$ that form a basis of $~\mathrm{range}(\pi_{i})$. Then there is an inclusion of commutative subsets $\mathcal{D}_{\Pi} \hookrightarrow \mathcal{D}_{\Pi_{\mc{U}}}$ that corresponds to an inclusion $\mathcal{S}_{l} \hookrightarrow \mathcal{S}_{c}$.
\end{enumerate}
\end{lemma}
\begin{proof}
Appendix~\ref{sec:appendix-quantum-state-assignment-flow}. 
\end{proof}
Now we establish that a restriction of the QSAF equation \eqref{eq:S-flow-density-b} to the commutative product submanifold can be expressed in terms of the AF equation \eqref{eq:S-AF}. Analogous to the definition \eqref{eq:def-mcQ} of the product manifold $\mc{Q}_{c}$, we set 
\begin{equation} 
\mc{D}_{\Pi,c}
= \underbrace{\mc{D}_{\Pi}\times\dotsb\times\mc{D}_{\Pi}}_{|\mc{V}|\;\text{factors}}.
\end{equation}
If $\Pi$ is given by an orthonormal basis as in \eqref{eq:Pi-ONB}, we define the unitary matrices
\begin{subequations}
\begin{align}
U &= (u_{1},\dotsc,u_{c}) \in\LG{Un}(c), \\
U_{c} &= \underbrace{\Diag(U,\dotsc,U)}_{|\mc{V}|\;\text{block-diagonal entries}}.
\end{align}
\end{subequations}
\begin{proposition}[\textbf{invariance of $\mc{D}_{\Pi,c}$}]\label{prop:commutativity-preservation}
Let $\Pi$ and $D_{\Pi}$ be given according to Definition \ref{def:mcD-Pi}. Then the following holds.
\begin{enumerate}[(i)]
%
\item 
If $\mu\in\mc{D}_{\Pi,c}\subset\mc{Q}_{c}$, then $\Rep\big[\Omega[\mu]\big]\in T_{\mu}\mc{D}_{\Pi,c}\subseteq T_{\mu}\mc{Q}_{c}$.
\item
If $\Pi_{\mc{U}}$ has the form \eqref{eq:Pi-ONB}, then
\begin{equation}
\Rep\big[\Omega[\mu]\big]
= U_{c}\Diag\big[R_{S}[\Omega S]\big] U_{c}^{\ast},
\end{equation}
where $S\in\mc{W}_{c}$ is determined by $\mu_{i}=U\Diag(S_{i}) U^{\ast},\; i\in\mc{V}$.
\end{enumerate}
In particular, the submanifold $\mathcal D_{\Pi,c}$ is preserved by the quantum state assignment flow.
\end{proposition}
\begin{proof}
	Appendix~\ref{sec:appendix-quantum-state-assignment-flow}.
\end{proof}
It remains to check that under suitable conditions on the data matrices $D_{i},\;i\in\mc{V}$ which define the initial point of \eqref{eq:S-flow-density-b} by the similarity mapping (Lemma \ref{lem:Si-rho-explicit}), the quantum state assignment flow becomes the ordinary assignment flow.
\begin{corollary}[\textbf{recovery of the AF by restriction}]\label{cor:af_leq_qaf}
In the situation of Proposition \ref{prop:commutativity-preservation}, assume that all data matrices $D_{i},\;i\in\mc{V}$ become diagonal in the same basis $\mc{U}$, i.e.
\begin{equation}
D_{i} = U\Diag(\lambda_{i}) U^{\ast},\quad\lambda_{i}\in\R^{c},\quad i\in\mc{V}.
\end{equation}
Then the solution of the QSAF
\begin{equation}
\dot\mu = \Rep_{\mu}\big[\Omega[\mu]\big],\quad
\mu(0)=S(\eins_{\mc{Q}_{c}})
\end{equation}
is given by 
\begin{equation}
\mu_{i}(t) = U\Diag\big(S_{i}(t)\big) U^{\ast},\quad i\in\mc{V},
\end{equation}
where $S(t)$ satisfies the ordinary AF equation
\begin{equation}
\dot S = R_{S}[\Omega S],\quad S(0)=S(\eins_{\mc{W}_{c}}),
\end{equation}
and the initial point is determined by the similarity map \eqref{eq:def-Si} evaluated at the barycenter $W=\eins_{\mc{W}_{c}}$ with the vectors $\lambda_{i},\,i\in\mc{V}$ as data points.
\end{corollary}
\begin{proof}
	Appendix~\ref{sec:appendix-quantum-state-assignment-flow}.
\end{proof}


\section{Experiments and Discussion}\label{sec:Experiments}

We report in this section few academical experiments in order to illustrate the novel approach. In comparison to the original formulation, it enables a continuous assignment without the need to specify explicitly prototypical labels beforehand. The experiments highlight the following properties of the novel approach which extend the expressivity of the original assignment flow approach:
\begin{itemize}
\item geometric \textit{adaptive} feature vector averaging even when \textit{uniform} weights are used (Section \ref{sec:Bloch-sphere-averaging});
\item structure-preserving feature \textit{patch} smoothing \textit{without} accessing data at individual \textit{pixels}  (Section \ref{sec:patch-smoothing});
\item seamless incorporation of feature \textit{encoding} using finite \textit{frames} (Section \ref{sec:patch-smoothing}).
\end{itemize}
In Section \ref{sec:Conclusion}, we indicate the potential for representing spatial feature \textit{context} via entanglement. 
Working out more thoroughly the potential for various applications is beyond the scope of this paper, however.

\subsection{Geometric Integration}

In this section, we focus on the geometric integration of the reparameterized flow described by Equation~\eqref{eq:S-flow-density-b}. For a reasonable choice of a single stepsize parameter, the scheme is accurate, stable and amenable to highly parallel implementations.

We utilize that the e-geodesic from Proposition~\ref{prop:e-geodesic} constitutes a retraction \cite[Def.~4.1.1 and Prop.~5.4.1]{Absil:2008aa} onto the state manifold $\mc{Q}_{c}$.

Consequently, the iterative step for updating $\mu_t \in \mathcal{Q}_{c},\; t\in\N_{0}$ and stepsize $\epsilon > 0$ is given by

\begin{subequations}
\begin{align}\label{eq:update-step1}
    (\mu_{t+1})_i   &= \Big(\Exp_{\mu_t}^{(e)}\big(\epsilon\,  \mathfrak{R}_{\mu_t} \big[ \Omega[\mu_t] \big]\big)\Big)_i = \big(\Exp_{(\mu_t)_{i}}^{(e)} \circ \mathfrak{R}_{(\mu_t)_i} \big[ \epsilon (\Omega [\mu_t])_i \big]\big) \\
                &\overset{\eqref{eq:def-exp-rho-a}}{=} \exp_{(\mu_t)_i} \big(\epsilon (\Omega[\mu_t])_i\big),\quad\forall i\in\mc{V}.
\end{align}
\end{subequations}
for all $i \in \mc{V}$. 
Using \eqref{eq:def-exp-rho-b} and assuming  
\begin{equation}\label{eq:mu-by-Gamma}
(\mu_t)_i = \Gamma\big((A_t)_i\big),\quad \forall i \in \mc{V},
\end{equation} 
with $A_t \in \mc{T}_{c}$, we obtain
\begin{subequations}\label{eq:dot_Gamma_A}
\begin{align}
     \Gamma\big((A_{t+1})_i \big)    
     &:= \exp_{(\mu_t)_i} \big(\epsilon (\Omega[\mu_t])_i\big) \\
                            &= \Gamma \big( \Gamma^{-1}((\mu_t)_i) + \epsilon (\Omega[\mu_t])_i \big)\\
                            &= \Gamma \big( \Gamma^{-1} \circ \Gamma ((A_t)_i) + \epsilon (\Omega[\mu_t])_i \big)\\
                            &= \Gamma \big( (A_t)_i + \epsilon  (\Omega[\mu_t])_i \big),\quad i\in\mc{V},
\end{align}
\end{subequations}
and we conclude in view of \eqref{eq:Gamma-Pi0} and \eqref{eq:mu-by-Gamma}
\begin{equation}\label{eq:QSAF-S-update}
    A_{t+1} = A_t + \epsilon \Pi_{c,0} \Omega[\Gamma(A_t)].
\end{equation}

\begin{remark}
We note that the numerical evaluation of the replicator operator \eqref{eq:def-Rrho} is not required. This makes the geometric integration scheme, summarized by Algorithm~\ref{alg:QSAF-S-geometric-integration}, quite efficient.
\end{remark}

\begin{figure}[ht]
\centerline{
\includegraphics[width=0.24\textwidth]{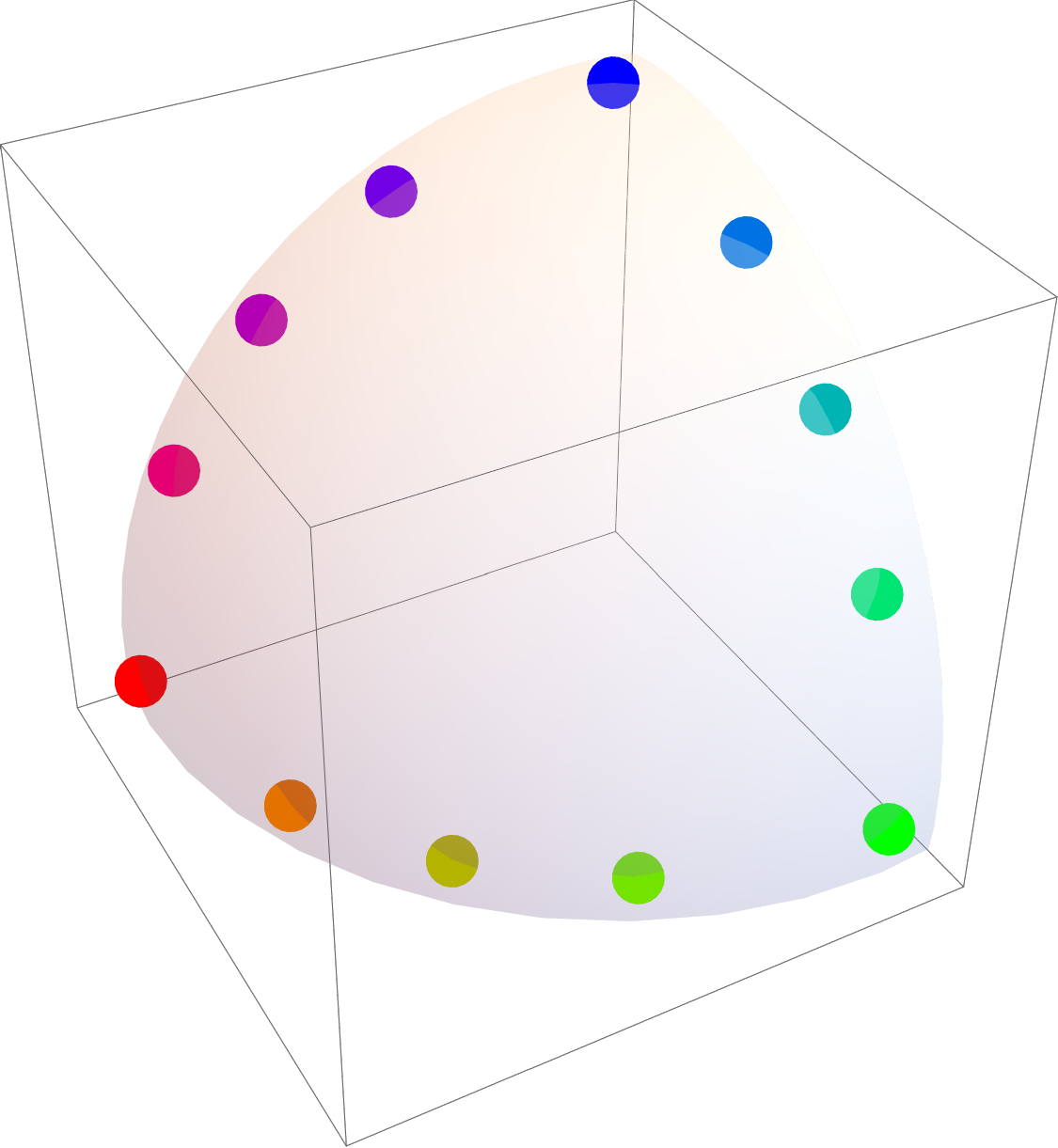}\hfill
\includegraphics[width=0.24\textwidth]{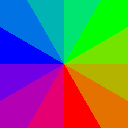}\hfill
\includegraphics[width=0.24\textwidth]{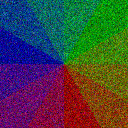}\hfill
\includegraphics[width=0.24\textwidth]{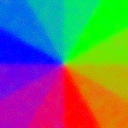}\hfill
}
\centerline{
\parbox{0.24\textwidth}{\centering (a)}\hfill
\parbox{0.24\textwidth}{\centering (b)}\hfill
\parbox{0.24\textwidth}{\centering (c)}\hfill
\parbox{0.24\textwidth}{\centering (d)}
} 
\centerline{
\includegraphics[width=0.24\textwidth]{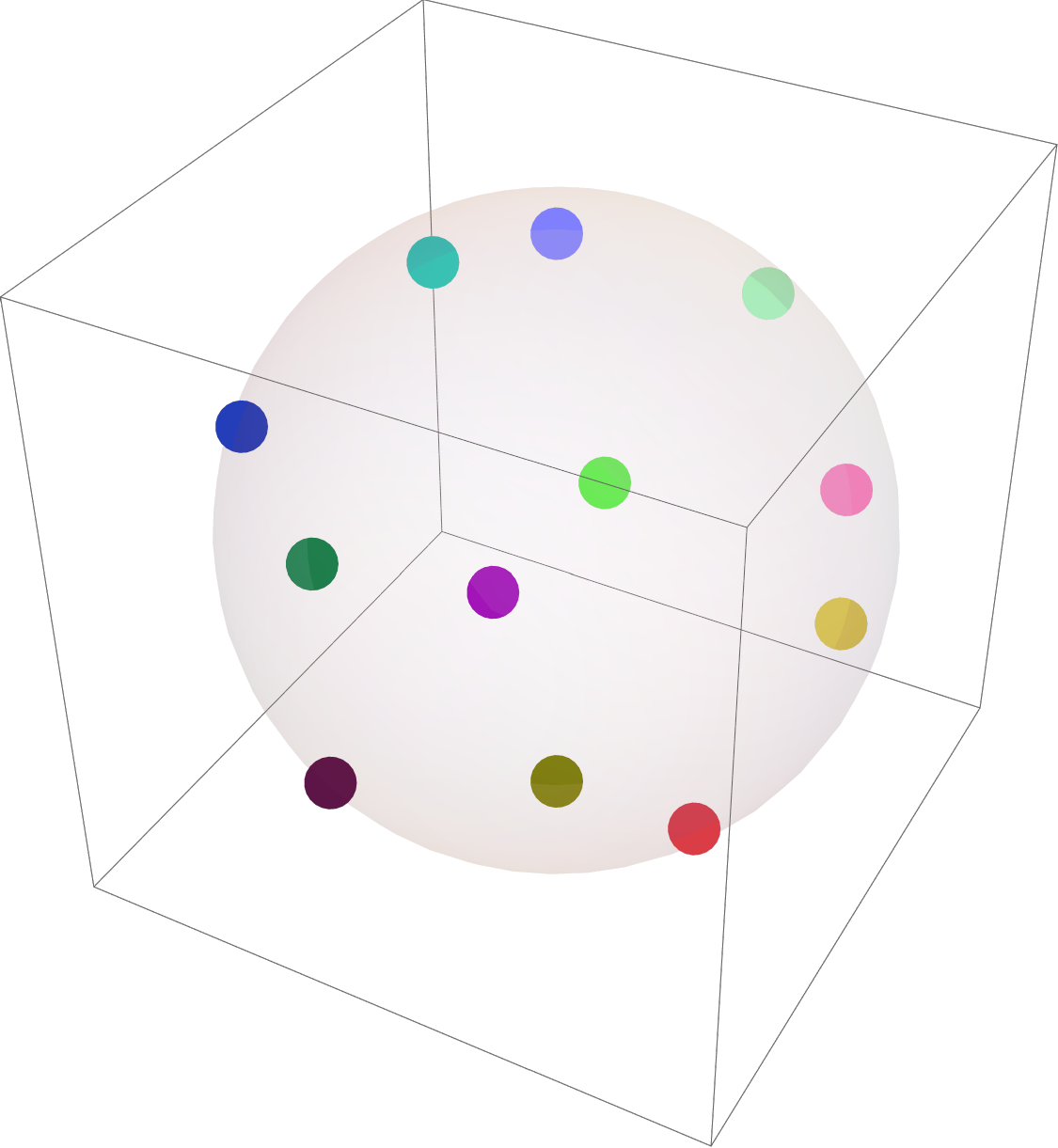}\hfill
\includegraphics[width=0.24\textwidth]{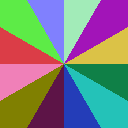}\hfill
\includegraphics[width=0.24\textwidth]{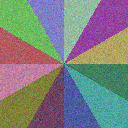}\hfill
\includegraphics[width=0.24\textwidth]{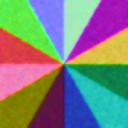}\hfill
}
\centerline{
\parbox{0.24\textwidth}{\centering (e)}\hfill
\parbox{0.24\textwidth}{\centering (f)}\hfill
\parbox{0.24\textwidth}{\centering (g)}\hfill
\parbox{0.24\textwidth}{\centering (h)}
} 
\caption{
\textbf{(a)} A range of RGB unit color vectors in the positive orthant. \textbf{(b)} An image with data according to (a). \textbf{(c)} A noisy version of (b) constituting the initial points $\rho_{i}(0),\;i\in\mc{V}$ of the QSAF. \textbf{(d)} The labels (pure states) generated by integrating the quantum state assignment flow using uniform weights. \textbf{(e)} The vectors depicted by (a) are replaced by the unit vectors corresponding to the vertices of the icosahedron, centered at $0$. \textbf{(f)-(h)} Analogous to (b)-(d), based on (e) instead of (a) and using the same noise level in (g). The colors in (f)-(h) merely visualize the bloch vectors by RGB vectors that result from translating the sphere of (e) to the center $\frac{1}{2}(1,1,1)^{\T}$ of the RGB cube and scaling it by $\frac{1}{2}$. We refer to the text for a discussion. 
}
\label{fig:12-colors}
\end{figure}

\begin{algorithm}[H]
    \textbf{Initialization} \\
        Determine an initial $A_{0} \in \mc{T}_{c;0}$ and compute $\mu_{0}$ by $(\mu_{0})_i = \Gamma((A_{0})_i) \in \mc{Q}_{c},\;\forall i \in \mc{V}$\\
    \While{
			\textbf{not converged}
    }{
$
(A_{t+1})_i = (A_{t})_i + \epsilon \Pi_{c,0} (\Omega[\mu_{t}])_i \quad \forall i \in \mc{V}
$
\\
$
(\mu_{t+1})_i = \Gamma\big((A_{t+1})_i\big), \quad \forall i \in \mc{V}
$.
        }
    \caption{Geometric Integration Scheme}
    \label{alg:QSAF-S-geometric-integration}
\end{algorithm}
We list few further implementation details.
\begin{itemize}
\item A reasonable convergence criterion which measures how close the states are to a rank one matrix, is $|\tr(\mu_t)_{i}-\tr(\mu_t^2)_{i}|\leq\veps,\;\forall i\in\mc{V}$.
\item A resonable range for the stepsize parameter is $\epsilon \leq 0.1$.
\item In order to remove spurious non-Hermitian numerical rounding errors, we replace each matrix $(\Omega[\mu_{t}]_{i})$ by $\frac{1}{2}\big((\Omega[\mu_{t}])_{i} + (\Omega[\mu_{t}])_{i}^{\ast}\big)$.
\item The constraint $\tr\rho=1$ of \eqref{eq:def-mcDc} can be replaced by $\tr\rho=\tau$ with any constant $\tau>1$. This ensures for larger matrix dimensions $c$ that the entries of $\rho$ vary in a reasonable numerical range and the stability of the iterative updates.
\end{itemize}

Up to moderate matrix dimensions, say $c \leq 100$, the matrix exponential in \eqref{eq:def-Gamma-a} can be computed using any of the basic established algorithms \cite[Ch.~10]{Higham:2008aa} or available solvers. In addition, depending on size of the neighborhood $\mc{N}_{i}$ induced by the weighted adjacency relation of the underlying graph in \eqref{eq:omega-sum-1}, Algorithm \ref{alg:QSAF-S-geometric-integration} can be implemented in a fine-grained parallel fashion.

\begin{figure}
\centerline{
\includegraphics[width=0.2\textwidth]{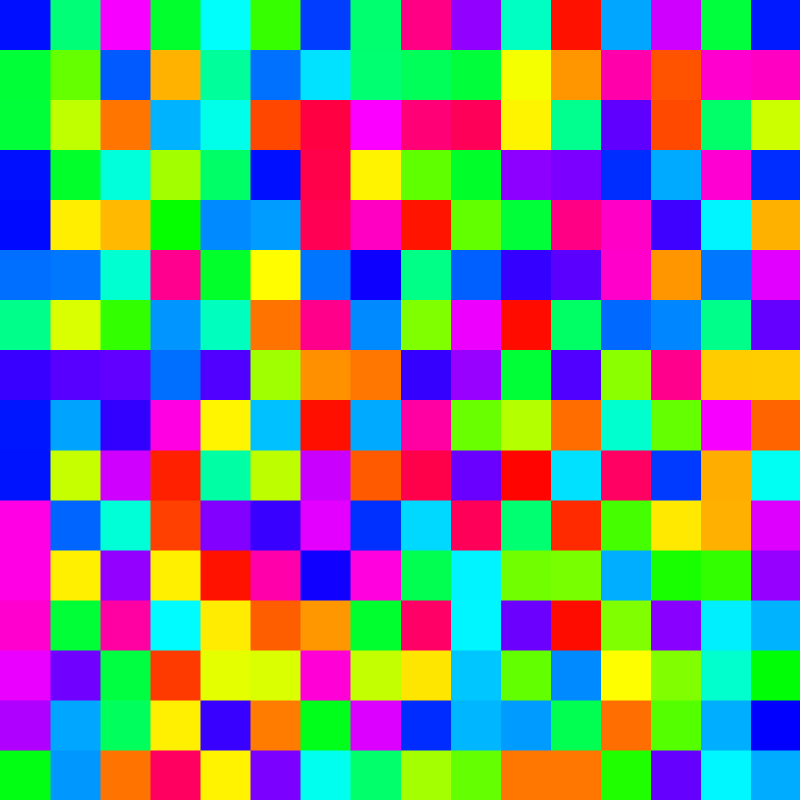}
\includegraphics[width=0.2\textwidth]{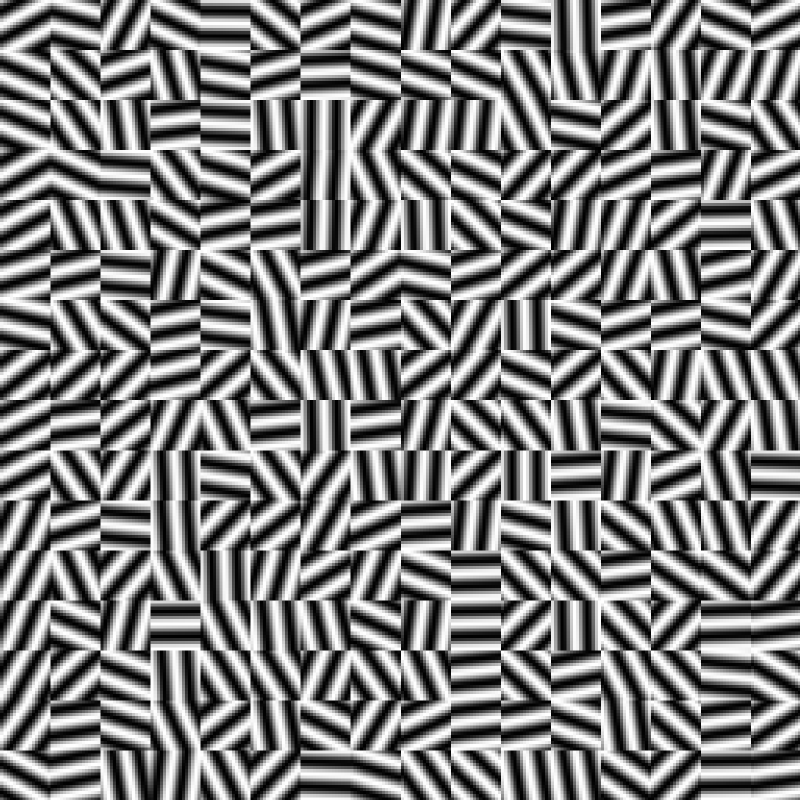}
\hfill
\includegraphics[width=0.2\textwidth]{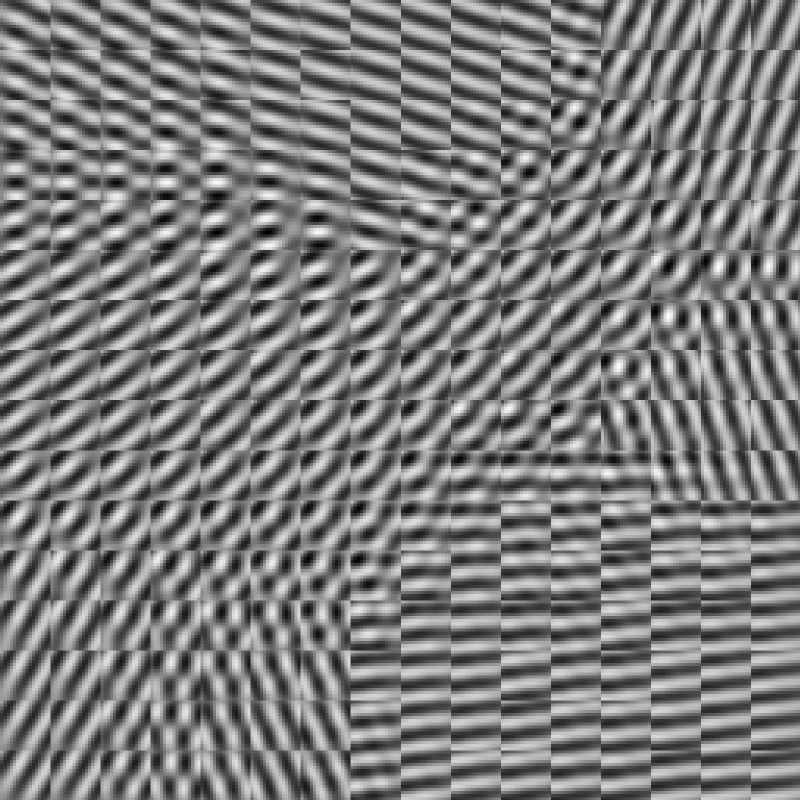}
\includegraphics[width=0.2\textwidth]{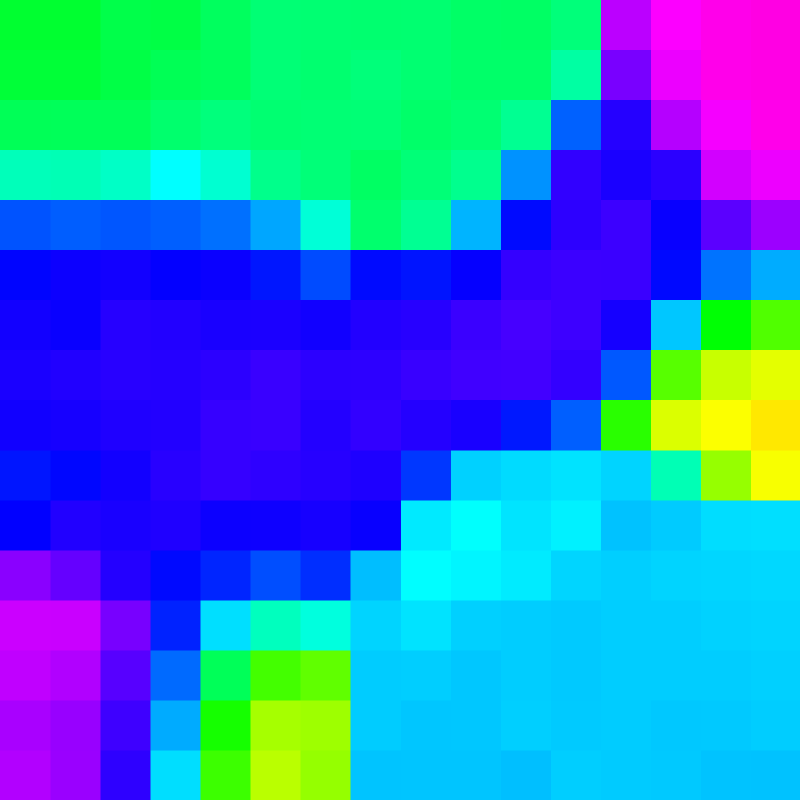}
\hfill
\includegraphics[width=0.11\textwidth]{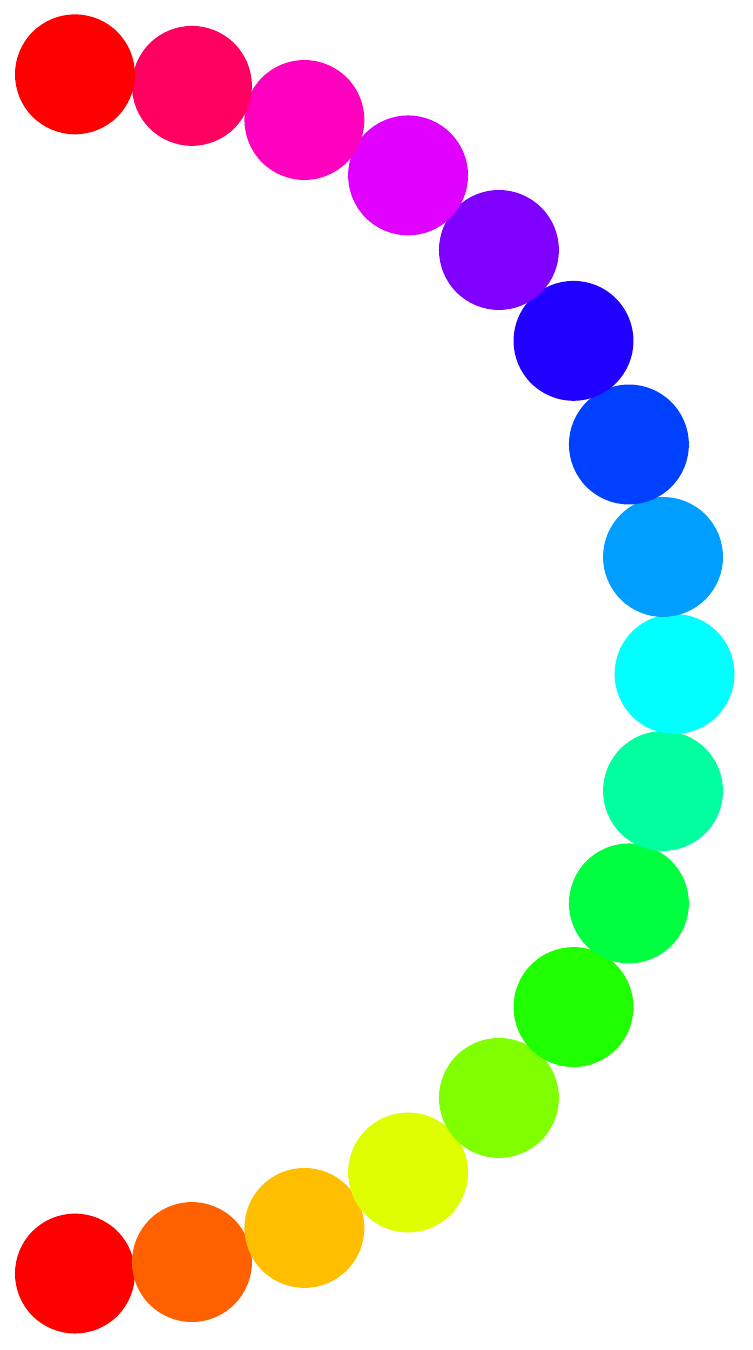}
}
\caption{
\textbf{Left pair:} A random collection of patches with oriented image structure. The colored image displays for each patch its orientation using the color code depicted by the rightmost panel. Each patch is represented by a rank-one matrix $D$ in \eqref{eq:def-L-rho-D}, obtained by vectorizing the patch and taking the tensor product. \textbf{Center pair:} The final state of the QSAF obtained by geometric integration with uniform weighting $\w_{ik}=\frac{1}{|\mc{N}_{i}|},\;\forall k\in\mc{N}_{i},\;\forall i\in\mc{V}$, of the nearest neighbors states. It represents an image partition but preserves image structure, due to geometric smoothing of patches encoded by non-commutative state spaces.
}
\label{fig:randomOrientations}
\end{figure}

\begin{figure}
\centerline{
\subcaptionbox{}{\includegraphics[width=0.2\textwidth]{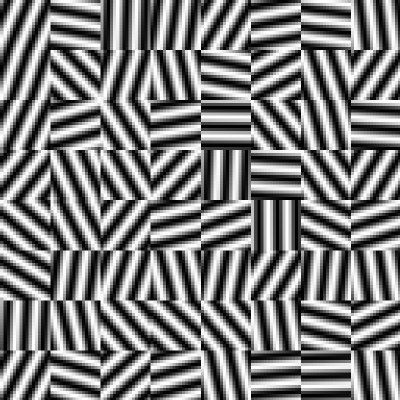}}
\hspace{0.01\textwidth}
\subcaptionbox{}{\includegraphics[width=0.2\textwidth]{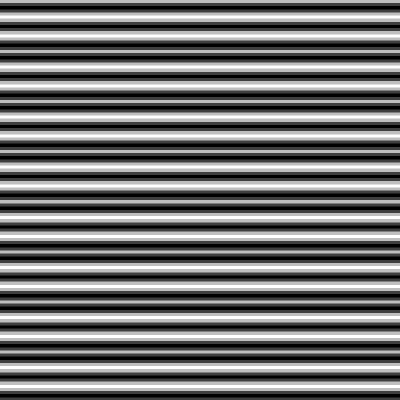}}
\hspace{0.01\textwidth}
\subcaptionbox{}{\includegraphics[width=0.2\textwidth]{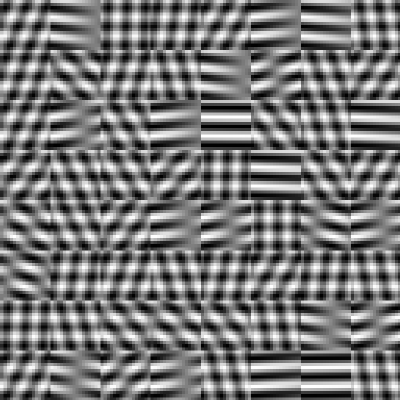}}
\hspace{0.01\textwidth}
\subcaptionbox{}{\includegraphics[width=0.2\textwidth]{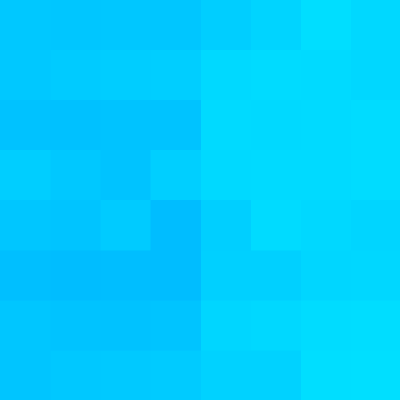}}
}
\caption{
\textbf{(a)} A random collection of patches with oriented image structure. \textbf{(b)} A collection of patches with the same oriented image structure. \textbf{(c)} Pixelwise mean of the patches (a) (b) at each location. \textbf{(d)} The QSAF recovers a close approximation of (b) (color code: see Fig.~\ref{fig:randomOrientations}) by iteratively smoothing the states $\rho_{k},\;k\in\mc{N}_{i}$ corresponding to (c) through geometric integration.
}
\label{fig:randomOrientationPairs}
\end{figure}

\begin{figure}
\centerline{
\includegraphics[width=0.32\textwidth]{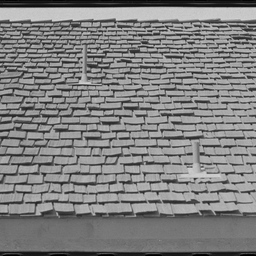}\hspace{0.025\textwidth}
\includegraphics[width=0.32\textwidth]{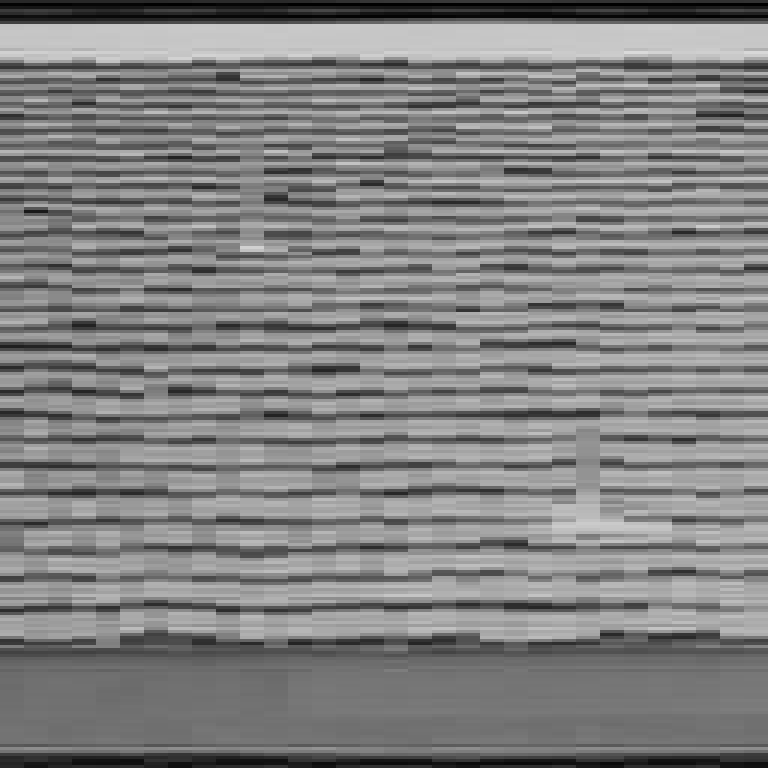}\hspace{0.025\textwidth}
\includegraphics[width=0.32\textwidth]{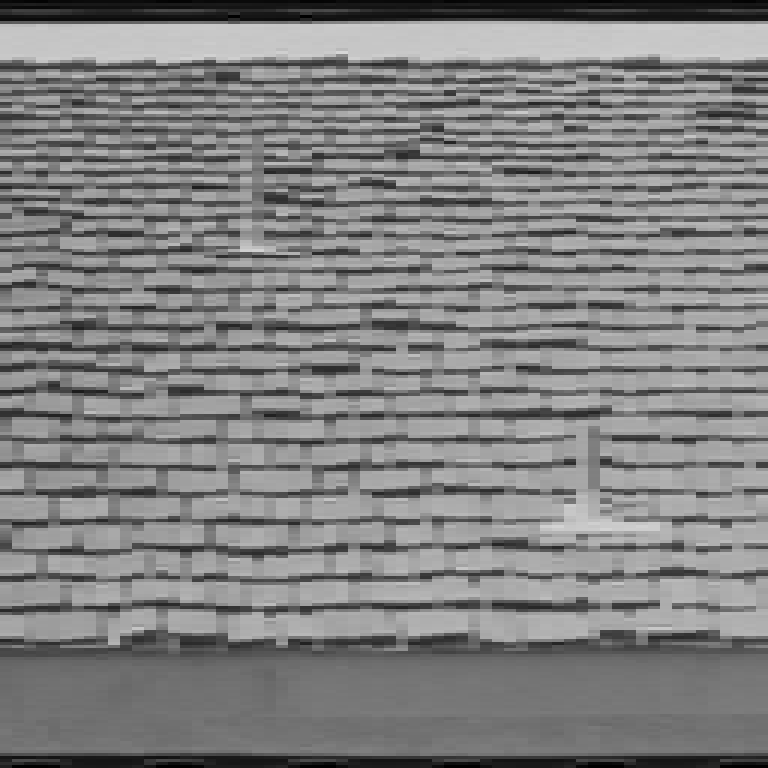}
}
\centerline{
\parbox{0.32\textwidth}{\centering (a)}\hspace{0.025\textwidth}
\parbox{0.32\textwidth}{\centering (b)}\hspace{0.025\textwidth}
\parbox{0.32\textwidth}{\centering (c)}
}
\caption{
\textbf(a) A real image, partitioned into patches of size $8\times 8$ and $4\times 4$ pixels, respectively. Each patch is represented as pure state with respect to a Fourier frame (see text). Instead of the nearest neighbor adjacency on a regular grid, each patch is adjacent to its 8 closest patches in the entire collection. Integrating the QSAF and decoding the resulting states (see text) yields the results (b) ($8\times 8$  patches) and (c) ($4\times 4$  patches), respectively. Result (b) illustrated the effect of smoothing at the patch level, in the Fourier domain, where as the smaller spatial scale used to compute (c) represents the input data fairly accurately, after  significant data reduction.
}
\label{fig:roof-Fourier}
\end{figure}

\subsection{Labeling 3D Data on Bloch Spheres}\label{sec:Bloch-sphere-averaging}

For the purpose of visual illustration, we consider the smoothing of 3D color vectors $d = (d_{1}, d_{2}, d_{3})^{\T}$, interpreted as Bloch vectors which parametrize density matrices \cite[Section 5.2]{Bengtsson:2017aa}
\begin{equation}\label{eq:rho-bloch}
\rho 
= \rho(d)
= \frac{1}{2}\bigg(I + d_{1}\bpm 0 & 1 \\ 1 & 0 \epm + d_{2} \bpm 0 & -i \\ i & 0 \epm + d_{3}\bpm 1 & 0 \\ 0 & -1 \epm\bigg) \in \C^{2\times 2},\qquad \|d\|\leq 1.
\end{equation}
Pure states $\rho$ correspond to unit vectors $d,\;\|d\|=1$, whereas vectors $d,\;\|d\|<1$ parametrize mixed states $\rho$. Given data $d_{i}=(d_{i,1},d_{i,2},d_{i,3})^{\T},\; i\in\mc{V}$ with $\|d_{i}\|\leq 1$, we initialized the QSAF at $\rho_{i}=\rho(d_{i}),\;i\in\mc{V}$, and integrated the flow. Each integration step involves geometric state averaging across the graph causing mixed states $\rho_{i}(t) = \rho(d_{i}(t)),\; i\in\mc{V}$,  which eventually converge towards pure states. Integration was stopped at time $t=T$, when $\min\{\|d_{i}(T)\|\colon i\in\mc{V}\}\geq 0.999$. The resulting vectors $d_{i}(T)$ are visualized as explained in the caption of the Figure \ref{fig:12-colors}. We point out that the two experiments discussed next are supposed to illustrate the behaviour of the QSAF and the impact of the underlying geometry, rather than a contribution to the literature on the processing of color images.

Figure \ref{fig:12-colors}(c) shows a noisy version of the image (b) used to initialize the quantum state assignment flow (QSAF). Panel (d) shows the labeled image, i.e.~the assigment of a pure state (depicted as Bloch vector) to each pixel of the input data (c). Although uniform weights were used and any prior information was absent, the result (d) demonstrates that the QSAF removes the noise and preserves the signal transition fairly well, both for large-scale local image structure (away from the image center) and for small-scale local image structure (close to the image center). This behaviour is quite unusual in comparison to traditional image denoising methods which inevitably require \textit{adaption} of regularization to the scale of local image structure. In addition, we note that noise removal is `perfect' for the three extreme points red, green and blue of panel (a), but suboptimal only for the remaining non-extreme points.

Panels (f)-(h) show the same results when the data are encoded in a better way, as depicted by (e) using unit vectors not only on the positive orthant but on the whole unit sphere. These data are illustrated by RGB vectors that result from translating the unit sphere (e) to the center $\frac{1}{2}(1,1,1)^{\T}$ of the RGB color cube $[0,1]^{3}$ and scaling it by $\frac{1}{2}$. This improved data encoding is clearly visible in panel (g) which displays the \textit{same} noise level as shown in panel (c). Accordingly, noise removal while preserving signal structure at \textit{all} local scales is more effectively achieved by the QSAF in (h), in comparison to (d).

\begin{figure}
\centerline{
\includegraphics[width=0.2\textwidth]{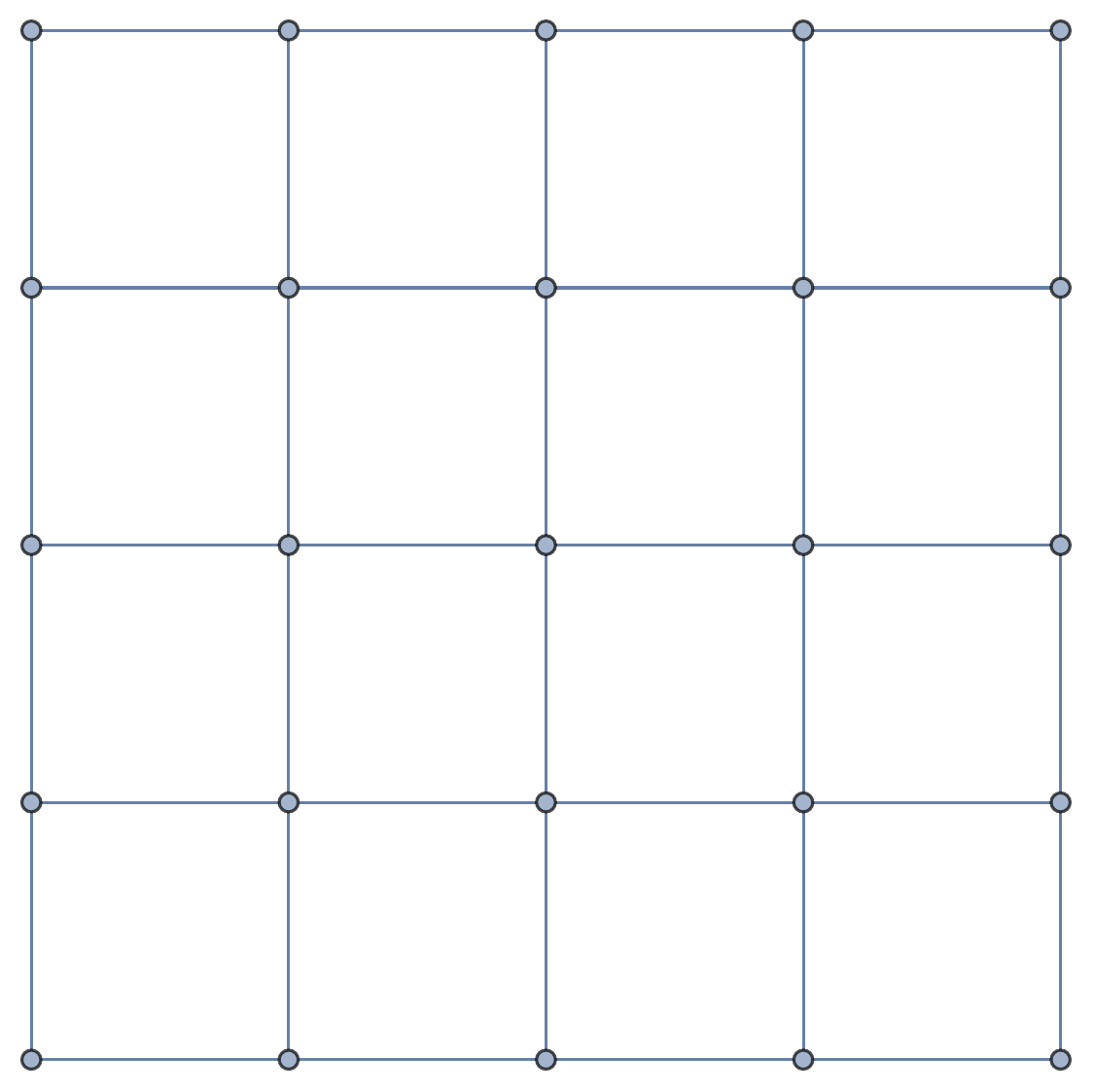}\hspace{0.15\textwidth}
\includegraphics[width=0.2\textwidth]{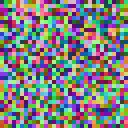}\hspace{0.15\textwidth}
\includegraphics[width=0.2\textwidth]{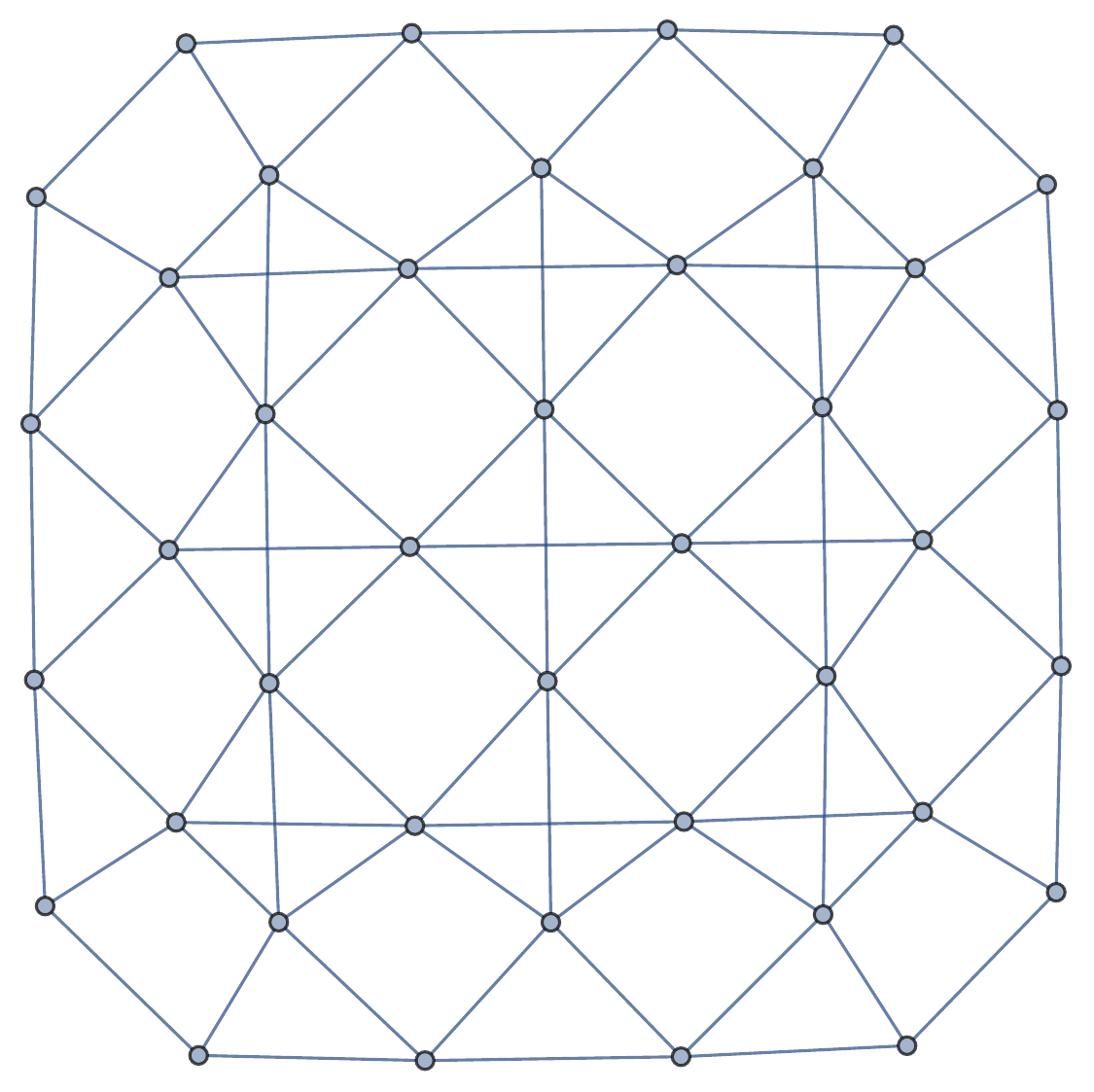}
}
\centerline{
\parbox{0.2\textwidth}{\centering (a)}\hspace{0.15\textwidth}
\parbox{0.2\textwidth}{\centering (b)}\hspace{0.15\textwidth}
\parbox{0.2\textwidth}{\centering (c)}
}
\centerline{
\includegraphics[width=0.3\textwidth]{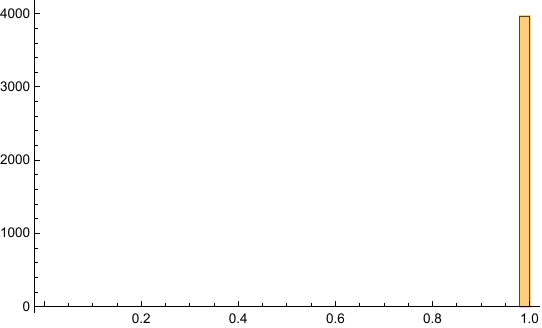}\hspace{0.05\textwidth}
\includegraphics[width=0.3\textwidth]{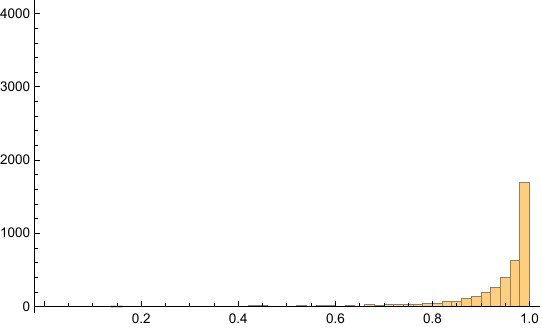}\hspace{0.05\textwidth}
\includegraphics[width=0.3\textwidth]{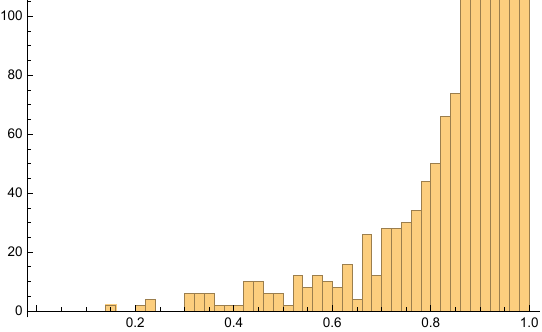}
}
\centerline{
\parbox{0.3\textwidth}{\centering (d)}\hspace{0.05\textwidth}
\parbox{0.3\textwidth}{\centering (e)}\hspace{0.05\textwidth}
\parbox{0.3\textwidth}{\centering (f)}
}
\centerline{
\includegraphics[width=0.3\textwidth]{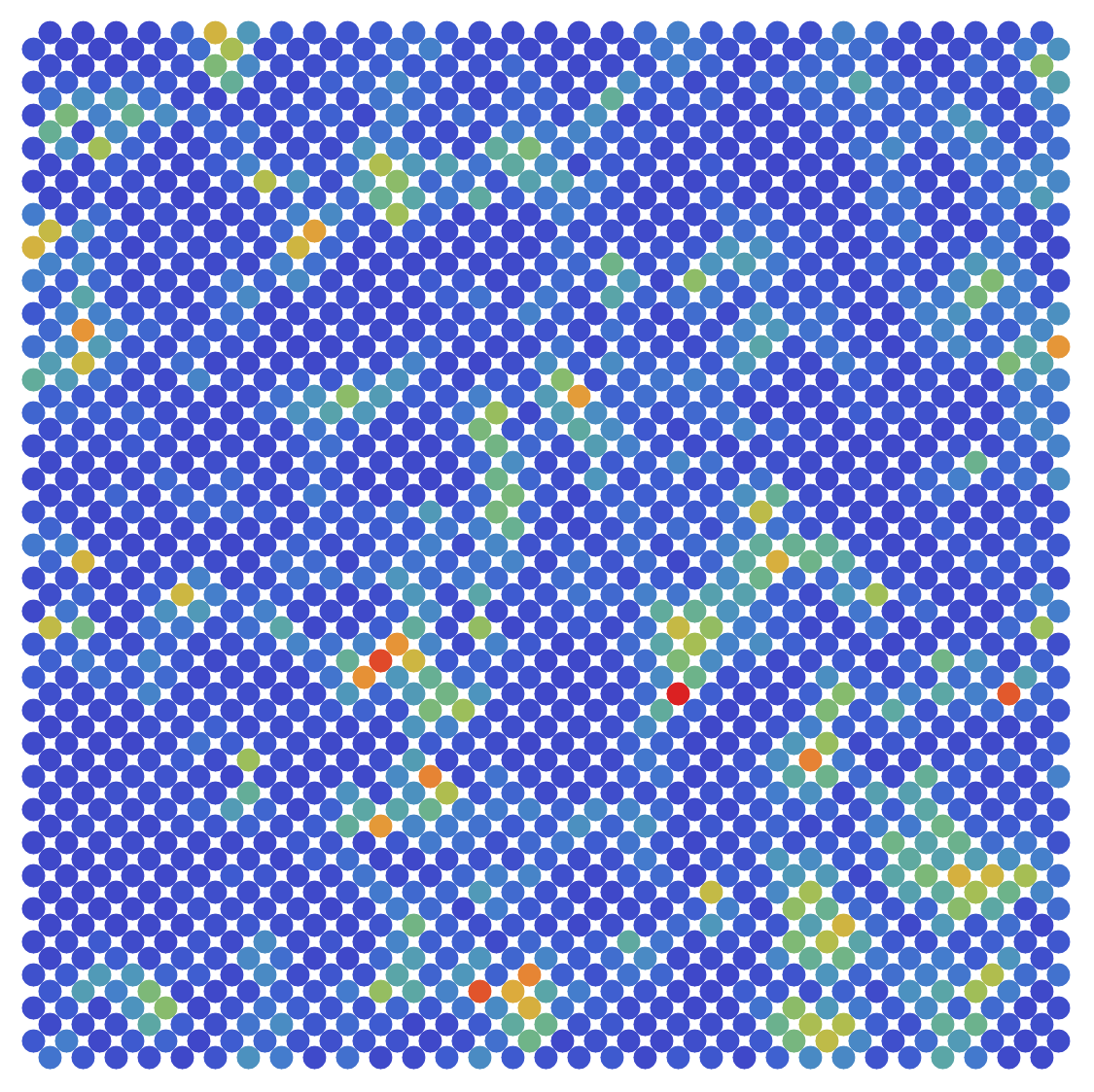}\hspace{0.05\textwidth}
\includegraphics[width=0.3\textwidth]{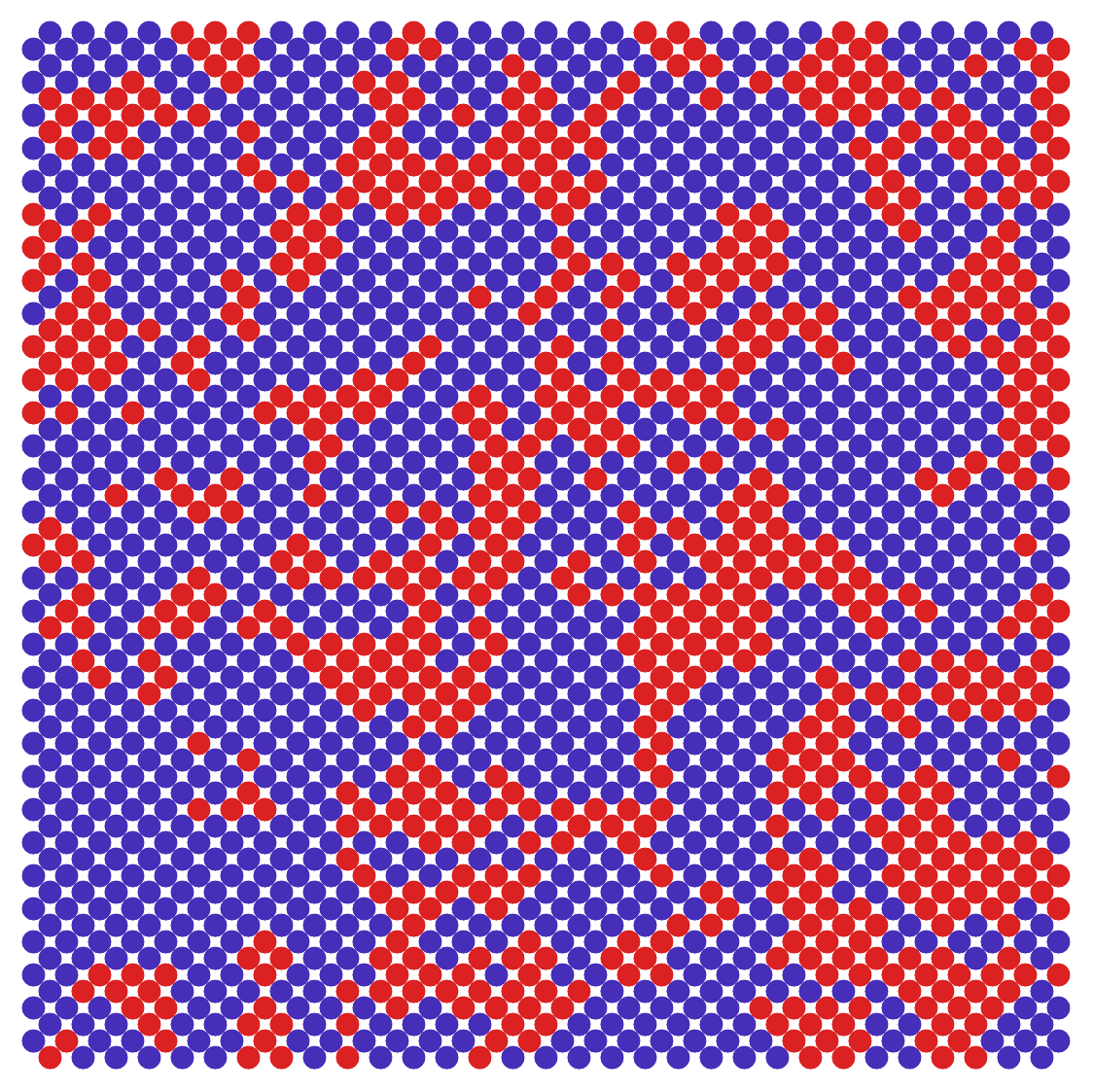}\hspace{0.05\textwidth}
\includegraphics[width=0.3\textwidth]{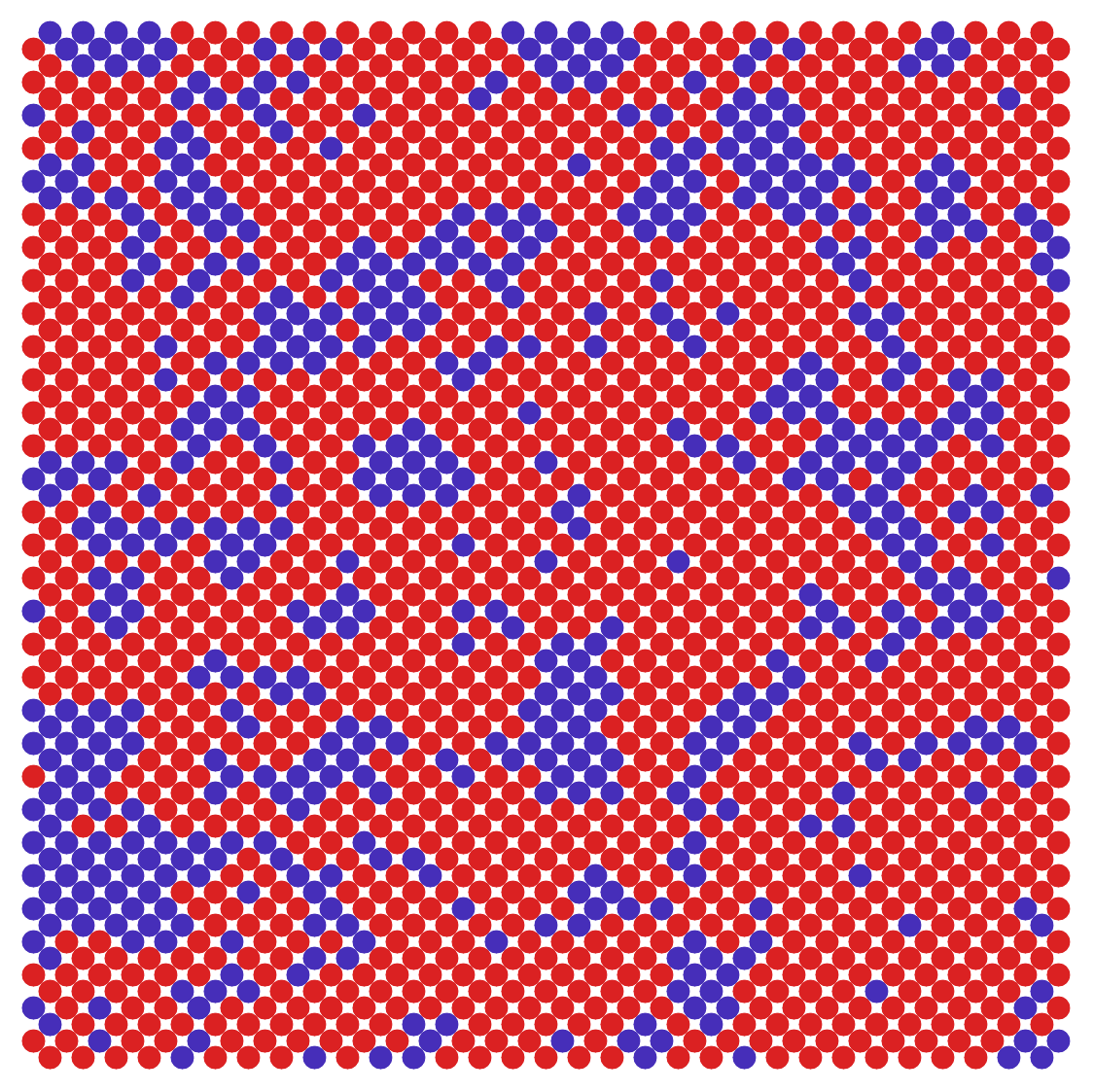}
}
\centerline{
\parbox{0.3\textwidth}{\centering (g)}\hspace{0.05\textwidth}
\parbox{0.3\textwidth}{\centering (h)}\hspace{0.05\textwidth}
\parbox{0.3\textwidth}{\centering (i)}
}
\caption{
        (a) A $5\times 5$ grid graph. (b) Random Bloch vectors $d_{i}\in S^{2}\subset \R^{3}$ (visualized using pseudo-color) defining states 
        $\rho_{i}$ by Eq.~\eqref{eq:rho-bloch} for each vertex of a $32\times 32$ grid graph. 
        (c) The line graph corresponding to (a). Each vertex corresponds to an edge $ij$ of the graph (a) and 
        an initially separable state $\rho_{ij} = \rho_{i}\otimes\rho_{j}$. This defines a simple shallow tensor network.\\ 
        The histograms display the norms of the Bloch vectors of the states $\tr_{j}(\rho_{ij})$ and $\tr_{i}(\rho_{ij})$ obtained
        by the partially tracing out one factor, for each state $\rho_{ij}$ indexed by a vertex $ij$ of the line graph of the grid graph (b).
        (d) The histogram shows that in the initial state, indeed all states are separable, while (e), (f) booth display the histogram of  
        the norms of all Bloch vectors after convergence of the quantum state assignment flow with uniform weights towards pure states. (g) Using the center coordinates of each edge of the grid graph (b), the entanglement represented by $\rho_{ij}$ 
        is visualized by a disk and `heat map' colors (blue: low entanglement, red: large entanglement). 
        For visual clarity, (h) and (i) again display the \textit{same} information after thresholding, using two colors only: Entangled states are marked with red when 
        the norm of the Bloch vectors dropped below the thresholds $0.95$ and $0.99$, respectively, and otherwise with blue.
}
\label{fig:entanglement}
\end{figure}

\subsection{Basic Image Patch Smoothing}\label{sec:patch-smoothing}

Figure \ref{fig:randomOrientations} shows an application of the QSAF to a \textit{random} spatial arrangement (grid graph) of normalized patches, where each vertex represents a patch, not a pixel. Applying vectorization taking the tensor product with itself, each patch is represented as a pure state in terms of a rank-one matrix $D_{i}$ at the corresponding vertex $i\in\mc{V}$, which constitute the input data in the similarity mapping \eqref{eq:def-Si-rho}. Integrating the flow causes the non-commutative interaction of the associated state spaces $\rho_{i},\; i\in\mc{V}$ through geometric averaging, here with uniform weights \eqref{eq:omega-sum-1},  until convergence towards pure states. The resulting patches are then simply given by the corresponding eigenvector, possibly after reversing the arbitrary sign of each eigenvector component, depending on the distance to the input patch.

The result shown in Figure \ref{fig:randomOrientations} reveals an interesting behaviour: structure-preserving patch smoothing without accessing explicitly individual pixels. In particular, the flow induces a \textit{partition} of the patches without any prior assumption on the data.

Figure \ref{fig:randomOrientationPairs} shows a variant of the scenario of Figure \ref{fig:randomOrientations} in order to demonstrate in another way the ability to separate local image structure by geometric smoothing at the patch level.

Figure \ref{fig:roof-Fourier} generalizes the set-up in two ways. Firstly, patches were encoded using the harmonic frame given by the two-dimensional discrete Fourier matrix. Secondly, non-uniform weights $\w_{ik}=e^{-\tau \|P_{i}-P_{j}\|_{F}^{2}},\;\tau>0$ were used depending on the distance of adjacent patches $P_{i}, P_{j}$.

Specifically, let $P_{i}$ denote the patch at vertex $i\in\mc{V}$ after removing the global mean and normalization using the Frobenius norm. Then, applying the FFT to each patch and vectorization, formally with the discrete two-dimensional Fourier matrix $F_{2}=F \otimes F$ (Kronecker product) and followed by stacking the rows, $\wh{p}_{i}= F_{2}\vvec(P_{i})$, the input data were defined as $D_{i} = F_{2}\Diag(-|\wh{p}_{i}|^{2}) F_{2}^{\ast}$, where the squared magnitude $|\cdot|^{2}$ was computed componentwise. Integrating the flow yields again pure states which were interpreted and decoded accordingly: the eigenvector was used as multiplicative filter of the magnitude of the Fourier transformed patch (keeping its phase), followed by rescaling the norm and adding the mean by approximating the original patch in terms of these two parameters.

The results shown as panels (b) and (c) of Figure \ref{fig:roof-Fourier} illustrate the effect of `geometric diffusion' at the patch level through integrating the flow, and how the input data are approximated depending on the chosen spatial scale (patch size), subject to significant data reduction.

\section{Conclusion}\label{sec:Conclusion}

We generalized the assignment flow approach for categorial distributions \cite{Astrom:2017ac} to density matrices on weighted graphs. While the former flows assign to each data point a label selected from a \textit{finite} set, the latter assign  to each data point a generalized `label' from the \textit{uncountable} submanifold of pure states.

Various further directions of research are indicated by the numerical experiments. This includes the unusual behavior of feature vector smoothing which parametrize complex-valued non-commutative state spaces (Figure \ref{fig:12-colors}), the structure-preserving interaction of spatially indexed feature patches without accessing individual pixels (Figures \ref{fig:randomOrientations} and \ref{fig:randomOrientationPairs}), the use of frames for signal representation and as obvervables whose expected values are governed by a quantum state assignment flow (Figure \ref{fig:roof-Fourier}), and the representation of spatial correlations by entanglement and tensorization (Figure \ref{fig:entanglement}). Extending to the novel quantum assignment flow approach the representation of the original assignment flow in the broader framework of geometric mechanics, as developed recently by \cite{Savarino:2021aa},  defines another promising research project spurred by established concepts of mathematics and physics. 

From these viewpoints, this paper adds a novel concrete approach based on information theory to the emerging literature on network design based on concepts from quantum mechanics; cf., e.g.~\cite{Levine:2018aa} and references therein. Our main motivation is the definition of a novel class of `neural ODEs' \cite{Chen:2018ab} in terms of the dynamical systems which generate a quantum state assignment flow. The layered architecture of a corresponding `neural network' is implicitly given by geometric integration. The inherent smoothness of the parametrization enables to learn the weight parameters from data. This will be explored in our future work along the various lines of research indicated above.


\clearpage
\section{Proofs}\label{sec:appendix}

\subsection{Proofs of Section~\ref{sec:information-geometry}}
\label{sec:appendix-sec-information-geometry}

\begin{proof}[Proof of Proposition~\ref{prop:properties-of-Rp}]
    We verify \eqref{eq:Rp-Pi0} by direct computation. For any $p\in\mc{S}_{c}$,
    \begin{subequations}
    \begin{align}
    R_{p}\eins_{c} &= \big(\Diag(p)-p p^{\T}\big)\eins_{c}
    = p - \la p,\eins_{c}\ra p = 0,\\
    R_{p}\pi_{c,0} &= R_{p}(I-\eins_{c}\eins_{\mc{S}_{c}}^{\T}) 
    = R_{p},\\
    \pi_{c,0} R_{p} &= (I-\eins_{c}\eins_{\mc{S}_{c}}^{\T}) R_{p}
    = R_{p} - \frac{1}{c}\eins_{c}(R_{p}\eins_{c})^{\T} = R_{p}.
    \end{align}
    \end{subequations}
    \noindent
    Next we characterize the geometric role of $R_{p}$ and show \eqref{eq:grad-R-simplex}. Let $p\in\mc{S}_{c}$ be parametrized by the local coordinates 
    \begin{subequations}
    \begin{align}
    \ol{p} &= \vphi(p) := (p_{1},p_{2},\dotsc,p_{c-1})^{\T} \in \R_{++}^{c-1}\\ 
    \label{eq:p-from-local}
    p &= \vphi^{-1}(\ol{p}) = (\ol{p}_{1},\dotsc,\ol{p}_{c-1},1-\la\eins_{c-1},\ol{p}\ra)^{\T}\in\mc{S}_{c}.
    \end{align}
    \end{subequations}
    Choosing the canonical basis $e_{1},\dotsc,e_{c}$ on $\mc{S}_{c}\subset\R^{c}$, we obtain a basis of the tangent space $T_{c,0}$
    \begin{equation}
    e_{j}-e_{c} = d\vphi^{-1}(e_{j}),\qquad j\in[c-1].
    \end{equation}
    Using these vectors a columns of the matrix
    \begin{equation}\label{eq:def-B-T0}
    B := (e_{1}-e_{c},\dotsc,e_{c-1}-e_{c}) = \bpm I_{c-1} \\ 
    -\eins_{c-1}^{\T} \epm \in \R^{c\times (c-1)},
    \end{equation}
    one has for any $v\in T_{c,0}$
    \begin{subequations}
    \begin{align}
    v &= B \ol{v} = \bpm \ol{v} \\ v_{c} \epm
    = \bpm \ol{v} \\ -\la\eins_{c-1},\ol{v}\ra \epm, &
    \ol{v} &= (v_{1},\dotsc,v_{c-1})^{\T}\\
    \ol{v} &= B^{\dagger} v, &
    B^{\dagger} &= \bpm I_{c-1} & 0\epm \pi_{c,0},
    \end{align}
    \end{subequations}
    where $B^{\dagger} := (B^{\T} B)^{-1} B^{\T}$ denotes the Moore-Penrose generalized inverse of $B$. Substituting this parametrization and evaluating the metric \eqref{eq:def-gp} gives
    \begin{subequations}
    \begin{align}
    g_{p}(u,v) 
    &= \la \ol{u},B^{\T}\Diag(p)^{-1}B\ol{v}\ra
    = \Big\la\ol{u},\bpm I_{c-1} & -\eins_{c-1}\epm \Diag(p)^{-1} \bpm I_{c-1} \\ -\eins_{c-1}^{\T}\epm \ol{v}\Big\ra\\
    &= \Big\la\ol{u},\Big(\Diag(\ol{p})^{-1} + \frac{1}{1-\la\eins_{c-1},\ol{p}\ra}\eins_{c-1}\eins_{c-1}^{\T}\Big)\ol{v}\Big\ra\\
    &=: \la\ol{u},G(\ol{p})\ol{v}\ra.
    \end{align}
    \end{subequations}
    Applying the Sherman-Morrison-Woodbury matrix inversion formula \cite[p.~9]{Horn:2013aa}
    \begin{equation}
    (A+x y^{\T})^{-1} = A^{-1} - \frac{A^{-1}x y^{\T} A^{-1}}{1 + \la y, A^{-1} x\ra}
    \end{equation}
    yields
    \begin{subequations}\label{eq:R_olp_G-1}
    \begin{align}
    G(\ol{p})^{-1} 
    &= \Diag(\ol{p}) - \frac{1}{1-\la\eins_{c-1},\ol{p}\ra}\frac{\Diag(\ol{p})\eins_{c-1}\eins_{c-1}^{\T}\Diag(\ol{p})}{1 + \frac{1}{1-\la\eins_{c-1},\ol{p}\ra}\la\eins_{c-1},\ol{p}\ra}\\
    &= \Diag(\ol{p}) - \Diag(\ol{p})\eins_{c-1}\eins_{c-1}^{\T}\Diag(\ol{p}) = \Diag(\ol{p}) - \ol{p}\,\ol{p}^{\T}\\
    &= R_{\ol{p}}.
    \end{align}
    \end{subequations}
    Let $v\in T_{c,0}$. Then, using the equations
    \begin{subequations}
    \begin{align}
    p_{c}                   &\overset{\eqref{eq:p-from-local}}{=}1-\la\eins_{c-1},\ol{p}\ra, \\
    R_{\ol{p}}\eins_{c-1}   &= \ol{p}-\la\eins_{c-1},\ol{p}\ra\ol{p} = p_{c}\ol{p}, 
    \end{align}
    \end{subequations}
    we have
    \begin{subequations}\label{eq:Rp-v-local}
    \begin{align}
    R_{p} v &= \bpm R_{\ol{p}} & - p_{c}\ol{p}\\
    - p_{c}\ol{p}^{\T} & p_{c}-p_{c}^{2} \epm \bpm \ol{v} \\ 
    v_{c} \epm = \bpm R_{\ol{p}}\ol{v} - v_{c} R_{\ol{p}}\eins_{c-1}\\
    -\la R_{\ol{p}}\eins_{c-1},\ol{v}\ra + v_{c}p_{c}\la\eins_{c-1},\ol{p}\ra \epm\\
    &= \bpm R_{\ol{p}}\ol{v} \\ -\la\eins_{c-1},R_{\ol{p}}\ol{v}\ra
    \epm - v_{c} \bpm R_{\ol{p}}\eins_{c-1} \\ 
    -\la\eins_{c-1},R_{\ol{p}}\eins_{c-1}\ra \epm\\
    &\overset{\eqref{eq:def-B-T0}}{=} B R_{\ol{p}} (\ol{v} - v_{c}\eins_{c-1}).
    \end{align}
    \end{subequations}
    Now consider any smooth function $f\colon\mc{S}_{c}\to\R$. Then
    \begin{subequations}
    \begin{align}
    \partial_{\ol{p}_{i}}\big(f\circ\vphi^{-1}(\ol{p})\big)
    &= \sum_{j\in[c]}\partial_{j}f(p)\partial_{\ol{p}_{i}}\vphi^{-1}(\ol{p})
    \overset{\eqref{eq:p-from-local}}{=}
    \partial_{i}f(p)-\partial_{c}f(p),
    \\
    \partial_{\ol{p}}\big(f\circ\vphi^{-1}(\ol{p})\big)
    &= \ol{\partial f(p)}-\partial_{c}f(p)\eins_{c-1}.
    \end{align}
    \end{subequations}
    Comparing the last equation and \eqref{eq:Rp-v-local} shows that 
    \begin{equation}\label{eq:Rp-partial-f-local}
    R_{p}\partial f(p) = B R_{\ol{p}}\partial_{\ol{p}}\big(f\circ\vphi^{-1}(\ol{p})\big)
    \overset{\eqref{eq:R_olp_G-1}}{=} 
    B G(p)^{-1} \partial_{\ol{p}}\big(f\circ\vphi^{-1}(\ol{p})\big),
    \end{equation}
    which proves \eqref{eq:grad-R-simplex}.
\end{proof}

\vspace{1cm}

\newpage
\subsection{Proofs of Section~\ref{sec:AF}}\label{sec:appendix_standard_af_proofs}

\begin{proof}[Proof of Lemma~\ref{lem:derivatives_exp_standard_af}]
    Let $v(t)\in T_{c,0}$ be a smooth curve with $\dot v(t) = u$. Then
    \begin{subequations}
    \begin{align}
    \frac{d}{dt}\exp_{p}\big(v(t)\big)
    &= \frac{d}{dt}\frac{p\cdot e^{v(t)}}{\la p,e^{v(t)}\ra}
    = \frac{p\cdot u\cdot e^{v(t)}}{\la p,e^{v(t)}\ra} - \la p,u\cdot e^{v(t)}\ra\frac{p\cdot e^{v(t)}}{\la p,e^{v(t)}\ra^{2}}
    \\
    &= \exp_{p}\big(v(t)\big)\cdot u - \big\la u,\exp_{p}\big(v(t)\big)\big\ra \exp_{p}\big(v(t)\big)
    = R_{\exp_{p}(v(t))} u.
    \end{align}
    \end{subequations}
    Similarly, for a smooth curve $p(t)\in\mc{S}_{c}$ with $\dot p(t)=u$, one has
    \begin{subequations}
    \begin{align}
    \frac{d}{dt}\exp_{p(t)}(v) 
    &= \frac{d}{dt}\frac{p(t)\cdot e^{v}}{\la p(t),e^{v}\ra}
    = \frac{\dot p(t)\cdot e^{v}}{\la p(t),e^{v}\ra} - \la\dot p(t),e^{v}\ra\frac{p(t)\cdot e^{v}}{\la p(t),e^{v}\ra^{2}}
    \\
    &= \exp_{p(t)}(v)\cdot\frac{u}{p(t)} - \Big\la\frac{u}{p(t)},\exp_{p(t)}(v)\Big\ra \exp_{p(t)}(v)
    = R_{\exp_{p(t)}(v)}\frac{u}{p(t)}. \qedhere
    \end{align}
    \end{subequations}
\end{proof}

\begin{proof}[Proof of Theorem~\ref{thm:single-vertex-AF-parametrization}]
    Put 
    \begin{equation}\label{eq:def-q-L-p-t-D}
    q(t) = L_{p(t)}(D)
    \end{equation}
    where $p(t)$ solves \eqref{eq:single-vertex-AF}. Using \eqref{eq:def-LpD}, \eqref{eq:dLpD} and \eqref{eq:def-q-L-p-t-D}, we obtain
    \begin{equation}
    \dot q 
    = d_{p}L_{p(t)}(D)[\dot p(t)]
    = R_{q(t)}\Big(\frac{\dot p(t)}{p(t)}\Big)
    \overset{\eqref{eq:single-vertex-AF}}{=}
    R_{q(t)} \big(q(t)-\la p(t),q(t)\ra\eins_{c}\big)\overset{\eqref{eq:Rp-eins}}{=}R_{q(t)} q(t),  
    \end{equation}
    which shows \eqref{eq:single-vertex-AF-parametrization-b}. Write $p(t)=\exp_{\eins_{\mc{S}_{c}}}(r(t))$.  Then differentiating \eqref{eq:single-vertex-AF-parametrization-a} yields with $r(t) = \int_{0}^{t}q(\tau)d\tau$
    \begin{equation}
    \dot p(t) 
    \overset{\eqref{eq:dexp-p}}{=}
    R_{\exp_{\eins_{\mc{S}_{c}}}(r(t))} \dot r(t)
    \overset{\eqref{eq:single-vertex-AF-parametrization-a}}{=}
    R_{p(t)} q(t)
    \overset{\eqref{eq:def-q-L-p-t-D}}{=} 
    R_{p(t)} L_{p(t)}(D),
    \end{equation}
    which proves the equivalence of \eqref{eq:single-vertex-AF} and \eqref{eq:single-vertex-AF-parametrization}.
\end{proof}

\begin{proof}[Proof of Corollary~\ref{prop:single-vertex-AF}]
        The solution $p(t)$ to \eqref{eq:single-vertex-AF} is given by \eqref{eq:single-vertex-AF-parametrization}. Proposition \eqref{prop:properties-of-Rp} and  Eq.~\eqref{eq:grad-R-simplex} show that \eqref{eq:single-vertex-AF-parametrization-b} is the Riemannian ascent flow of the function $\mc{S}_{c}\ni q\mapsto \frac{1}{2}\|q\|^{2}$. The stationary points satisfy
        \begin{equation}
        R_{q} q = (q-\|q\|^{2})\cdot q = 0
        \end{equation}
        and form the set
        \begin{equation}
        Q^{\ast} := \Big\{q^{\ast} = \frac{1}{|\mc{J}^{\ast}|}\sum_{j\in\mc{J}^{\ast}} e_{j}\colon \mc{J}^{\ast}\subseteq [c]\Big\}.
        \end{equation}
        The case $\mc{J}^{\ast}=[c]$, i.e.~$q^{\ast}=\eins_{\mc{S}_{c}}$, can be ruled out if $\frac{D}{\la\eins_{c},D\ra}\neq\mc{S}_{c}$, which will always be the case in practice where $D$ corresponds to real data (measurement, observation). The global maxima correspond to the vertices of $\Delta_{c}=\ol{\mc{S}}_{c}$, i.e.~$|\mc{J}^{\ast}|=1$. The remaining stationary points are local maxima and degenerate, since vectors $D$ with non-unique minimal component form a negligible null set. In any case, $\lim_{t\to\infty} p(t) \overset{\eqref{eq:single-vertex-AF-parametrization-a}}{=} \lim_{t\to\infty}q(t) = q^{\ast}$, depending on the index set $\mc{J}^{\ast}$ determined by $D$.
\end{proof}
%

\newpage
\subsection{Proofs of Section~\ref{sec:DAF-all}}\label{sec:appendix-quantum-state-assignment-flow}

\begin{proof}[Proof of Proposition~\ref{prop:grad-Dc}]
	The Riemannian gradient is defined by \cite[pp.~337]{Kobayashi:1969aa}
	\begin{subequations}
	\begin{align}
	0 &= df[X]-g_{\rho}(\ggrad_{\rho}f, X)
	\overset{\eqref{eq:def-inner-rho-T}}{=} 
	\la\partial f, X\ra - \la \mb{T}_{\rho}[\ggrad_{\rho}f],X\ra
	\\
	&= \la\partial f-\mb{T}_{\rho}[\ggrad_{\rho}f],X\ra,\qquad
	\forall X\in\mc{H}_{c,0}.
	\end{align}
	\end{subequations}
	Choosing the parametrization $X=Y-\tr(Y)I\in\mc{H}_{c,0}$ with $Y\in\mc{H}_{c}$, we further obtain
	\begin{subequations}
	\begin{align}
	0 &= \la\partial f-\mb{T}_{\rho}[\ggrad_{\rho}f],Y\ra - \tr(Y)\tr(\partial f-\mb{T}_{\rho}[\ggrad_{\rho}f])
	\\
	&= \big\la\partial f-\mb{T}_{\rho}[\ggrad_{\rho}f] - \tr(\partial f-\mb{T}_{\rho}[\ggrad_{\rho}f]) I, Y \big\ra,\quad
	\forall Y\in\mc{H}_{c}.
	\end{align}
	\end{subequations}
	The left factor must vanish. Applying the linear mapping $\mb{T}_{\rho}^{-1}$ and solving for $\ggrad_{\rho} f$ and  gives
	\begin{equation}
	\ggrad_{\rho} f 
	= \mb{T}_{\rho}^{-1}[\partial f] - \tr(\partial f-\mb{T}_{\rho}[\ggrad_{\rho}f]) \mb{T}_{\rho}^{-1}[I].
	\end{equation}
	Since $\ggrad_{\rho}f\in\mc{H}_{c,0}$, taking the trace on both sides and using $\tr\mb{T}_{\rho}^{-1}[I]=\tr\rho=1$ yields
	\begin{equation}
	0 = \tr\mb{T}_{\rho}^{-1}[\partial f] - \tr\partial f + \tr\mb{T}_{\rho}[\ggrad_{\rho}f].
	\end{equation}
	Substituting the last two summands in the equation before gives
	\begin{subequations}
	\begin{align}
	\ggrad_{\rho} f
	&= \mb{T}_{\rho}^{-1}[\partial f] - (\tr\mb{T}_{\rho}^{-1}[\partial f])\rho
	\\
	&= \mb{T}_{\rho}^{-1}[\partial f] - \la\rho, \partial f \ra\rho,
	\end{align}
	\end{subequations}
	where the last equation follows from \eqref{eq:ToDo-tr}. 
	\end{proof}

\begin{proof}[Proof of Lemma~\ref{lem:repl-proj-rho-commute}]
        The equation $\Pi_{c,0}\circ \Rep_{\rho} = \Rep_{\rho}$ follows from $\Rep_{\rho}[X]\in\mc{H}_{c,0}$ and hence 
        \begin{equation}
        \tr \Rep_{\rho}[X] 
        \overset{\eqref{eq:def-Rrho}}{=} 
        \tr \mb{T}_{\rho}^{-1}[X]-\la\rho,X\ra\tr\rho
        \overset{\eqref{eq:ToDo-tr}}{=} 
        \la \rho,X\ra - \la \rho,X\ra = 0.
        \end{equation} 
        Thus
        \begin{subequations}
        \begin{align}
        \Pi_{c,0}\circ \Rep_{\rho}[X]
        &= \Rep_{\rho}[X]
        = \Rep_{\rho}[X] - \frac{\tr X}{c}\big(\rho-\underbrace{\la\rho, I\ra}_{=1}\rho\big)
        \\
        &= \Rep_{\rho}[X] - \frac{\tr X}{c}\Rep_{\rho}[I]
        = \Rep_{\rho}\Big[X-\frac{\tr X}{c}I\Big]
        \overset{\eqref{eq:def-Pic0}}{=} 
        \Rep_{\rho}\circ\Pi_{c,0}[X]. \qedhere
        \end{align}
        \end{subequations}
\end{proof}

\begin{proof}[Proof of Lemma~\ref{lem:Gamma-Pi0}]
	Using \eqref{eq:def-Pic0} we compute
	\begin{equation}
	\expm(\Pi_{c,0}[Z])
	= \expm\Big(Z-\frac{\tr Z}{c}I\Big)
	= e^{\frac{\tr Z}{c}} \expm(Z),
	\end{equation}
	where the last equation holds since $Z$ and $I$ commute. Substitution into \eqref{eq:def-Gamma-a} cancels the scalar factor $e^{\frac{\tr Z}{c}}$ and shows \eqref{eq:Gamma-Pi0}.
\end{proof}

\begin{proof}[Proof of Proposition~\ref{prop:Gamma-invertible}]
	We show $\Gamma\circ\Gamma^{-1}=\mrm{id}_{\mc{D}_{c}}$ and $\Gamma^{-1}\circ\Gamma=\mrm{id}_{\mc{H}_{c,0}}$. As for the first relation, we compute
	\begin{subequations}
	\begin{align}
	\Gamma\circ\Gamma^{-1}(\rho)
	&= \exp_{m}\Big(\Gamma^{-1}(\rho)-\psi\big(\Gamma^{-1}(\rho)\big)I\Big)
	\\
	&= \expm\Big(\logm\rho-\frac{\tr(\logm\rho)}{c}I - \log\Big(\tr\expm\big(\logm\rho-\frac{\tr(\logm\rho)}{c}I\big)\Big)I\Big)
	\intertext{and since $\logm\rho$ and $I$ commute}
	&= \expm\Big(\logm\rho-\frac{\tr(\logm\rho)}{c}I - \log\tr\big(e^{-\frac{1}{c}\tr(\logm\rho)}\rho\big)I\Big)
	\\
	&\overset{\tr\rho=1}{=} \expm(\logm\rho)
	\\
	&= \rho.
	\end{align}
	\end{subequations}
	As for the second relation, we compute
	\begin{subequations}\allowdisplaybreaks
	\begin{align}
	\Gamma^{-1}\circ\Gamma(X)
	&= \Pi_{c,0}[\logm\circ\Gamma(X)]
	= \Pi_{c,0}\big[\logm\circ\expm\big(X-\psi(X)I\big)\big]
	\\
	&= \Pi_{c,0}[X]-\psi(X)\Pi_{c,0}[I]
	= \Pi_{c,0}[X]
	\\
	&= X,
	\end{align}
	\end{subequations}
	since $X\in\mc{H}_{c,0}$ by assumption.
\end{proof}

\begin{proof}[Proof of Lemma~\ref{lem:dGamma-1}]
    In view of the definition \eqref{eq:def-Gamma} of $\Gamma$, we compute using the chain rule
    \begin{subequations}\label{eq:proof-dGamma-T-roh-1}
    \begin{align}
    d\Gamma(H)[Y]
    &= \frac{d}{dt}\exp_{m}\big(H+t Y-\psi(H+t Y)I\big)\big|_{t=0}
    \\
    &= d\expm\big(H-\psi(H)I\big)\big[Y-d\psi(H)[Y] I\big]
    \\
    &\overset{\eqref{eq:def-mcT-inverse}}{=}
    \mb{T}_{\rho}^{-1}\big[Y-d\psi(H)[Y] I\big].
    \end{align}
    \end{subequations}
    Furthermore,
    \begin{subequations}
    \begin{align}
    d\psi(H)[Y]
    &\overset{\eqref{eq:def-psi-Gamma}}{=}
    \frac{1}{\tr\exp_{m}(H)}\tr(d\exp_{m}(H)[Y])
    \\
    &\overset{\eqref{eq:def-mcT-inverse}}{=}
    \frac{1}{\tr\exp_{m}(H)}\tr\big(\mb{T}^{-1}_{\exp_{m}(H)}[Y]\big),\qquad
    \exp_{m}(H)
    \overset{\eqref{eq:def-Gamma-a}}{=}
    \big(\tr\exp_{m}(H)\big)\Gamma(H)
    \\
    &\overset{\eqref{eq:ToDo-tr}}{=} 
    \frac{1}{\tr\exp_{m}(H)} \la\exp_{m}(H),Y\ra
    \\
    &\overset{\eqref{eq:def-Gamma-a}}{=} 
    \la\Gamma(H),Y\ra 
    = \la\rho,Y\ra,
    \end{align}
    \end{subequations}
    where the last equation follows from the assumption $\rho=\Gamma(H)$. 
    Substitution into \eqref{eq:proof-dGamma-T-roh-1} gives \eqref{eq:dGamma-H}. 
    Regarding \eqref{eq:dGamma-1}, using the expression  \eqref{eq:def-Gamma-1} for $\Gamma^{-1}$, we compute
    \begin{subequations}
    \begin{align}
    d\Gamma^{-1}(\rho)[X]
    &= \Pi_{c,0}\circ d\logm(\rho)[X]
    \\
    &\overset{\eqref{eq:def-mcT}}{=}
    \Pi_{c,0}\circ \mb{T}_{\rho}[X],
    \end{align}
    \end{subequations}
    which verifies \eqref{eq:dGamma-1}.
\end{proof}

\begin{proof}[Proof of Proposition~\ref{prop:e-geodesic}]
    The e-geodesic connecting the two points $Q, R\in\mc{D}_{c}$ is given by \cite[Section V]{Petz:1994aa}
    \begin{equation}\label{eq:e-geodesic-Petz}
    \Gamma(K+t A),\quad t\in[0,1],\qquad
    K=\logm Q,\quad 
    A=\logm R - \logm Q.
    \end{equation}
    Setting $\Gamma^{-1}(\rho)=\Pi_{c,0}[K]$ and $\mb{T}_{\rho}[X]=A$ yields \eqref{eq:def-Exp-rho-e-b}, since the orthogonal projections $\Pi_{c,0}$ onto $\mc{H}_{c,0}$ are implicitly carried out in \eqref{eq:e-geodesic-Petz} as well, due to Lemma \ref{lem:Gamma-Pi0}. The expression \eqref{eq:def-Exp-rho-e} is equal to \eqref{eq:def-Exp-rho-e-b} due to \eqref{eq:dGamma-1}. It remains to check that the geodesic emanates at $\rho$ in the direction $X$. We compute
    \begin{subequations}
    \begin{align}
    \gamma_{\rho,X}^{(e)}(0)
    &= \Gamma(\Gamma^{-1}(\rho)) = \rho
    \\
    \frac{d}{dt}\gamma_{\rho,X}^{(e)}(0)
    &= \frac{d}{dt}\Gamma\big(\Gamma^{-1}(\rho)+ t d\Gamma^{-1}(\rho)[X]\big)\big|_{t=0}
    \\
    &= d\Gamma\big(\Gamma^{-1}(\rho)\big)\big[d\Gamma^{-1}(\rho)[X]\big]
    = \mrm{id}[X] = X.     
    \end{align}
    \end{subequations}
\end{proof}

\begin{proof}[Proof of Corollary~\ref{cor:Exp-1}]
    Setting
    \begin{equation}
    \mu = \Exp_{\rho}^{(e)}(X)
    \overset{\eqref{eq:def-Exp-rho-e}}{=} \Gamma\big(\Gamma^{-1}(\rho)+d\Gamma^{-1}(\rho)[X]\big)
    \end{equation}
    we solve for $X$,
    \begin{subequations}
    \begin{align}
    \Gamma^{-1}(\mu)
    &= \Gamma^{-1}(\rho)+d\Gamma^{-1}(\rho)[X]
    \\
    d\Gamma^{-1}(\rho)[X]
    &= \Gamma^{-1}(\mu) - \Gamma^{-1}(\rho)
    \\
    X &= d\Gamma\big(\Gamma^{-1}(\rho)\big)\big[\Gamma^{-1}(\mu) - \Gamma^{-1}(\rho)],
    \end{align}
    \end{subequations}
    which shows \eqref{eq:Exp-1}
     and where $d\Gamma(\Gamma^{-1}(\rho))^{-1} = d\Gamma^{-1}(\rho)$ was used to obtain the last equation.
\end{proof}

\begin{proof}[Proof of Lemma~\ref{lem:exp-rho}]
    We compute
    \begin{subequations}
    \begin{align}
    \Exp_{\rho}^{(e)}\circ \Rep_{\rho}[X]
    &\overset{\eqref{eq:def-Exp-rho-e}}{=} \Gamma\big(\Gamma^{-1}(\rho)+\Pi_{c,0}\circ\mb{T}_{\rho}\big[\Rep_{\rho}[X]\big]\big)
    \\
    &\overset{\eqref{eq:def-Rrho}}{=}
    \Gamma\big(\Gamma^{-1}(\rho)+\Pi_{c,0}\circ\mb{T}_{\rho}\big[\mb{T}_{\rho}^{-1}[X]-\la\rho,X\ra\rho\big]\big)
    \\
    &\overset{\rho=\mb{T}^{-1}[I]}{=}
    \Gamma\big(\Gamma^{-1}(\rho)+\Pi_{c,0}[X-\la\rho,X\ra I]\big)
    \\
    &= \Gamma\big(\Gamma^{-1}(\rho) + X\big)
    \end{align}
    \end{subequations}
    and omit the projection map $\Pi_{c,0}$ in the last equation, due to Lemma \ref{lem:repl-proj-rho-commute} or Lemma \ref{lem:Gamma-Pi0}.
\end{proof}

\begin{proof}[Proof of Lemma~\ref{lem:dexp-rho}]
    We compute
    \begin{subequations}
    \begin{align}
    d\exp_{\rho}(X)[Y]
    &\overset{\eqref{eq:def-exp-rho-b}}{=}
    d\Gamma\big(\Gamma^{-1}(\rho) + X)[Y]
    \overset{\eqref{eq:dGamma-H}}{=}
    \mb{T}_{\exp_{\rho}(X)}^{-1}[Y-\la \exp_{\rho}(X),Y\ra I]
    \\
    &\overset{\mb{T}_{\rho}^{-1}[I]=\rho}{=} 
    \mb{T}^{-1}_{\exp_{\rho}(X)}[Y]-\la\exp_{\rho}(X),Y\ra\exp_{\rho}(X)
    \\
    &\overset{\eqref{eq:def-Rrho}}{=}
    \Rep_{\exp_{\rho}(X)}[Y].   \qedhere
    \end{align}
    \end{subequations}
\end{proof}

\begin{proof}[Proof of Lemma~\ref{lem:rho-LD-diagonalized}]
    We compute
    \begin{subequations}
    \begin{align}
    L_{\rho}(D) &= \exp_{\rho}(-\Pi_{c,0}[D])
    \overset{\substack{\eqref{eq:Pic0-Rrho-commute} \\ \eqref{eq:def-exp-rho-a}}}{=}
    \exp_{\rho}(-D)
    \\
    &\overset{\eqref{eq:def-exp-rho-b}}{=}
    \Gamma\big(\Gamma^{-1}(\rho)-D\big)
    \overset{\eqref{eq:def-Gamma-1}}{=}
    \Gamma\big(\Pi_{c,0}[\log_{m}\rho]-D\big)
    \\
    &\overset{\eqref{eq:lem-rho-LD-decomposition}}{=} \Gamma\Big(Q(\log_{m}\Lambda_{\rho})Q^{\T} - \frac{1}{c}(\tr\Lambda_{\rho})I_{c} - Q\Lambda_{D}Q^{\T}\Big)
    \\
    &\overset{\eqref{eq:def-Gamma-a}}{=} 
    Q \frac{\Diag(e^{\log\lambda_{\rho}-\la\eins_{c},\lambda_{\rho}\ra\eins_{\mc{S}_{c}}-\lambda_{D}})}{\la\eins_{c},e^{\log\lambda_{\rho}-\la\eins_{c},\lambda_{\rho}\ra\eins_{\mc{S}_{c}}-\lambda_{D}}\ra} 
    Q^{\T}
    = Q \Diag\Big(\frac{e^{\log\lambda_{\rho}-\lambda_{D}}}{\la\eins_{c},e^{\log\lambda_{\rho}-\lambda_{D}}\ra}\Big)
    Q^{\T}
    \\
    &\overset{\eqref{eq:def_exp_standard_af}}{=} 
    Q \Diag\big(\exp_{\lambda_{\rho}}(-\lambda_{D})\big) Q^{\T}.
    \end{align}
    \end{subequations}
\end{proof}

\begin{proof}[Proof of Proposition~\ref{prop:single-vertex-matrix-flow}]
    Writing $\rho_{0}=\rho(0)=\eins_{\mc{D}_{c}}$ and $\Diag(\lambda_{D})=\Lambda_{D}$, 
    we have $$L_{\rho_{0}}(D)=Q \Diag\big(\exp_{\eins_{\mc{S}_{c}}}(-\lambda_{D})\big) Q^{\T}$$ by Lemma~\ref{lem:rho-LD-diagonalized} and
    \begin{subequations}
    \begin{align}
    \dot\rho(0)
    &\overset{\eqref{eq:single-vertex-matrix-flow-b}}{=} \mb{T}_{\rho_{0}}^{-1}[L_{\rho_{0}}(D)]-\la\rho_{0},L_{\rho_{0}}(D)\ra\rho_{0}
    \\
    &\overset{\eqref{eq:def-mcT-inverse}}{=}
    \int_{0}^{1}\Big(\frac{1}{c}I_{c}\Big)^{1-s} Q \Diag\big(\exp_{\eins_{\mc{S}_{c}}}(-\lambda_{D})\big) Q^{\T} \Big(\frac{1}{c}I_{c}\Big)^{s}ds
    \\ &\qquad\qquad
    - \frac{1}{c^{2}}\underbrace{\tr\Big(Q \Diag\big(\exp_{\eins_{\mc{S}_{c}}}(-\lambda_{D})\big) Q^{\T}\Big)}_{=1} I_{c}
    = \frac{1}{c} Q \Diag\big(\exp_{\eins_{\mc{S}_{c}}}(-\lambda_{D})\big) Q^{\T} - \frac{1}{c^{2}} I_{c}
    \\
    &= \frac{1}{c} Q\Diag\big(\exp_{\eins_{\mc{S}_{c}}}(-\lambda_{D}) - \eins_{\mc{S}_{c}}\big) Q^{\T}.
    \end{align}
    \end{subequations}
    Comparing this equation to the single vertex flow \eqref{eq:single-vertex-AF} at time $t=0$,
    \begin{subequations}\allowdisplaybreaks
    \begin{align}
    \dot p(0) &= \frac{1}{c}\big(\exp_{\eins_{\mc{S}_{c}}}(-D)
    -\frac{1}{c}\underbrace{\la\eins_{c},\exp_{\eins_{\mc{S}_{c}}}(-D)\ra}_{=1}\eins_{c}\big)
    = \frac{1}{c}\big(\exp_{\eins_{\mc{S}_{c}}}(-D)-\eins_{\mc{S}_{c}}\big)
    \end{align}
    \end{subequations}
    shows that
    \begin{equation}
    \dot\rho(0) = Q \Diag\big(\dot\lambda_{\rho}(0)\big) Q^{\T},
    \end{equation}
    that is
    \begin{subequations}
    \begin{align}
    \rho(t) &= Q \lambda_{\rho}(t) Q^{\T},
    \intertext{where $\lambda(t)$ solves the single vertex assignment flow equation \eqref{eq:single-vertex-AF} of the form}
    \dot\lambda_{\rho} &= \Rep_{\lambda_{\rho}}L_{\lambda_{\rho}}(\lambda_{D}).
    \end{align}
    \end{subequations}
    Corollary~\ref{prop:single-vertex-AF} completes the proof.
\end{proof}

\begin{proof}[Proof of Lemma~\ref{lem:Si-rho-explicit}]
    Put 
    \begin{equation}\label{eq:proof-def-Hi}
    H_{i}=\Gamma^{-1}(\rho_{i})
    \overset{\eqref{eq:def-Gamma-1}}{=}
    \Pi_{c,0}\logm \rho_{i},\quad i\in\mc{V}. 
    \end{equation}
    Then
    \begin{subequations}
    \begin{align}
    \big(\Exp^{(e)}_{\rho_{i}}\big)^{-1}\big(L_{\rho_{k}}(D_{k})\big)
    &\overset{\substack{\eqref{eq:def-L-rho-D} \\\eqref{eq:def-exp-rho-b}}}{=} 
    \big(\Exp^{(e)}_{\rho_{i}}\big)^{-1} \circ \Gamma\big(\Gamma^{-1}(\rho_{k})-D_{k}\big)
    \\
    &\overset{\eqref{eq:Exp-1}}{=} d\Gamma\big(\Gamma^{-1}(\rho_{i})\big)[\Gamma^{-1}(\rho_{k})-D_{k}-\Gamma^{-1}(\rho_{i})]
    \\
    &= d\Gamma(H_{i})[H_{k}-D_{k}-H_{i}].
    \end{align}
    \end{subequations}
    Substituting this expression into \eqref{eq:def-Si-rho} yields
    \begin{subequations}
    \begin{align}
    S_{i}(\rho)
    &\overset{\substack{\eqref{eq:def-Exp-rho-e}\\ \eqref{eq:omega-sum-1}}}{=} \Gamma\Big(H_{i}+\underbrace{d\Gamma^{-1}(\rho_{i})\circ d\Gamma(H_{i})}_{=I}\Big[\sum_{k\in\mc{N}_{i}}\w_{ik}(H_{k}-D_{k})-H_{i}\Big]\Big)
    \\
    &= \Gamma\Big(\sum_{k\in\mc{N}_{i}}\w_{ik}(H_{k}-D_{k})\Big).
    \end{align}
    \end{subequations}
    Substituting \eqref{eq:proof-def-Hi} and omitting the projection map $\Pi_{c,0}$ due to Lemma \ref{lem:Gamma-Pi0} yields \eqref{eq:def-Si-rho-simple}.
\end{proof}

\vspace{1cm}

\begin{proof}[Proof of Proposition~\ref{prop:Si-riemannian_center}]
        Substituting as in the proof of Lemma \ref{lem:Si-rho-explicit}, we get
        \begin{subequations}
        \begin{align}
        0 &= 
        d\Gamma\big(\Gamma^{-1}(\ol{\rho})\big)\Big[\sum_{k\in\mc{N}_{i}}\w_{ik}(\Pi_{c,0}\logm\rho_{k}-D_{k})-\Gamma^{-1}(\ol{\rho})\Big].
        \end{align}
        \end{subequations}
        Since $d\Gamma$ is one-to-one, the expression inside the brackets must vanish. Solving for $\ol{\rho}$ and omitting the projection map $\Pi_{c,0}$, due to Lemma \ref{lem:Gamma-Pi0}, gives \eqref{eq:def-Si-rho-simple}.
\end{proof}

\begin{proof}[Proof of Proposition~\ref{prop:reparametrization_QSAF}]
    Let $\rho(t)$ solve \eqref{eq:density-AF} and denote the argument of the replicator operator $\Rep_{\rho}$ on the right-hand side by
    \begin{equation}\label{eq:mu=S-rho}
    \mu(t) := S\big(\rho(t)\big),
    \end{equation}
    which yields \eqref{eq:S-flow-density-a} and \eqref{eq:density-AF}, respectively. It remains to show \eqref{eq:S-flow-density-b}. Differentiation yields
    \begin{subequations}\allowdisplaybreaks
    \begin{align}
    \dot \mu_{i} &= dS_{i}(\rho)[\dot\rho]
    \\
    &\overset{\substack{\eqref{eq:def-Si-rho-simple} \\\eqref{eq:def-mcT}}}{=}
    d\Gamma\Big(\sum_{k\in\mc{N}_{i}}\w_{ik}(\logm\rho_{k}-D_{k})\Big)\Big[
    \sum_{k\in\mc{N}_{i}}\w_{ik}\mb{T}_{\rho_{k}}[\dot\rho_{k}]\Big]
    \\
    &\overset{\eqref{eq:def-Si-rho-simple}}{=} 
    d\Gamma\Big(\Gamma^{-1}(S_{i}(\rho)\big)\Big)\Big[
    \sum_{k\in\mc{N}_{i}}\w_{ik}\mb{T}_{\rho_{k}}[\dot\rho_{k}]\Big]
    \\
    &\overset{\eqref{eq:dGamma-H}}{=}
    \mb{T}_{S_{i}(\rho)}^{-1}\Big[
    \sum_{k\in\mc{N}_{i}}\w_{ik}\mb{T}_{\rho_{k}}[\dot\rho_{k}] - \Big\la S_{i}(\rho),\sum_{k\in\mc{N}_{i}}\w_{ik}\mb{T}_{\rho_{k}}[\dot\rho_{k}]\Big\ra I\Big]
    \\
    &\overset{\mb{T}^{-1}_{\rho}[I]=\rho}{=}
    \mb{T}_{S_{i}(\rho)}^{-1}\Big[
    \sum_{k\in\mc{N}_{i}}\w_{ik}\mb{T}_{\rho_{k}}[\dot\rho_{k}]\Big]
    - \Big\la S_{i}(\rho),\sum_{k\in\mc{N}_{i}}\w_{ik}\mb{T}_{\rho_{k}}[\dot\rho_{k}]\Big\ra S_{i}(\rho)
    \\
    &\overset{\eqref{eq:def-Rrho}}{=}
    \Rep_{S_{i}(\rho)}\Big[
    \sum_{k\in\mc{N}_{i}}\w_{ik}\mb{T}_{\rho_{k}}[\dot\rho_{k}]\Big]
    \overset{\eqref{eq:mu=S-rho}}{=}
    \sum_{k\in\mc{N}_{i}}\w_{ik}\Rep_{\mu_{i}}\big[\mb{T}_{\rho_{k}}[\dot\rho_{k}]\big]
    \\
    &\overset{\eqref{eq:S-flow-density-a}}{=}
    \sum_{k\in\mc{N}_{i}}\w_{ik}\Rep_{\mu_{i}}\Big[\mb{T}_{\rho_{k}}\big[\Rep_{\rho_{k}}[\mu_{k}]\big]\Big]
    \\
    &\overset{\eqref{eq:def-Rrho}}{=}
    \sum_{k\in\mc{N}_{i}}\w_{ik}\Rep_{\mu_{i}}\Big[\mb{T}_{\rho_{k}}\big[\mb{T}_{\rho_{k}}^{-1}[\mu_{k}]-\la\rho_{k},\mu_{k}\ra T_{\rho_{k}}^{-1}[I]\big]\Big]
    \\
    &= \sum_{k\in\mc{N}_{i}}\w_{ik}\Rep_{\mu_{i}}[\mu_{k}-\la\rho_{k},\mu_{k}\ra I]
    \overset{\eqref{eq:Pic0-Rrho-commute}}{=}
    \sum_{k\in\mc{N}_{i}}\w_{ik}\Rep_{\mu_{i}}[\mu_{k}]
    \\
    &\overset{\eqref{eq:def-Omega-rho}}{=}
    \Rep_{\mu_{i}}\big[\Omega[\mu]_{i}\big].
    \end{align}
    \end{subequations}
    The initial condition for $\rho$ is given by \eqref{eq:density-AF}. The initial condition for $\mu$ follows from \eqref{eq:mu=S-rho}.
\end{proof}

\begin{proof}[Proof of Proposition~\ref{prop:potentail-QSAF}]
    We compute using \eqref{eq:ass-symmetry}
    \begin{subequations}
    \begin{align}
    J(\mu)
    &= -\frac{1}{2}\sum_{j\in\mc{V}}\Big\la\mu_{j},\sum_{k\in\mc{N}_{j}}\w_{jk}\mu_{k}\Big\ra
    \\
    \partial_{\mu_{i}}J(\mu)
    &= -\frac{1}{2}\Big(\sum_{k\in\mc{N}_{i}}\w_{ik}\mu_{k} + \sum_{j\in\mc{N}_{i}}\w_{ji}\mu_{j}\Big)
    \overset{\eqref{eq:ass-symmetry-om}}{=} 
    -\Omega[\mu]_{i}.
    \end{align}
    \end{subequations}
    Consequently, by Proposition \ref{prop:grad-Dc} and \eqref{eq:def-Rrho},
    \begin{equation}
    -\ggrad_{\mu_{i}}J(\mu)
    = -\Rep_{\mu_{i}}[\partial_{\mu_{i}}J(\mu)]
    = \Rep_{\mu_{i}}\big[\Omega[\mu]_{i}\big]
    \end{equation}
    which shows that the right-hand sides of \eqref{eq:dot-mu-grad} and \eqref{eq:S-flow-density-b} are equal.
\end{proof}

\begin{proof}[Proof of Proposition~\ref{prop:potentail-QSAF-laplacian}]
    Starting at \eqref{eq:dot-mu-potential} we compute
    \begin{equation}
    J(\mu) = -\frac{1}{2}\la\mu,\Omega[\mu]\ra
    = \frac{1}{2}\la\mu,(\mrm{id}-\Omega-\mrm{id})[\mu]\ra
    = \frac{1}{2}\big(\la\mu,L_{\mc{G}}[\mu]\ra - \|\mu\|^{2}\big).
    \end{equation}
    It remains to rewrite the quadratic form involving $L_{\mc{G}}$.
    \begin{subequations}
    \begin{align}
    \la\mu,L_{\mc{G}}[\mu]\ra
    &= \sum_{i\in\mc{V}}\la\mu_{i},\mu_{i}-\Omega[\mu]_{i}\ra
    = \sum_{i\in\mc{V}}\Big(\|\mu_{i}\|^{2}-\Big\la\mu_{i},\sum_{k\in\mc{N}_{i}}\w_{ik}\mu_{k}\Big\ra\Big)
    \\
    &= \sum_{i\in\mc{V}}\Big(\underbrace{\sum_{k\in\mc{N}_{i}}\w_{ik}}_{=1}\|\mu_{i}\|^{2} - \sum_{k\in\mc{N}_{i}}\w_{ik}\la\mu_{i},\mu_{k}\ra\Big)
    = \sum_{i\in\mc{V}}\sum_{k\in\mc{N}_{i}}\w_{ik}\big(\la\mu_{i},\mu_{i}-\mu_{k}\ra\big)
    \\
    &= \frac{1}{2}\Big(
    \sum_{i\in\mc{V}} \sum_{k\in\mc{N}_{i}}\w_{ik}\big(\la\mu_{i},\mu_{i}-\mu_{k}\ra\big) + 
    \sum_{k\in\mc{V}} \sum_{i\in\mc{N}_{k}}\w_{ki}\big(\la\mu_{k},\mu_{k}-\mu_{i}\ra\big)
    \Big),
    \end{align}
    \end{subequations}
    where the last sum results from the first one by interchanging the indices $i$ and $k$. Using the symmetry relations \eqref{eq:ass-symmetry}, we rewrite the second sum
    \begin{subequations}
    \begin{align}
    \la\mu,L_{\mc{G}}[\mu]\ra
    &= \frac{1}{2}\Big(
    \sum_{i\in\mc{V}} \sum_{k\in\mc{N}_{i}}\w_{ik}\big(\la\mu_{i},\mu_{i}-\mu_{k}\ra\big) + 
    \sum_{i\in\mc{V}} \sum_{k\in\mc{N}_{i}}\w_{ik}\big(\la\mu_{k},\mu_{k}-\mu_{i}\ra\big)
    \Big)
    \\
    &= \frac{1}{2}\sum_{i\in\mc{V}} \sum_{k\in\mc{N}_{i}}\w_{ik}\big(\la\mu_{i},\mu_{i}-\mu_{k}\ra + \la\mu_{k},\mu_{k}-\mu_{i}\ra\big)
    \\
    &= \frac{1}{2}\sum_{i\in\mc{V}} \sum_{k\in\mc{N}_{i}}\w_{ik}\|\mu_{i}-\mu_{k}\|^{2}
    \end{align}
    \end{subequations}
    which completes the proof.
\end{proof}

\vspace{1cm}

\begin{proof}[Proof of Lemma~\ref{lem:prop_of_d_pi}] We start with claim (b).
\begin{enumerate}[(a)]
\item[(b)]
    Let Let $\mu \in \mathcal D_\Pi$ and $X \in T_\mu \mathcal D_\Pi$. Suppose the vector $X$ is represented by a curve $\eta :(-\varepsilon,\varepsilon) \to \mathcal D_\Pi$, such that $\eta(0) = \mu$ and $\eta'(0) = X$. In view of the definition \eqref{eq:def-mcD-Pi} of $\mc{D}_{\Pi}$, we thus have 
    \begin{equation}
    \label{eq:form_of_X}
         \eta(t) = \sum_{i\in [l]} \frac{p_i(t)}{\mathrm{tr} \pi_i} \pi_i \quad \implies \quad X = \sum_{i\in [l]} \frac{p_i'(0)}{\mathrm{tr} \pi_i} \pi_i.
    \end{equation} 
    Consequently, if $\mathcal U = \qty{u_1,...,u_c}$ is a basis of $\mathbb{C}^{c}$ that diagonalizes $\mu$, then the tangent vector $X$ is also diagonal in this basis $\mathcal U$ and $X$ commutes with $\mu$, i.e. $[\mu,X]=0$ and $X \in T_\mu^{c}\mathcal D_c$. This proves (b). 
\item[(a)]
    The bijection $\mathcal D_\Pi \to \mathcal S_{l}$ is explicitly given by
    \begin{equation}
        \Phi_\Pi\colon\mathcal D_\Pi \to \mathcal S_{l}, \qquad \sum_{i \in [l]} \frac{p_i}{\mathrm{tr}\pi_i} \pi_i \mapsto (p_1,...,p_l).
    \end{equation} 
    This is bijective by the definiton of $\mathcal D_\Pi$. It remains to be shown that it is an isometry. Consider another tangent vector $Y\in T_\mu \mathcal D_\Pi$. We know that $\mu,X,Y$ can all be diagonalized in a common eigenbasis. Denote this basis again by $\mathcal U$. Then we can write
    \begin{equation}
        \mu = \sum_{i \in [c]} \tilde p_i u_iu_i^{*}, \qquad X = \sum_{i \in [c]} \tilde x_i u_i u_i^{*}, \qquad
        Y = \sum_{i \in [c]} \tilde y_i u_i u_i^{*}       
    \end{equation}   
    and compute
    \begin{subequations}
    \begin{align}
        \iota^{*}g_{\sst{\mrm{BKM}},\mu}(X,Y) 
        &= 
        \int_0^{\infty} \mathrm{tr} \qty(X(\mu + \lambda I)^{-1}Y(\mu + \lambda I)^{-1}) d\lambda \\
        &= 
        \sum_{i\in [c]} \int_0^{\infty} \mathrm{tr} \bigg( 
            \frac{\tilde x_i \tilde y_i}{(\tilde p_i + \lambda)^{2}} u_i u_i^{\ast} 
        \bigg)d \lambda \\
        & = \sum_{i \in [c]} \frac{\tilde x_i \tilde y_i}{\tilde p_i}.
    \end{align}
    \end{subequations}
    Note that the vector $\tilde p = (\tilde p_1,...,\tilde p_c)$ comes from $\mu \in \mathcal D_\Pi$. Therefore, the value $p_j/\mathrm{tr} \pi_j$ must occur $\mathrm{tr} \pi_j$ times in $\tilde p$, for every $j \in [l]$. This observation holds for the vectors $\tilde x = (\tilde x_1,...,\tilde x_c)$ and $\tilde y = (\tilde y_1,...,\tilde y_c)$ as well. Thus, the sum above can be reduced to
    \begin{equation}
        \sum_{i \in [c]} \frac{\tilde x_i \tilde y_i}{\tilde p_i} = \sum_{j \in [l]} \frac{x_j y_j}{p_j}, 
    \end{equation} 
    where $(p_1,...,p_j) = \Phi(\mu)$, $(x_1,...,x_l) = d \Phi[X]$ and $(y_1,...,y_l) = d \Phi[Y]$.
    Taking into account that $(x_1,...,x_l)$ and $(y_1,...,y_l)$ are the images of $X,Y$ under the differential $d \Phi$, we conclude 
    \begin{equation}
        \iota^{*}g_{\sst{\mrm{BKM}},\mu}(X,Y) =   \sum_{i \in [l]} \frac{x_i y_i}{p_i}
        \overset{\eqref{eq:def-gp}}{=} 
        g_{\sst{\mrm{FR}},\Phi(\mu)}(d \Phi(X),d \Phi(Y)). 
    \end{equation}
    This proves part (a).
    \item[(c)]
    Part (c) is about the commutativity of the diagram
    \begin{equation}
    \label{eq:diag_pr_419}
        \begin{tikzcd}
            \mathcal{D}_{\Pi} \ar[r,hook,"\alpha_\Pi"] \ar[d,"\Phi_\Pi"] & \mathcal{D}_{\Pi_{\mathcal U}} \ar[d,"\Phi_{\Pi_{\mathcal{U}}}"] \\
            \mathcal S_{l} \ar[r,hook,"\beta_\Pi"] & \mathcal S_{c}.
        \end{tikzcd}
    \end{equation}
    The horizontal arrows can be described as follows. Recall that $\Pi = \qty{\pi_1,...,\pi_l}$. Denote by $k_i = \mathrm{tr} \pi_i$ the dimension of the images of the projectors $\pi_i$. For a fixed $p=(p_1,...,p_l)\in \mathcal S_{l}$, set
    \begin{equation}
        P=(P_1,...,P_c) := (\underbrace{p_1/k_1,...,p_1/k_1}_{k_1~\mathrm{times}},...,\underbrace{p_l/k_l,...,p_l/k_l}_{k_l~\mathrm{times}}) \in \mathcal S_{c}.        
    \end{equation}
    Then $\alpha_\Pi$ is given by
    \begin{equation}
        \alpha_\Pi \bigg( \sum_{i \in [l]} \frac{p_i}{k_i}\pi_i\bigg) = \sum_{j\in [c]} P_j u_j u_j^{\ast} \in \mathcal D_{\Pi_{\mathcal{U}}} \qq{and}
        \beta_\Pi(p_1,...,p_l) = (P_1,...,P_c) 
    \end{equation}  
    The diagram \eqref{eq:diag_pr_419} commutes by definition of the $\Phi$ maps.        
\end{enumerate}
\end{proof}

%
\begin{proof}[Proof of Lemma~\ref{prop:commutativity-preservation}]
\begin{enumerate}[(i)]
    \item Due to the commutativity of the components $\mu_i$ of $\mu \in \mathcal Q$, we can simplify the expression for the vector field of the QSAF as follows. 
    \begin{subequations}
    \begin{align}
    \mathfrak R_{\mu}[\Omega[\mu]]_{i} &\overset{\eqref{eq:def-R-rho}}{=}\mathfrak R_{\mu_{i}}\big[\Omega[\mu]_{i}\big]\\
    &\overset{\substack{\eqref{eq:def-Rrho} \\ \eqref{eq:def-mcT-inverse}}}{=}\sum_{k\in\mc{N}_{i}}\w_{ik}\Big(\int_{0}^{1}\mu_{i}^{1-\lambda}\mu_{k}\mu_{i}^{\lambda}d\lambda
    - \tr(\mu_{i}\mu_{k})\mu_{i}\Big)\\
    &= \sum_{k\in\mc{N}_{i}}\w_{ik}\big(\mu_{i}\mu_{k} - \tr(\mu_{i}\mu_{k})\mu_{i}\big).
    \end{align}    
    \end{subequations}
    Invoke that $\mu\in \mathcal D_{\Pi,c}$, such that all the components $\mu_i$ can be written as
    \begin{equation}
        \mu_i = \sum_{r \in [l]} \frac{p^{i}_r}{\mathrm{tr} \pi_r} \pi_r, \qquad p^{i} = (p^i_1,...,p^i_l) \in \mathcal S_l,\quad i \in \mathcal V. 
    \end{equation}
   Then we can further simplify 
    \begin{equation}
        \mu_{i}\mu_{k} = \sum_{r \in [l]} \frac{p^{i}_r p^k_r}{ (\mathrm{tr} \pi_r)^2} \pi_r \qq{and}
        \mathrm{tr}(\mu_{i}\mu_{k})
        =
        \sum_{r \in [l]} \frac{p^{i}_r p^k_r}{\mathrm{tr} \pi_r}
    \end{equation}
    and consequently
    \begin{subequations}
    \begin{align}
        \sum_{k\in\mc{N}_{i}}\w_{ik}\Big(\mu_{i}\mu_{k} - \tr(\mu_{i}\mu_{k})\mu_{i}\Big)
        &=
        \sum_{r \in [l]}
        \sum_{k\in\mc{N}_{i}}\w_{ik}
        \bigg(
        \frac{p^k_r}{\mathrm{tr} \pi_r}  
        -
        \bigg(\sum_{s \in [l]} \frac{p^{i}_s p^k_s}{\mathrm{tr} \pi_s}\bigg)
        \bigg)
        \frac{p^i_r}{\mathrm{tr}\pi_r}     
        \pi_r \\
        &=
        \sum_{r \in [l]}
        \frac{x_r}{\mathrm{tr}\pi_r}     
        \pi_r,
    \end{align}
    \end{subequations}
    where
    \begin{equation}
        x_r : = 
        \sum_{k\in\mc{N}_{i}}\w_{ik}
        \bigg(
        \frac{p^k_r}{\mathrm{tr} \pi_r}  
        -
        \bigg(\sum_{s \in [l]} \frac{p^{i}_s p^k_s}{\mathrm{tr} \pi_s}\bigg)
        \bigg)
        p^i_r.
    \end{equation}
    Thus, $$\mathfrak R_{\mu}[\Omega[\mu]]_{i} = \sum_{r \in [l]} x_r/(\mathrm{tr}\pi_r) \pi_r.$$ This has to be compared with the general form of a tangent vector $X \in T_{\mu_i} \mathcal D_{\Pi}$, given by \eqref{eq:form_of_X}. The only condition the vector $p'(0)$ in \eqref{eq:form_of_X} has to satisfy, is that its components sum to $0$. This holds for $x=(x_1,...x_l)$ as well. We conclude that $\mathfrak R_{\mu}[\Omega[\mu]]_{i}$ lies in $T_{\mu_i}\mathcal D_\Pi$ for all $i \in \mathcal V$, or equivalently, $\mathfrak R_{\mu}[\Omega[\mu]] \in T_\mu \mathcal D_{\Pi,c}$.    
    \item Write $\mu_i = U \Diag(S_i) U^*$ for all $i \in \mathcal V$ with $S_i \in \mathcal S_c$, 
    and express $\mathfrak R_\mu\qty[\Omega[\mu]]$ in terms of $S \in \mathcal W$ as 
    \begin{subequations}
    \begin{align}
        \mathfrak R_{\mu}\qty[\Omega [\mu]]_{i} &= \sum_{k\in\mc{N}_{i}}\w_{ik}\big(\mu_{i}\mu_{k}- \tr(\mu_{i}\mu_{k})\mu_{i}\big) \\
        &= U \Diag\bigg(\sum_{k\in\mc{N}_{i}}\w_{ik}\big(S_i \cdot S_k - \langle S_i,S_k \rangle S_i  \big)\bigg) U^* \\
        &= U \Diag\qty(R_S[\Omega S])_i U^*.
    \end{align}
    \end{subequations}
\end{enumerate}
\end{proof}

\begin{proof}[Proof of Corollary~\ref{cor:af_leq_qaf}]
    Write $D_i = U \Diag\qty(\lambda_i) U^*$ for $\lambda_i \in \mathbb{R}^n$, diagonalized in the $U$-basis. Then the initial condition for the QSAF S-flow \eqref{eq:S-flow-density-b} is given by 
    \begin{equation}
        \mu(0)_i =  S(\eins_{\mathcal Q})_i 
        \overset{\eqref{eq:def-Si-rho-simple}}{=} \Gamma\bigg(\sum_{k \in \mathcal N_i} \omega_{ik}(-D_k)\bigg).
    \end{equation} 	
    Then set $\tilde D_i := \sum_{k \in \mathcal N_i} \omega_{ik}D_k = U \Diag(\tilde \lambda_i) U^*$, where 
    \begin{equation}
        \tilde \lambda_i = \sum_{k \in \mathcal N_i} \omega_{ik} \lambda_k \in \mathbb{R}^c.
    \end{equation}
    Recall further that $\Gamma$ is computed in terms of the matrix exponential as specified by \eqref{eq:def-Gamma}. Thus,
    \begin{equation}
        \mu(0)_i = \Gamma(- \tilde D_i) = \frac{\mathrm{exp}_{\mathrm{m}}(-\tilde D_i)}{\mathrm{tr} \exp_{\mathrm{m}}(-\tilde D_i)}= \frac{U \exp_{\mathrm{m}}(-\Diag(\tilde \lambda_i))U^* }{\mathrm{tr} (U \exp_{\mathrm{m}} (-\Diag(\tilde \lambda_i))U^*)} 
        = U \frac{\Diag(\exp(-\tilde \lambda_i))}{\mathrm{tr} \exp_{\mathrm{m}} (-\Diag(\tilde \lambda_i))} U^*.
    \end{equation}
    This shows that all the $\mu(0)_i's$ are diagonalized by the same basis $\mathcal U$ and $\mu(0) \in \mathcal D_{\Pi_{\mathcal U},c}$ and we can apply Proposition \ref{prop:commutativity-preservation} (ii). Therefore, the vector field of the quantum state assignment S-flow is also diagonalized in the basis $\mathcal U$ and we solve simply for the diagonal components. The quantum S-flow equation can be written as 
    \begin{equation}
        \dot \mu_i = U \Diag(R_{S_i} [\Omega S]) U^* ,  \qquad  \mu(0)_i = U \Diag(S(\eins_\mathcal W)) U^*   
    \end{equation}
    with the classical similarity map $S$ defined in terms of the data vectors $\lambda_i$ and $\mu_i$ related to $S_i \in \mathcal S_c$ by $\mu_i = U \mathrm{diag}(S_i) U^*$. The solution to this system is 
    \begin{equation}
        \mu_i(t) = U \Diag(S_i(t)) U^*,   
    \end{equation} 
    where $S \in \mathcal W$ solves the classical S-flow equation $\dot S = R_S [\Omega S]$ and $S(0) =S (\eins_{\mathcal W})$
\end{proof}


\newpage
\bibliographystyle{amsalpha}
\bibliography{QSAF2023}

\clearpage
\end{document}